\documentclass[11pt]{amsart}
\usepackage{amsmath,amssymb,amsthm}
\usepackage{enumerate}
\usepackage{epsfig}
\usepackage{graphics}
\usepackage{tikz}

\newtheorem{theorem}{Theorem}[section]
\newtheorem{lemma}[theorem]{Lemma}
\newtheorem{proposition}[theorem]{Proposition}
\newtheorem{corollary}[theorem]{Corollary}

\theoremstyle{definition}

\newtheorem{example}[theorem]{Example}

\theoremstyle{remark}
\newtheorem{remark}[theorem]{Remark}

\numberwithin{equation}{section}

\newcommand{\EE}{\mathbb{E}}
\newcommand{\NN}{\mathbb{N}}
\newcommand{\QQ}{\mathbb{Q}}
\newcommand{\RR}{\mathbb{R}}

\newcommand{\f}{\infty}
\newcommand{\eps}{\varepsilon}
\newcommand{\graph}{\mathrm{Graph}}
\newcommand{\osc}{\mathrm{osc}}
\newcommand{\var}{\mathrm{Var}}
\newcommand{\sgn}{\mathrm{sgn}}

\title[The partial derivatives of Okamoto's function]{The higher order partial derivatives of Okamoto's function with respect to the parameter}

\author[P. Allaart]{Pieter Allaart}
\address[P. Allaart]{Mathematics Department, University of North Texas, 1155 Union Cir \#311430, Denton, TX 76203-5017, U.S.A.}
\email{allaart@unt.edu}

\author[N. Dalaklis]{Nathan Dalaklis}
\address[N. Dalaklis]{Mathematics Department, University of North Texas, 1155 Union Cir \#311430, Denton, TX 76203-5017, U.S.A.}
\email{NathanDalaklis@my.unt.edu}

\author[K. Kawamura]{Kiko Kawamura}
\address[K. Kawamura]{Mathematics Department, University of North Texas, 1155 Union Cir \#311430, Denton, TX 76203-5017, U.S.A.}
\email{kiko@unt.edu}

\author[M. Ortiz]{Matthew Ortiz}
\address[M. Ortiz]{Mathematics Department, University of North Texas, 1155 Union Cir \#311430, Denton, TX 76203-5017, U.S.A.}
\email{MatthewOrtiz2@my.unt.edu}

\author[J. Zheng]{Jiajie Zheng}
\address[J. Zheng]{Mathematics Department, University of North Texas, 1155 Union Cir \#311430, Denton, TX 76203-5017, U.S.A.}
\email{Jiajie.Zheng@unt.edu}

\begin{document}

\begin{abstract}
Let $\{F_a: a\in(0,1)\}$ be Okamoto's family of continuous self-affine functions, introduced in [{\em Proc. Japan Acad. Ser. A Math. Sci.} {\bf 81} (2005), no. 3, 47--50]. This family includes well-known ``pathological" examples such as Cantor's devil's staircase and Perkins' continuous but nowhere differentiable function. It is well known that $F_a(x)$ is real analytic in $a$ for every $x\in[0,1]$. We introduce the functions
\[
M_{k,a}(x):=\frac{\partial^k}{\partial a^k}F_a(x), \qquad k\in\NN, \quad x\in[0,1].
\]
We compute the box-counting dimension of the graph of $M_{k,a}$, characterize its differentiability, and investigate in detail the set of points where $M_{k,a}$ has an infinite derivative. While some of our results are similar to the known facts about Okamoto's function, there are also some notable differences and surprising new phenomena that arise when considering the higher order partial derivatives of $F_a$.
\end{abstract}

\subjclass[2020]{Primary: 26A27, 26A30; Secondary: 28A78, 11A63}

\keywords{Continuous nowhere differentiable function, infinite derivative, ternary expansion, unique beta-expansion, Hausdorff dimension}

\maketitle

\tableofcontents

\section{Introduction}

In 2005, Okamoto \cite{Okamoto} introduced a one-parameter family of self-affine functions $\{F_a: 0<a<1\}$ defined on the interval $[0,1]$ as the unique continuous solution of the functional equation
\begin{equation} \label{eq:Okamoto-FE}
F_a(x)=\begin{cases}
aF_a(3x) & \mbox{if $0\leq x\leq 1/3$},\\
(1-2a)F_a(3x-1)+a & \mbox{if $1/3<x\leq 2/3$},\\
aF_a(3x-2)+1-a & \mbox{if $2/3<x\leq 1$}.
\end{cases}
\end{equation}
Special cases of this family had been considered before: Most famously, $F_{1/2}$ is the ternary Cantor function. Furthermore, $F_{5/6}$ is Perkins' continuous but nowhere differentiable function \cite{Perkins}, and $F_{2/3}$ was studied independently by Bourbaki \cite{Bourbaki} and Katsuura \cite{Katsuura}. However, Okamoto \cite{Okamoto} was the first to systematically study the differentiability of $F_a$ for all $a\in(0,1)$. He proved the following surprising result:

\begin{theorem}[Okamoto] \label{thm:Okamoto}
Let $a_0\approx .5592$ be the unique real root of $54a^3-27a^2=1$. 
\begin{enumerate}[(a)]
\item $F_a$ is nowhere differentiable if $2/3\leq a<1$;
\item $F_a$ is nondifferentiable almost everywhere but differentiable at uncountably many points if $a_0<a<2/3$;
\item $F_a$ is differentiable almost everywhere but nondifferentiable at uncountably many points if $0<a<a_0$ and $a\neq1/3$.
\end{enumerate}
\end{theorem}

Note the exclusion of $a=1/3$ in part (c). The reason is, that $F_{1/3}(x)=x$. Thus, $F_{1/3}$ is the only member of the family $\{F_a\}$ that is differentiable everywhere.
Kobayashi \cite{Kobayashi} later showed that the conslusion of case ({\em b}) also applies for $a=a_0$. Graphs of $F_a$ for several values of $a$ are shown in Figure \ref{fig:Okamoto-graphs}. 

\begin{figure}
\begin{center}
\epsfig{file=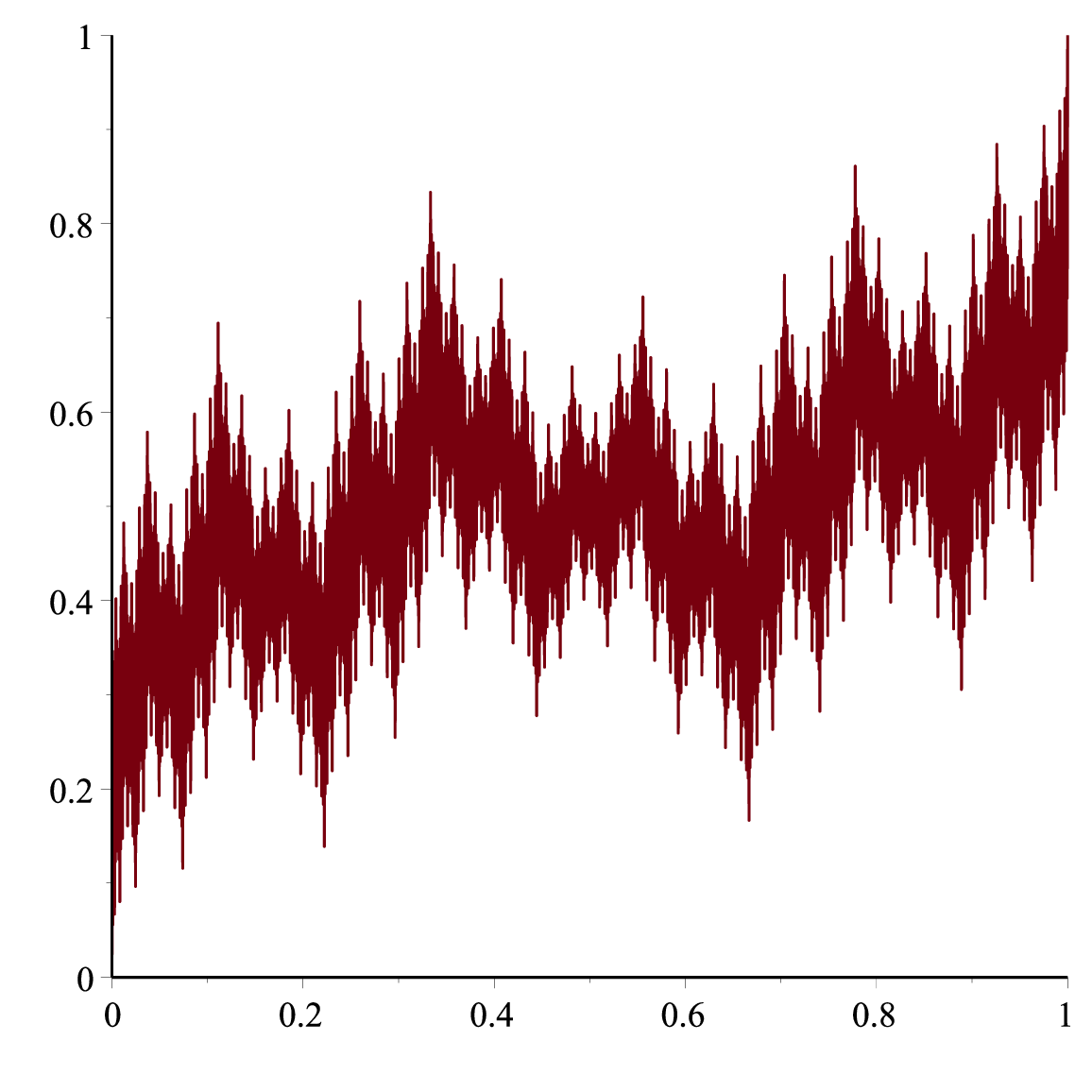, height=.3\textwidth, width=.3\textwidth} \qquad\quad
\epsfig{file=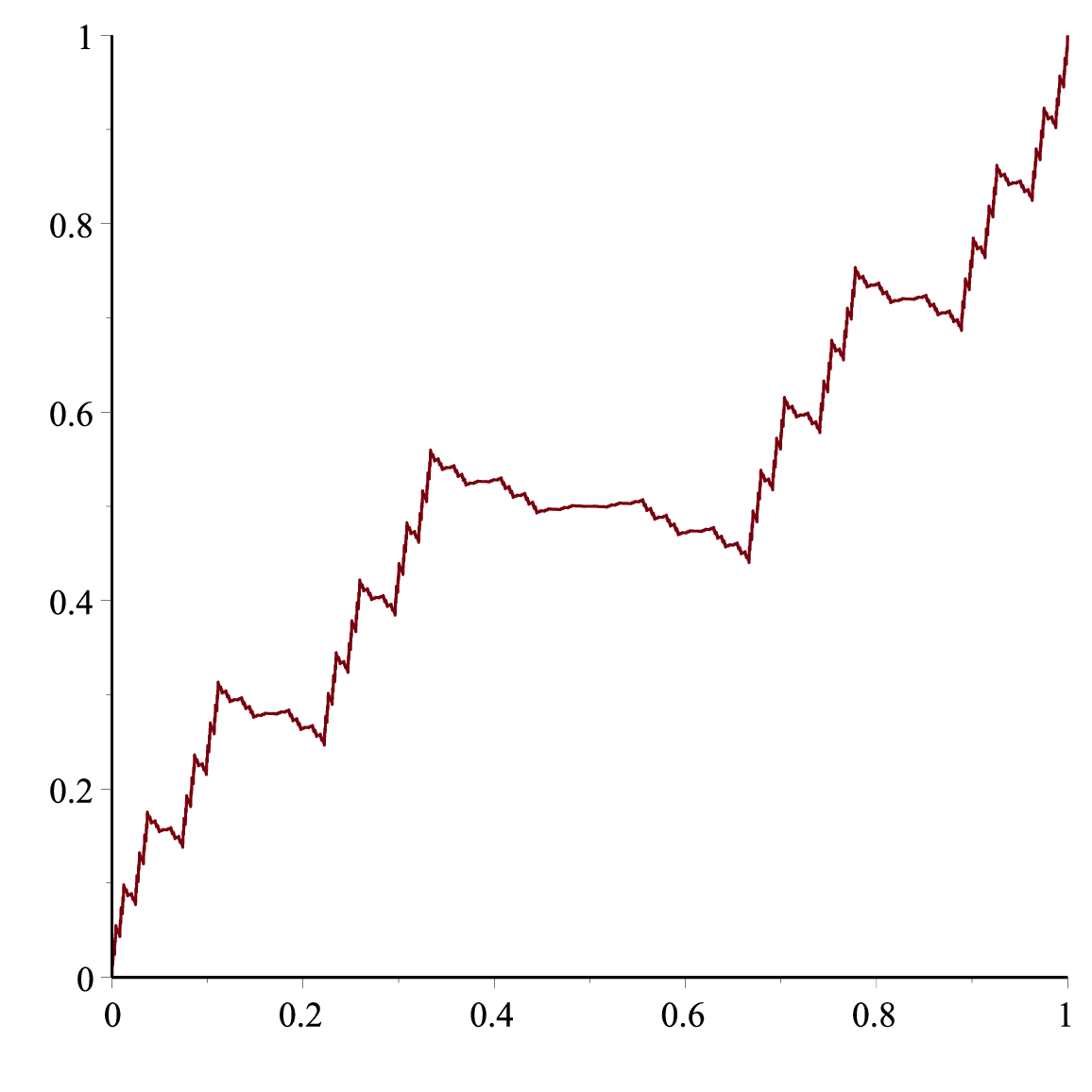, height=.3\textwidth, width=.3\textwidth}\\
\epsfig{file=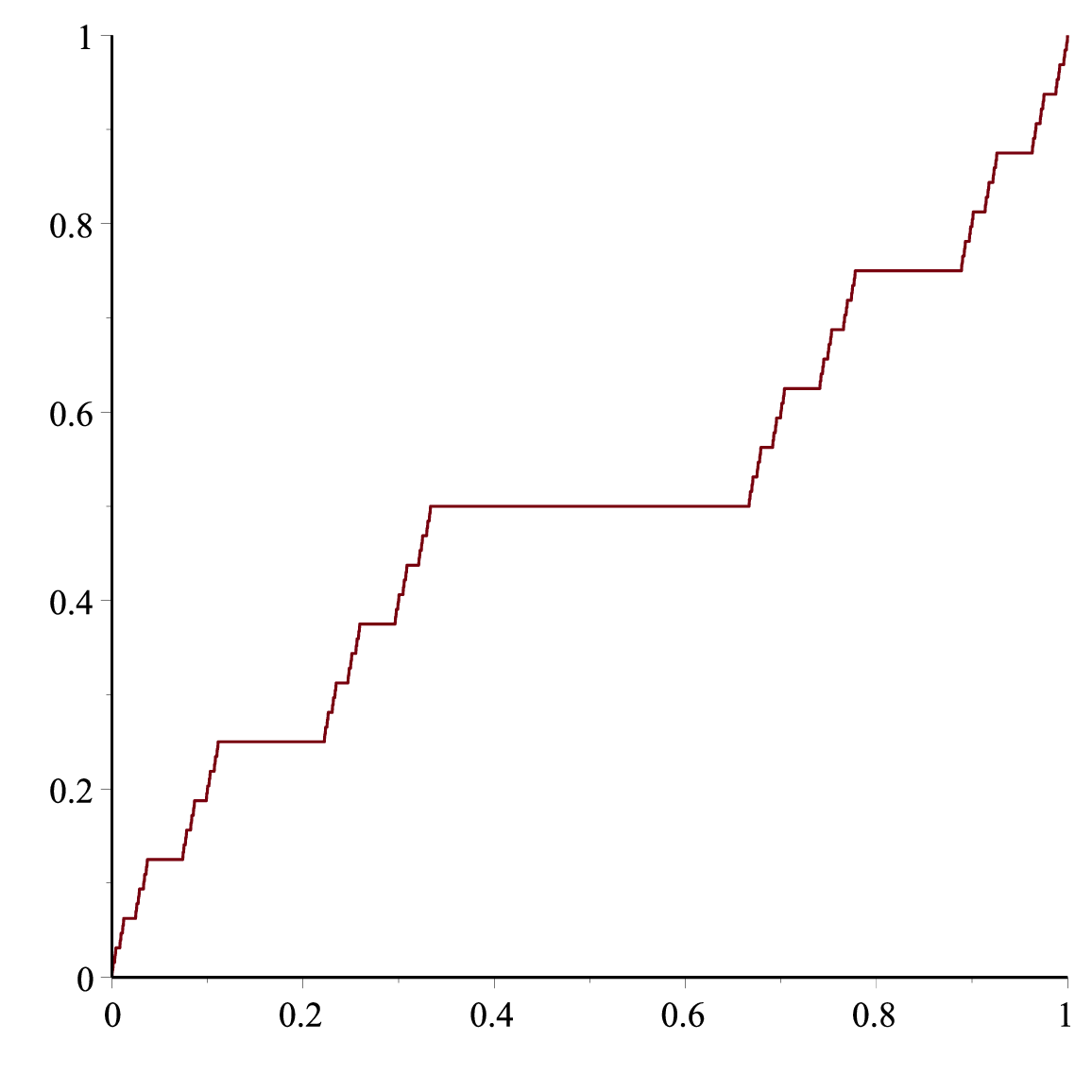, height=.3\textwidth, width=.3\textwidth} \qquad\quad
\epsfig{file=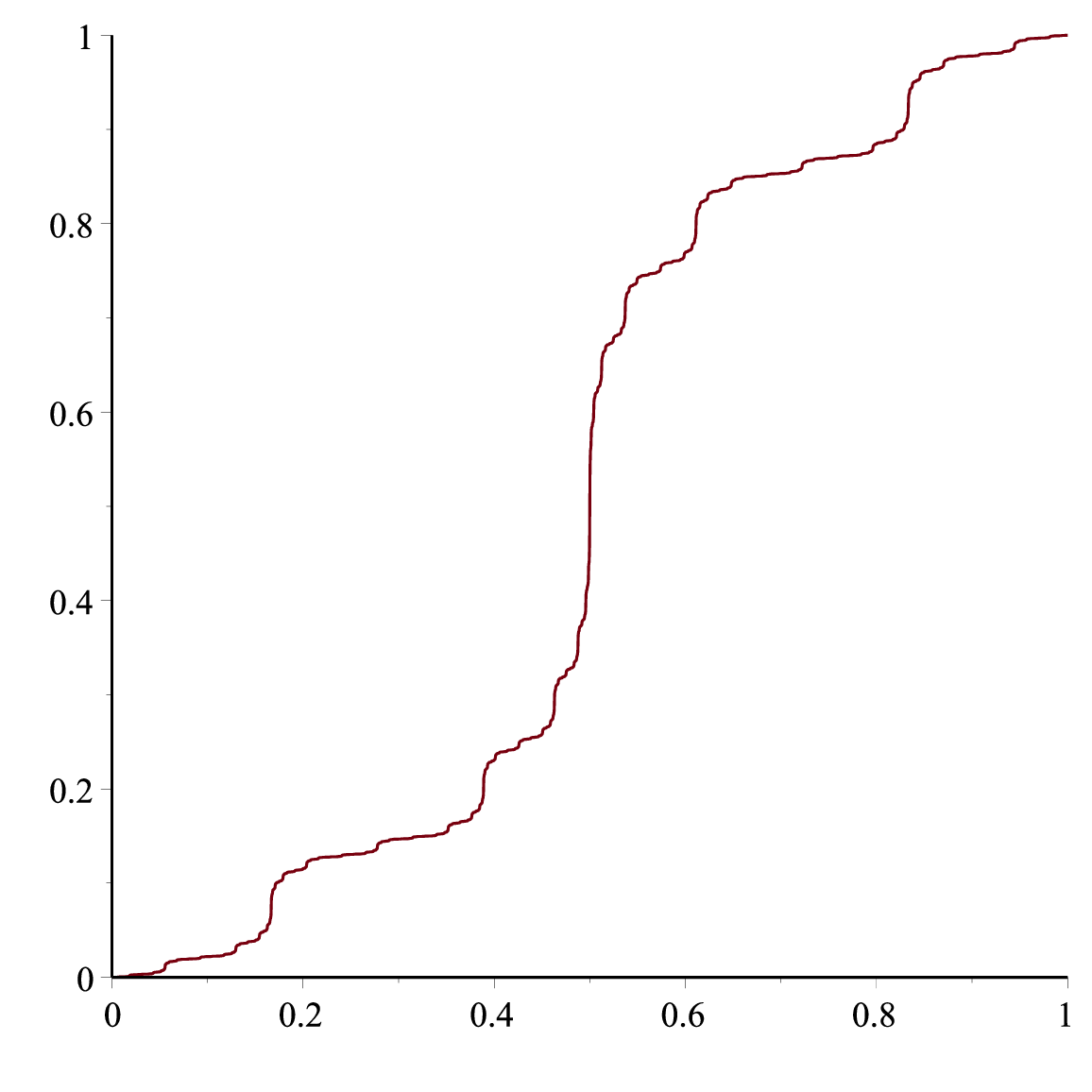, height=.3\textwidth, width=.3\textwidth}
\caption{Graph of $F_a$ for several values of $a$. Top left: $a=5/6$ (Perkins' function); top right: $a=a_0\approx .5592$; bottom left: $a=1/2$ (the Cantor function); bottom right: $a=1/5$ (a ``slippery devil's staircase").}
\label{fig:Okamoto-graphs}
\end{center}
\end{figure}

Various authors have subsequently studied Okamoto's function. For instance, Allaart \cite{Allaart} computed the Hausdorff dimension of the sets of exceptional points of measure zero in ({\em b}) and ({\em c}) of the above theorem, and described the set of points where $F_a$ has an infinite derivative. McCollum \cite{McCollum} used the self-affine structure of $F_a$ to show that its graph has box-counting dimension $1+\log_3(4a-1)$ for $a\geq 1/2$, and claimed (with an incorrect proof) that this is also the Hausdorff dimension. Only very recently have B\'ar\'any and Prokaj \cite{Barany-Prokaj} shown that for Lebesgue-almost all $a\in(1/2,1)$ the Hausdorff dimension of the graph of $F_a$ is indeed equal to $1+\log_3(4a-1)$. (For $a\leq 1/2$ the function $F_a$ is increasing, and so the dimension of its graph is trivially $1$.) Seuret \cite{Seuret} was the first to determine the multifractal structure of $F_a$.

In this article, we are interested in the partial derivatives of $F_a$ with respect to the parameter $a$. That is, we define for all $k\in\NN_0:=\NN\cup\{0\}$ and $0<a<1$ the function
\[
M_{k,a}(x):=\frac{\partial^k}{\partial a^k}F_a(x), \qquad k\in\NN, \quad x\in[0,1].
\]
(In particular, $M_{0,a}=F_a$.) These functions are well defined, because $F_a(x)$ is real analytic in $a$ for every fixed $x$, as shown in \cite{Kobayashi}. Figure \ref{fig:M-graphs} shows some of their graphs.

We observe that the function $M_{1,1/3}$ (first graph in second row of Figure \ref{fig:M-graphs}) was investigated by Dalaklis et al.~\cite{Dalaklis}, who showed that it is continuous but nowhere differentiable and characterized the set of points where it has an infinite derivative. (By ``nowhere differentiable" in this article, we shall mean ``not possessing a finite derivative at any point".) The nowhere-differentiability of $M_{1,1/3}$ is perhaps surprising because $F_{1/3}$ is simply the identity function on $[0,1]$, and is therefore differentiable everywhere. 

Motivated by the article \cite{Dalaklis} and the interesting structures of the graphs in Figure \ref{fig:M-graphs}, we study here the partial derivative $\partial{F_a}(x)/{\partial a}$ for other values of $a$, as well as the higher order partial derivatives. These functions are no longer strictly self-affine, but are still approximately self-affine in some sense. We will compute the box-counting dimension of the graph of $M_{k,a}$, characterize its differentiability, and investigate in detail the set of points where $M_{k,a}$ has an infinite derivative. While some of our results are similar to the known facts about Okamoto's function, there are also some notable differences and surprising new phenomena that arise when considering the higher order partial derivatives of $F_a$. We hope that this article, while highlighting a specific family of functions, will serve as a blueprint for studying functions that are only approximately self-affine.

While many papers (e.g. \cite{Darst, Falconer-2004, JKPS, Troscheit}) have studied points of non-differentiability or infinite derivatives of monotone functions, there is considerably less literature on analogous questions for non-monotone functions. Nonetheless, as was shown in \cite{Allaart,Allaart-2017a,Allaart-Kawamura}, especially the study of infinite derivatives for non-monotone self-affine functions can lead to surprising results and reveal intricate structures that are not encountered in the monotone case. The functions considered in this paper exhibit even finer and more complex behavior than those in \cite{Allaart,Allaart-2017a}.

\begin{figure}
\begin{center}
\epsfig{file=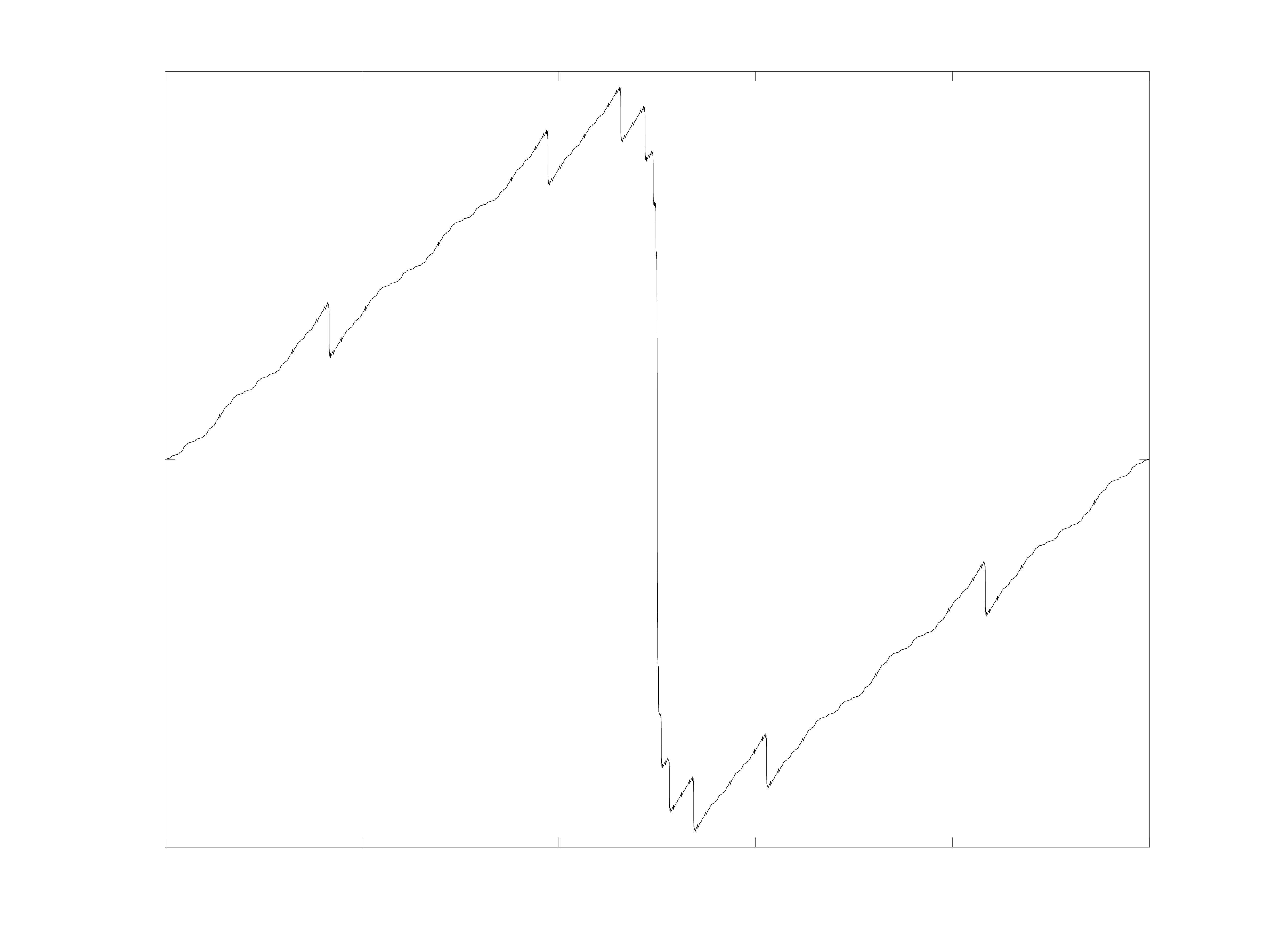, height=.3\textwidth, width=.3\textwidth} \quad
\epsfig{file=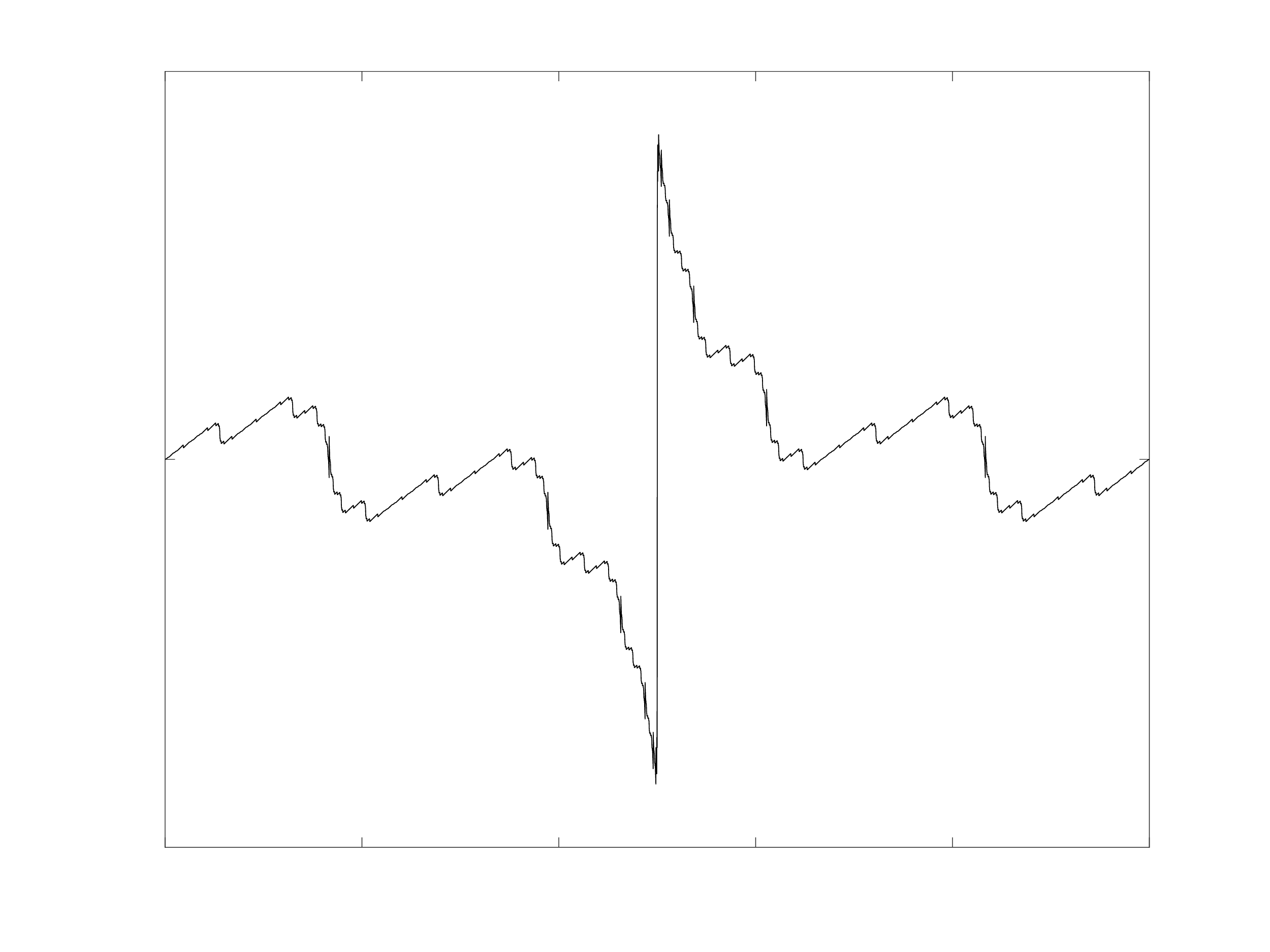, height=.3\textwidth, width=.3\textwidth} \quad
\epsfig{file=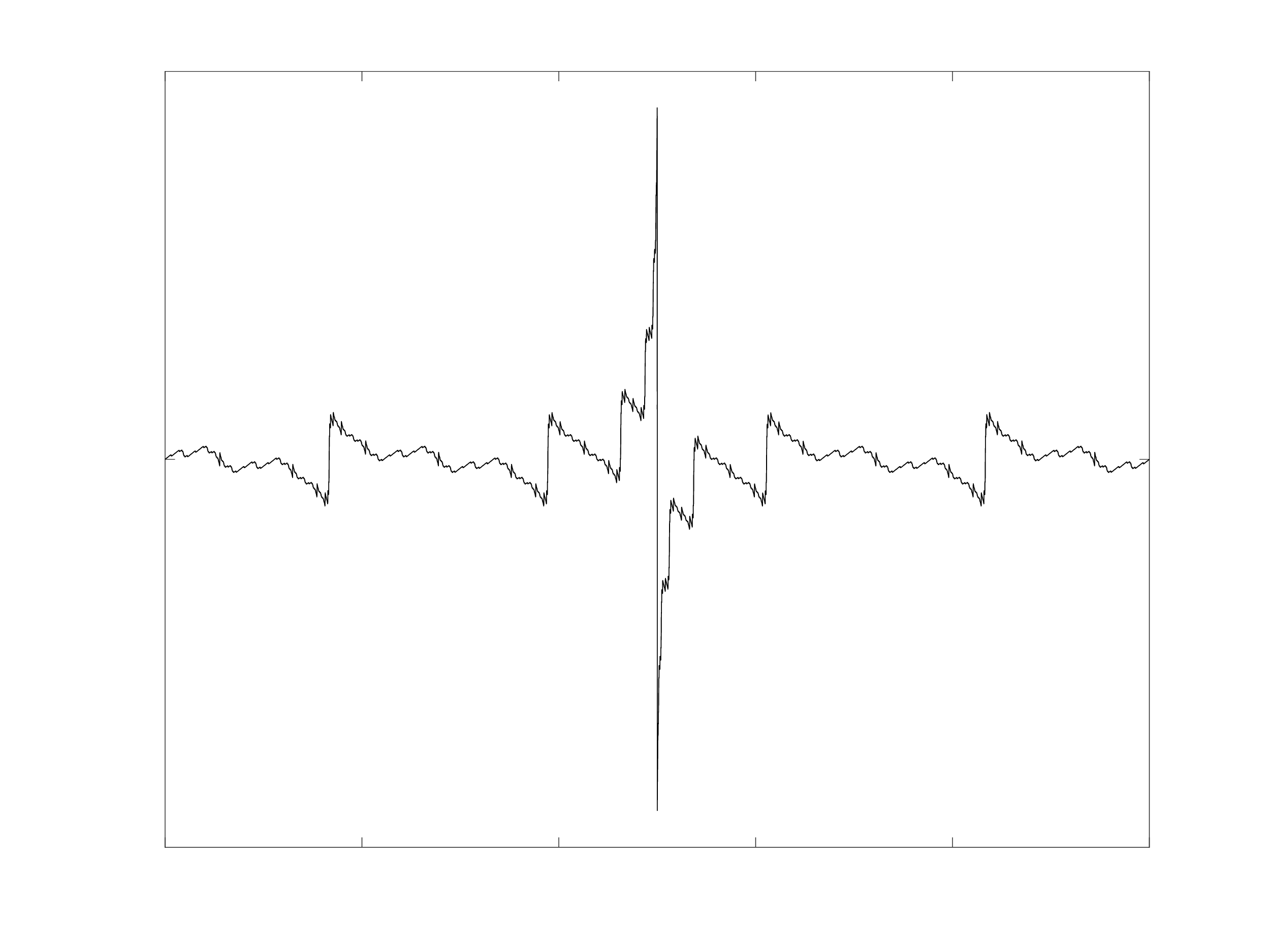, height=.3\textwidth, width=.3\textwidth}
\\
\epsfig{file=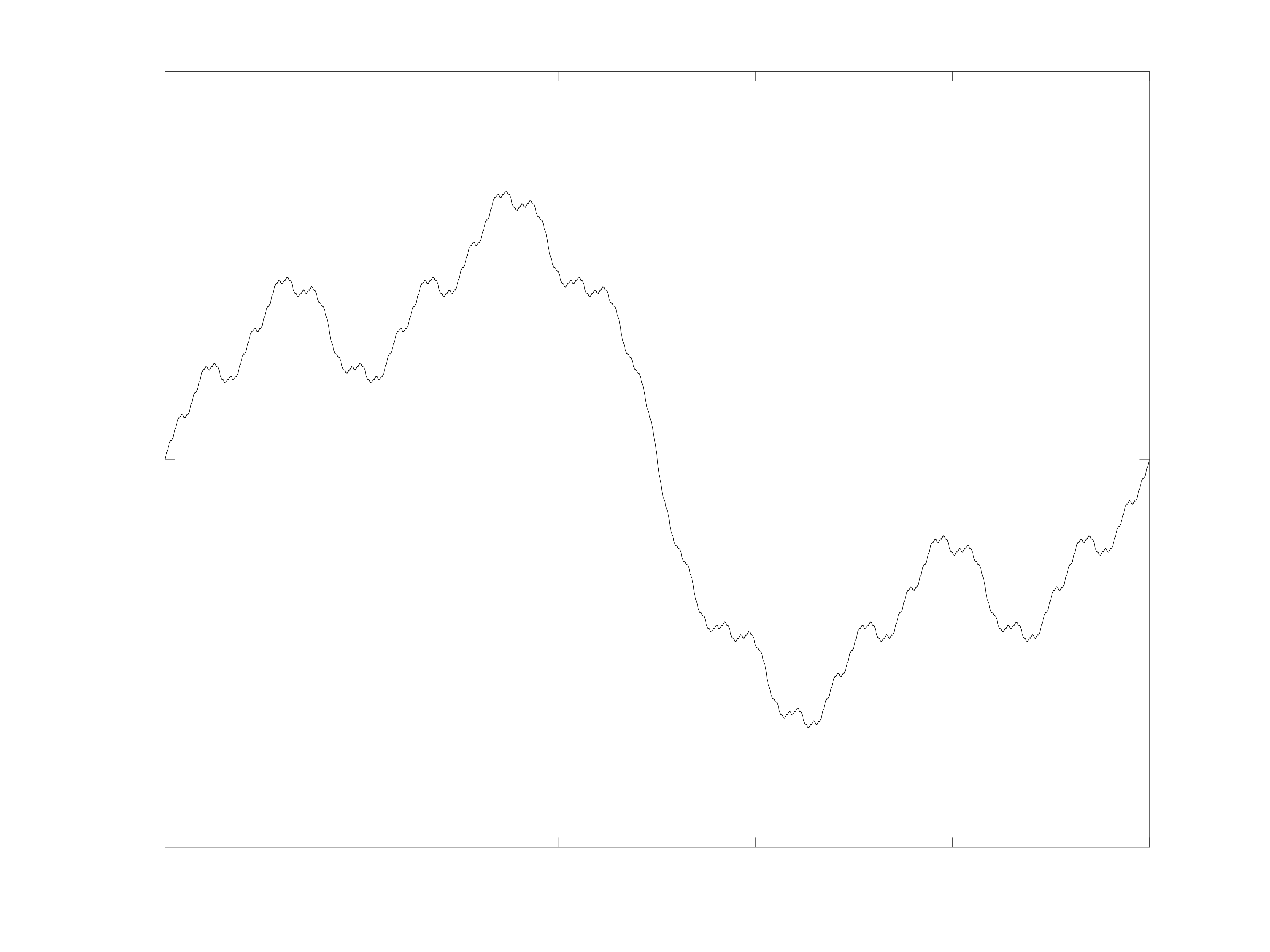, height=.3\textwidth, width=.3\textwidth} \quad
\epsfig{file=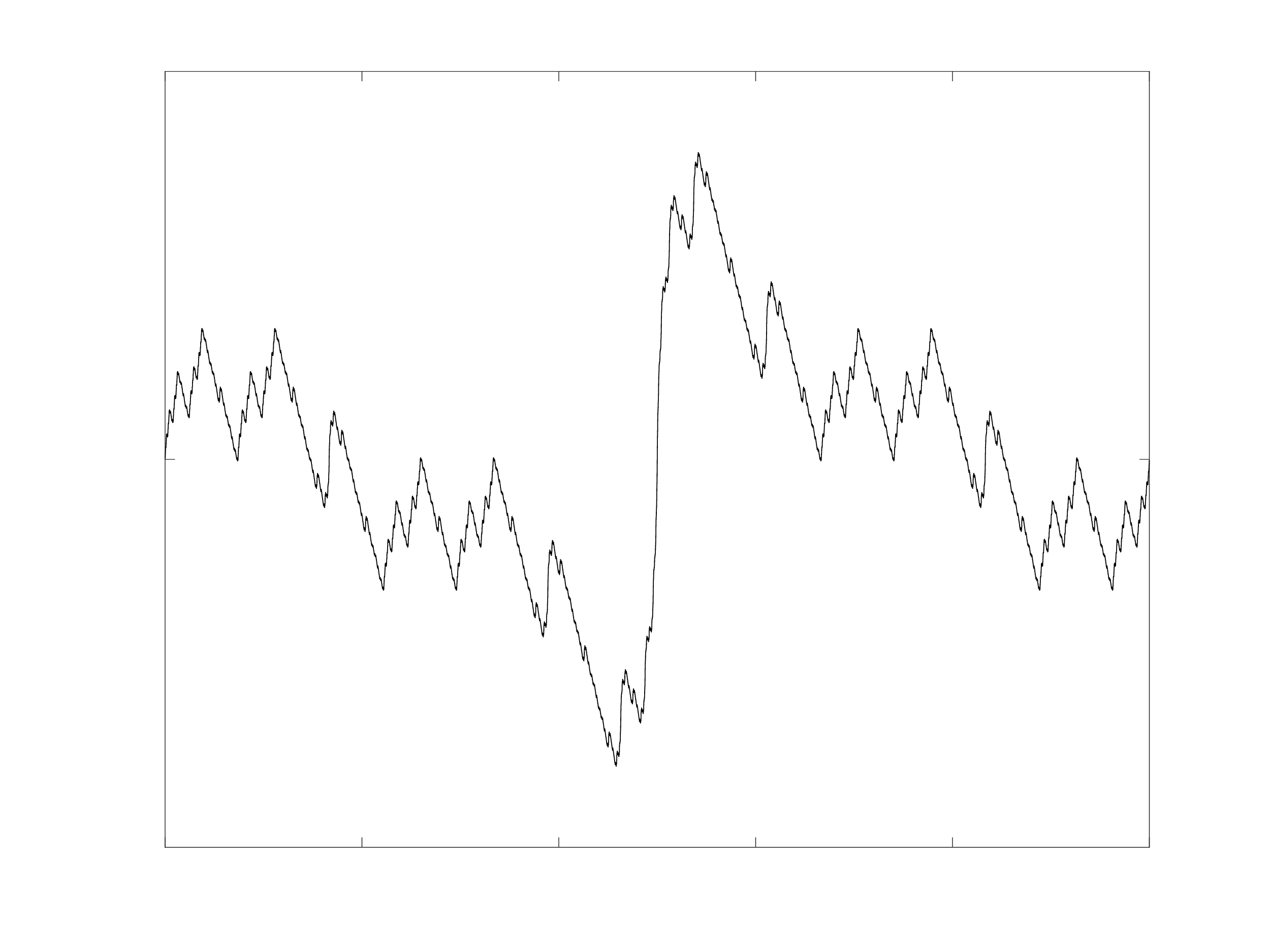, height=.3\textwidth, width=.3\textwidth} \quad
\epsfig{file=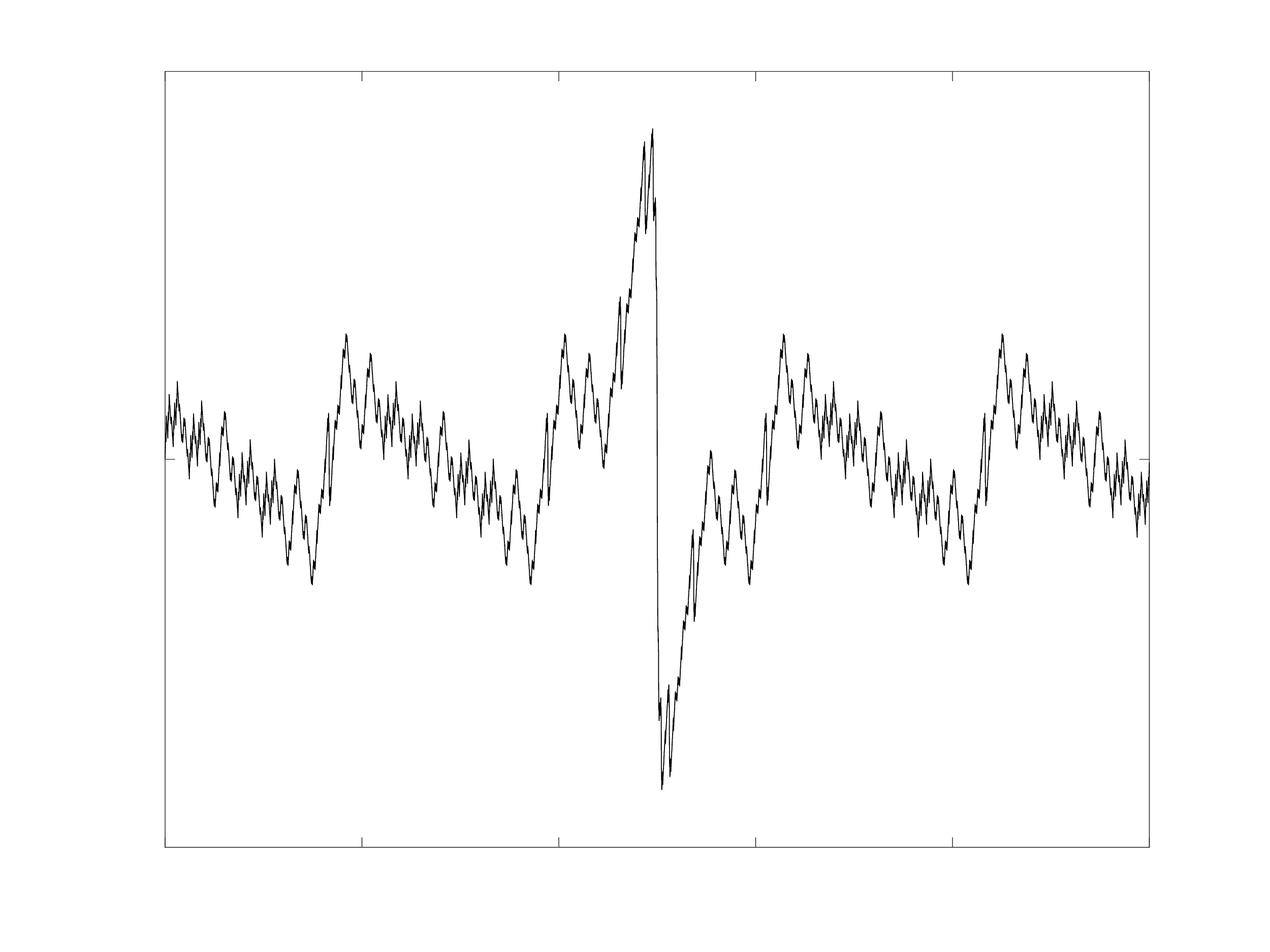, height=.3\textwidth, width=.3\textwidth}
\\
\epsfig{file=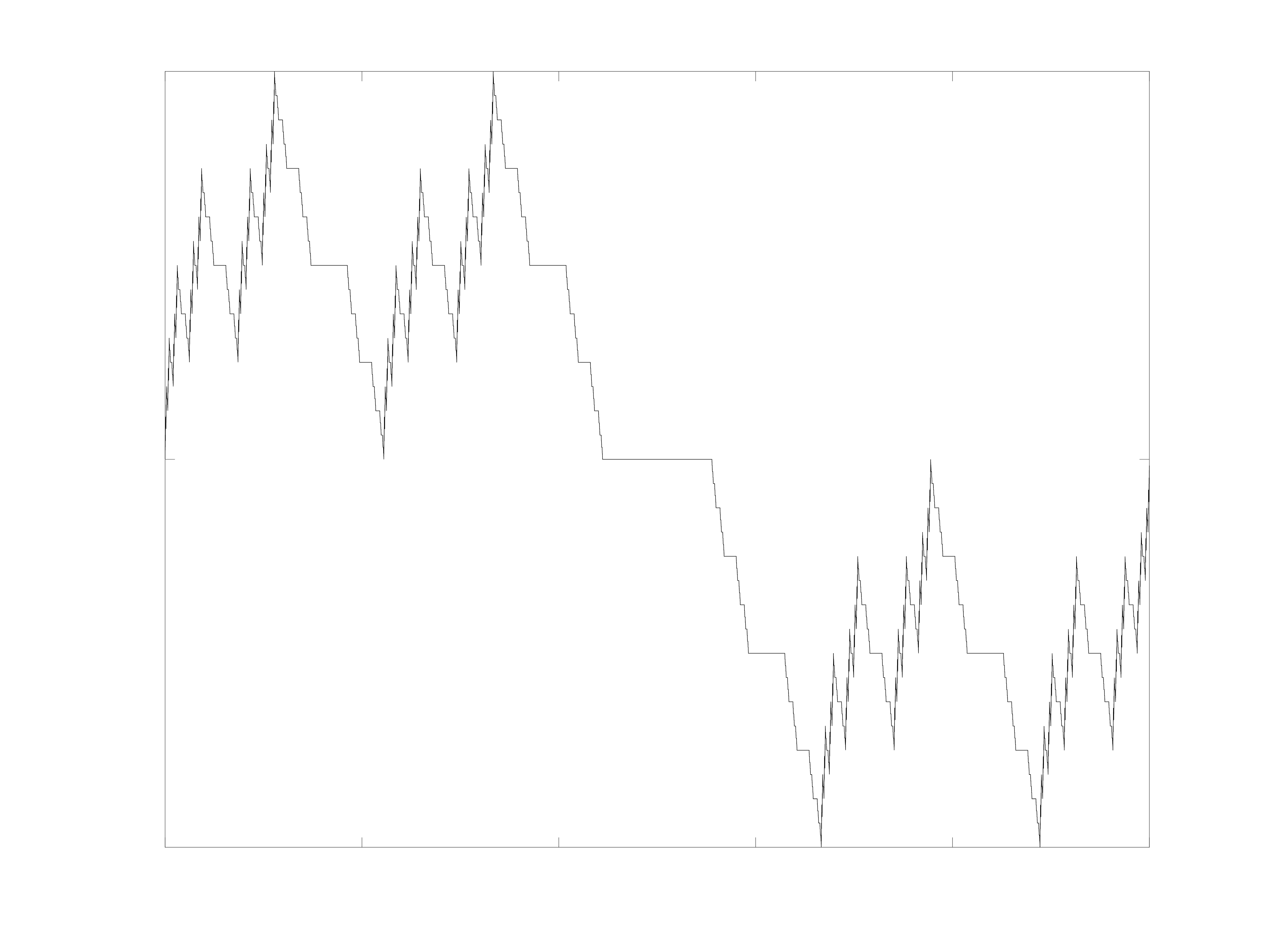, height=.3\textwidth, width=.3\textwidth} \quad
\epsfig{file=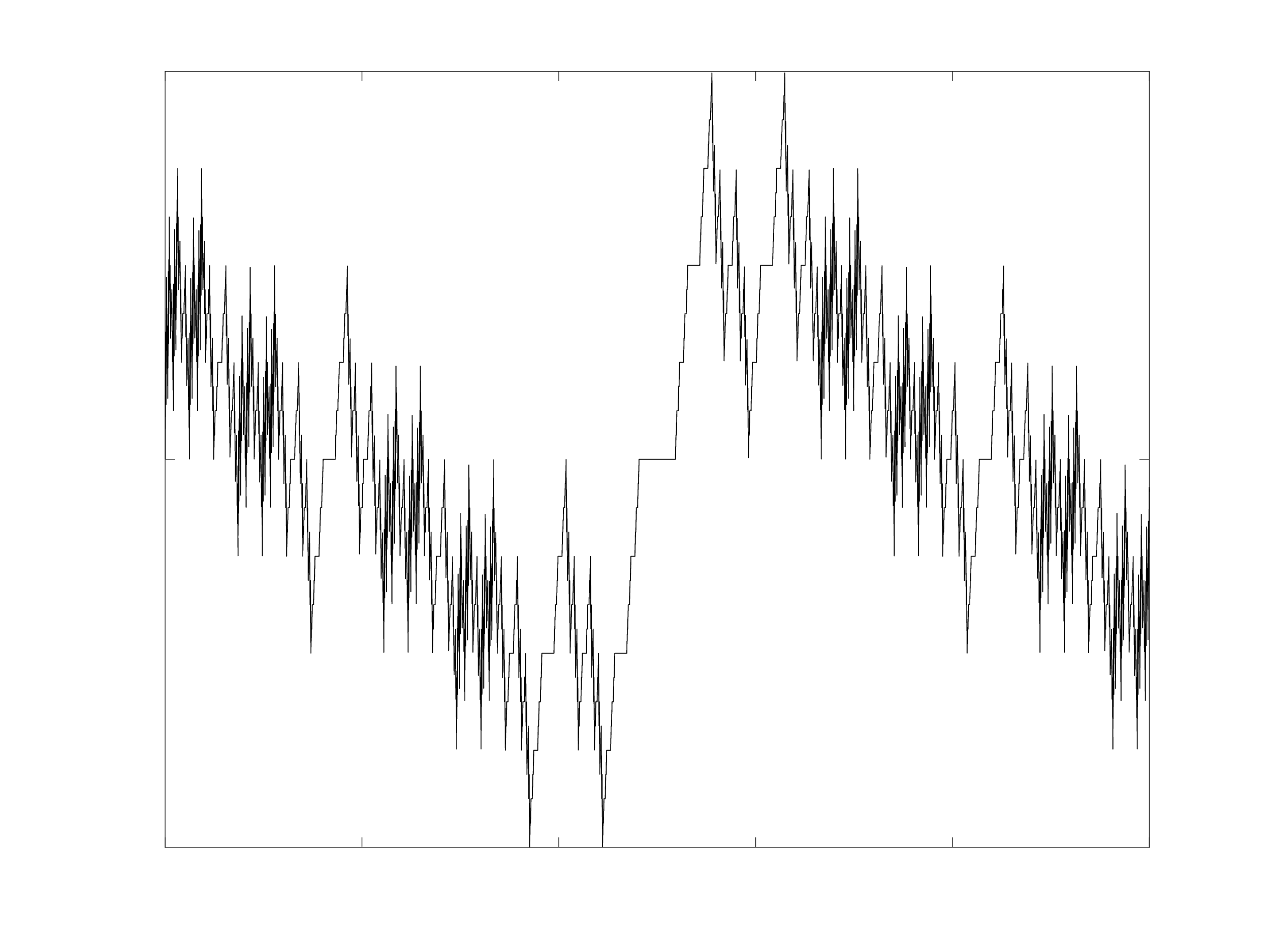, height=.3\textwidth, width=.3\textwidth} \quad
\epsfig{file=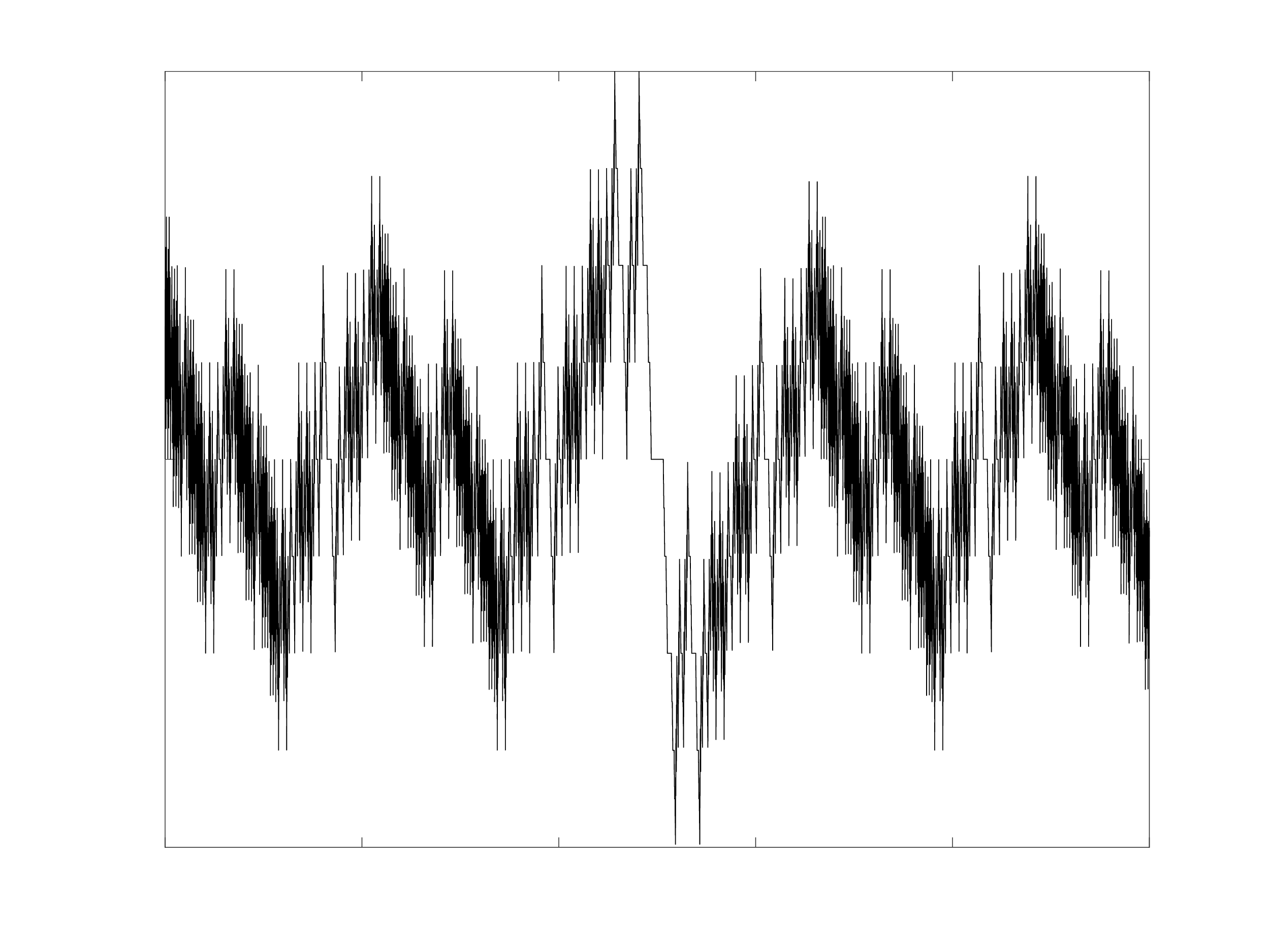, height=.3\textwidth, width=.3\textwidth}
\\
\epsfig{file=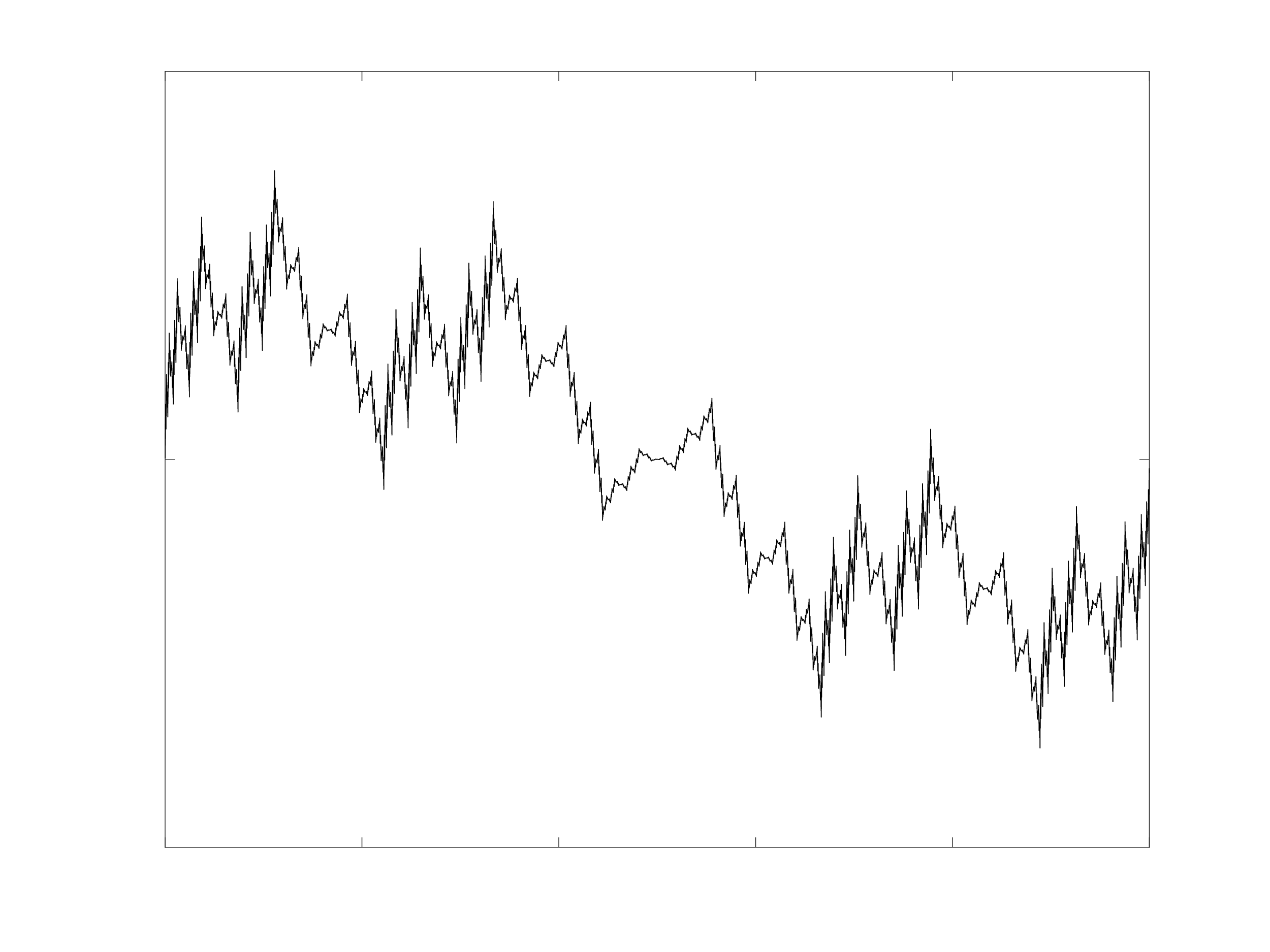, height=.3\textwidth, width=.3\textwidth} \quad
\epsfig{file=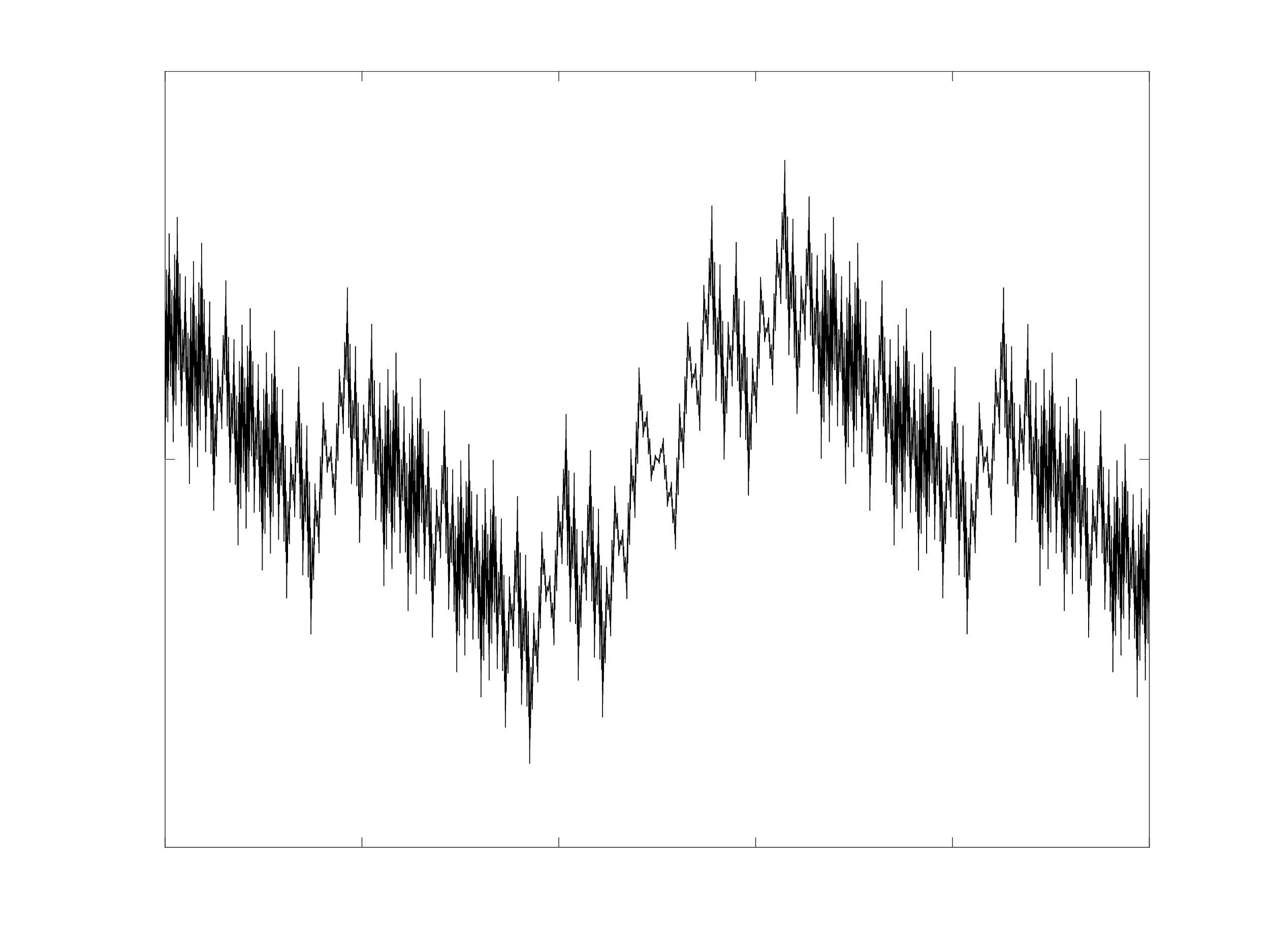, height=.3\textwidth, width=.3\textwidth} \quad
\epsfig{file=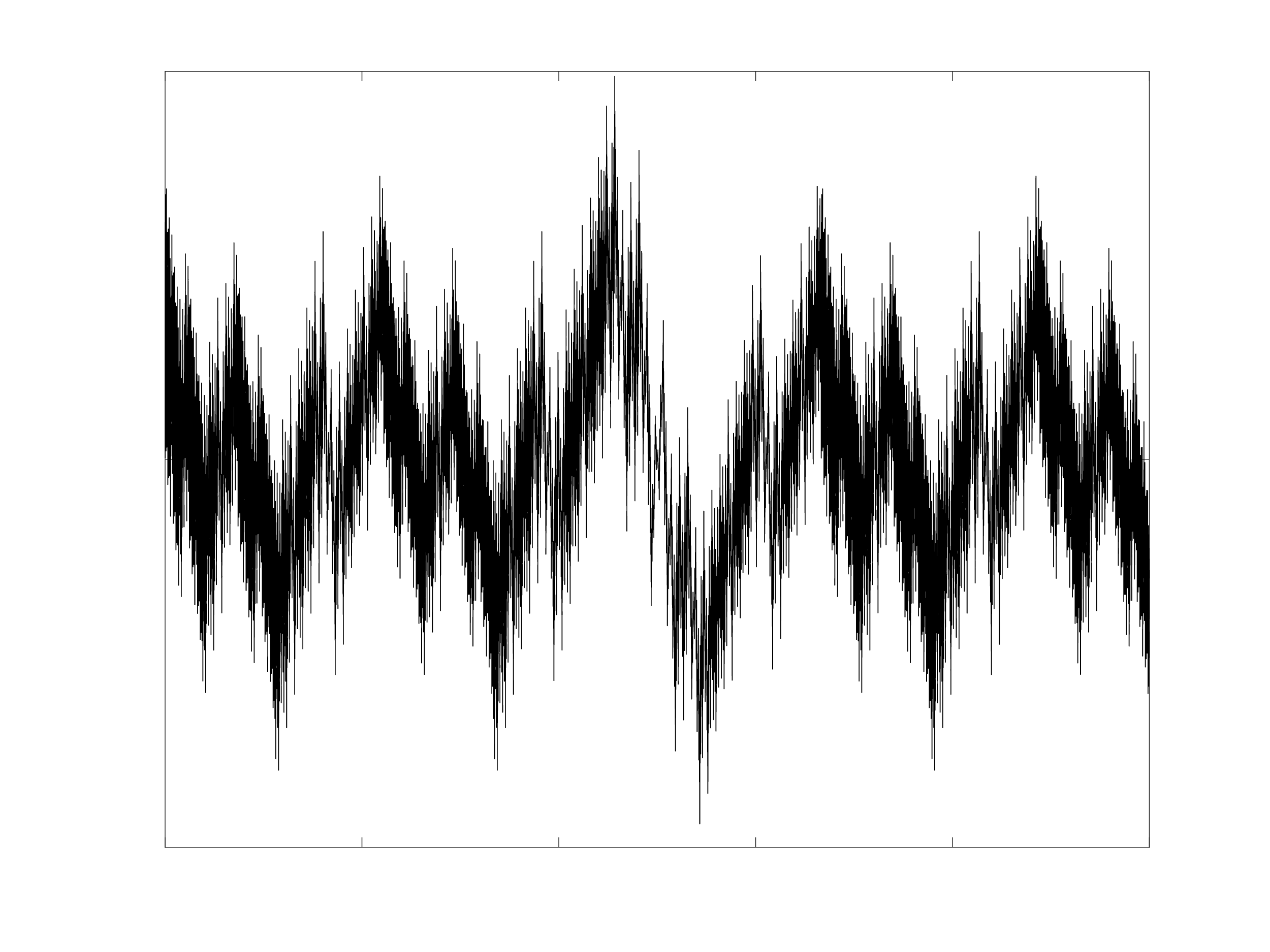, height=.3\textwidth, width=.3\textwidth}
\caption{Graphs of $M_{k,a}$ for various $a$ and $k$. Top row: $a=1/6$; second row: $a=1/3$; third row: $a=1/2$; bottom row: $a=a_0\approx .5592$. In each row, from left to right, $k=1$, $k=2$ and $k=3$. All graphs are symmetric with respect to the point $(1/2,0)$ (center of each box).}
\label{fig:M-graphs}
\end{center}
\end{figure}

\section{Statement of results}

We first address the continuity of $M_{k,a}$. From the functional equation \eqref{eq:Okamoto-FE} it follows easily that $F_a$ is H\"older continuous with exponent 
\[
\gamma:=-\log_3 \max\{a,|1-2a|\}.
\]
This means that there is a constant $C$ such that $|F_a(x)-F_a(y)|\leq C|x-y|^\gamma$ for all $x,y\in[0,1]$.

\begin{proposition} \label{prop:Holder-continuity}
For each $k$, $M_{k,a}$ is H\"older continuous with exponent $\alpha$ for any $0<\alpha<\gamma$. More precisely, there is a constant $C_k$ such that
\[
|M_{k,a}(x)-M_{k,a}(y)|\leq C_k|x-y|^\gamma \left(\log_3\frac{1}{|x-y|}\right)^k \qquad \forall x,y\in[0,1].
\]
\end{proposition}

Our next result gives the box-counting dimension of the graph of $M_{k,a}$. We denote the box-counting dimension of a set $E$ by $\dim_B E$.

\begin{theorem} \label{thm:box-dimension}
For each $k\in\NN$ and $a\in(0,1)$, we have
\[
\dim_B \graph(M_{k,a})=\dim_B \graph(F_a)=\begin{cases}
1+\log_3(4a-1) & \mbox{if $a\geq 1/2$},\\
1 & \mbox{if $a\leq 1/2$}.
\end{cases}
\]
\end{theorem}

\subsection{Finite derivatives}

To state our results on differentiability, we need some further notation. Let $(x_i)_i=(x_1,x_2,\dots)$ denote the ternary expansion of a point $x\in(0,1)$, so that $x_i\in\{0,1,2\}$ for each $i$ and $x=\sum_{i=1}^\infty x_i 3^{-i}$. If $x$ has two such representations, we choose the one ending in all zeros. For $x\in(0,1)$ and $n\in\NN$, let 
\[
l_n(x):=\#\{i\leq n: x_i=1\} 
\]
denote the number of $1$'s in the first $n$ ternary digits of $x$. Define the function
\begin{equation} \label{eq:phi}
\phi(a):=\frac{\log(3a)}{\log a-\log|1-2a|}, \qquad a\in\left(0,2/3\right]\backslash\left\{1/3,1/2\right\}.
\end{equation}
We extend $\phi$ continuously to $[0,2/3]$ by setting $\phi(0):=1$, $\phi(1/3):=1/3$, and $\phi(1/2):=0$. It can be shown that $\phi$ is strictly decreasing on $[0,1/2]$, and strictly increasing on $[1/2,2/3]$; see Figure \ref{fig:phi-and-h}. 
For each $a\in(0,1)\backslash\{1/3,1/2\}$, we also define the constant
\[
C_0:=C_0(a):=\frac{1}{\log a-\log|1-2a|}.
\]
Note that $C_0<0$ for $a<1/3$, whereas $C_0>0$ for $a>1/3$. Furthermore, $C_0(a)$ tends to $-\infty$ as $a\nearrow 1/3$, to $+\infty$ as $a\searrow 1/3$, and to $0$ as $a\to 1/2$.

\begin{proposition} \label{prop:only-zero}
Let $k\geq 1$ and $a\in(0,1)$. If $M_{k,a}$ is differentiable at a point $x$, then $M_{k,a}'(x)=0$.
\end{proposition}

\begin{theorem} \label{thm:differentiability-summary}
Let $k\geq 1$ and $a\in(0,1)$. For $x\in(0,1)$ and $n\in\NN$, define 
$$r_n(x):=l_n(x)-n\phi(a).$$
\begin{enumerate}[(a)]
\item If $0<a<1/3$, then $M_{k,a}'(x)=0$ if and only if $r_n(x)+k|C_0|\log n \to -\infty$ as $n\to\infty$;
\item If $1/3<a<2/3$ and $a\neq 1/2$, then $M_{k,a}'(x)=0$ if and only if $r_n(x)-kC_0\log n \to \infty$ as $n\to\infty$;
\item If $2/3\leq a<1$ or $a=1/3$, then $M_{k,a}$ is nowhere differentiable.
\end{enumerate}
\end{theorem}

The case $a=1/2$ is a bit special; we characterize the differentiability of $M_{k,1/2}$ in Proposition \ref{prop:one-half} below.

Defining the set
\[
D_{k,a}^0:=\{x\in(0,1): M_{k,a}'(x)=0\},
\]
we obtain the following consequence of the last theorem, which makes precise that the functions $M_{k,a}$ become progressively ``less differentiable" as $k$ increases.

\begin{corollary} \label{cor:strictly-descending-sets}
For each $0<a<2/3$ except $a=1/3$, the sequence of sets $(D_{k,a}^0)_{k\geq 0}$ is strictly decreasing; that is,
\[
D_{0,a}^0\supsetneq D_{1,a}^0 \supsetneq D_{2,a}^0 \supsetneq  \dots.
\]
Moreover, for each $k\geq 0$, $D_{k,a}^0\backslash D_{k+1,a}^0$ has positive Hausdorff dimension.
\end{corollary}


\begin{figure}
\begin{center}
\epsfig{file=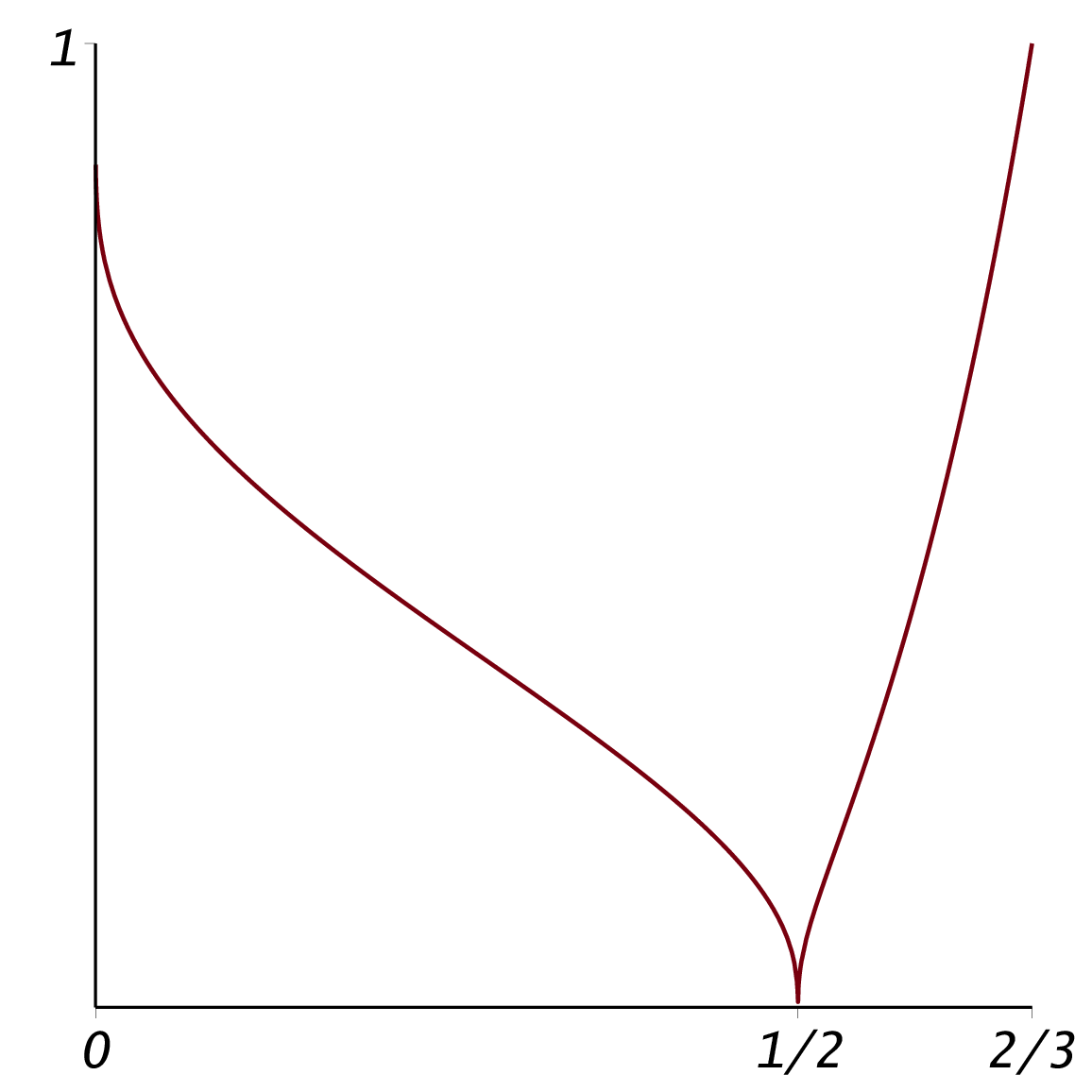, height=.2\textheight, width=.35\textwidth} \qquad\qquad
\epsfig{file=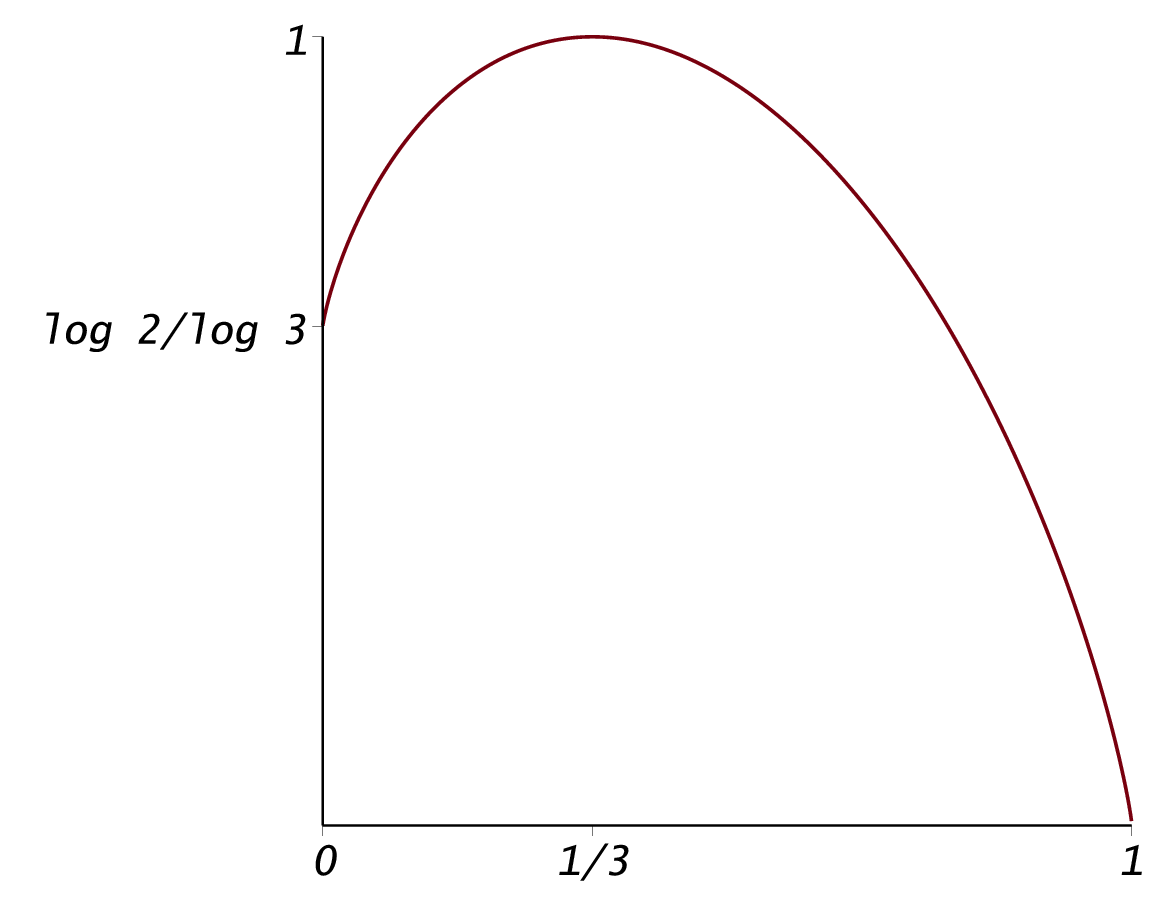, height=.2\textheight, width=.35\textwidth}
\caption{Graphs of $\phi$ (left) and $h$ (right)}
\label{fig:phi-and-h}
\end{center}
\end{figure}

Another consequence of Theorem \ref{thm:differentiability-summary} is that we obtain for $k\geq 1$ nearly the same trichotomy as in Okamoto's theorem. We include in the next result also the Hausdorff dimensions of the exceptional sets. To this end, define the functions
\begin{equation*}
h(p):=\frac{-p\log p-(1-p)\log(1-p)+(1-p)\log 2}{\log 3}, \qquad 0\leq p\leq 1,
\end{equation*}
where, following standard convention, we set $0\log 0\equiv 0$; and
\begin{equation} \label{eq:d}
d(a):=h(\phi(a)), \qquad 0\leq a\leq 2/3.
\end{equation}
The significance of the function $h$ is that it gives the Hausdorff dimension of the set of points $x$ in whose ternary expansion the digit $1$ appears with a given frequency, namely
\[
\dim_H \left\{x\in(0,1): \lim_{n\to\infty} \frac{l_n(x)}{n}=p\right\}=h(p),
\]
where $\dim_H E$ denotes the Hausdorff dimension of a set $E$.
Depending on $p$, the limit may also be replaced with a $\liminf$ or $\limsup$ and the equality with a weak or strict inequality. The function $h$ is hence maximized at $p=1/3$, with $h(1/3)=1$; see Figure \ref{fig:phi-and-h}.  We can further interpret $\phi(a)$ as the ``critical value" of the frequency of $1$'s at which the differentiability behavior changes. This makes the dimension statements in the next theorem plausible.

Finally, let $N_{k,a}$ denote the set of points at which $M_{k,a}$ has neither a finite nor an infinite derivative.

\begin{theorem} \label{thm:M_k-differentiability}
Let $a_0$ be as in Theorem \ref{thm:Okamoto}. Then for each $k\geq 1$,
\begin{enumerate}[(a)]
\item $M_{k,a}$ is nowhere differentiable if $2/3\leq a<1$ or $a=1/3$;
\item $M_{k,a}$ is nondifferentiable almost everywhere but $\dim_H D_{k,a}^0=d(a)>0$ if $a_0\leq a<2/3$;
\item $M_{k,a}$ is differentiable almost everywhere but $\dim_H N_{k,a}=\dim_H\big((0,1)\backslash D_{k,a}^0\big)=d(a)>0$ if $0<a<a_0$ and $a\neq 1/3$.
\end{enumerate}
\end{theorem}

(The dimensions of the exceptional sets are equal to $d(a)$ also for Okamoto's function (i.e. $k=0$); see \cite{Allaart}.) Note the exceptional role of the parameter value $a=1/3$ in the above theorem: it is an isolated point in the set of values of $a$ for which $M_{k,a}$ is nowhere differentiable. We contrast this with Okamoto's function itself, which for $a=1/3$ is the identity function and hence differentiable everywhere.

The graph of $d$ is shown in Figure \ref{fig:d}. Note that, since $\phi(a_0)=1/3$, $d(a)$ attains its maximum value of $1$ at both $a=1/3$ and $a=a_0$.


\begin{figure}
\begin{center}
\epsfig{file=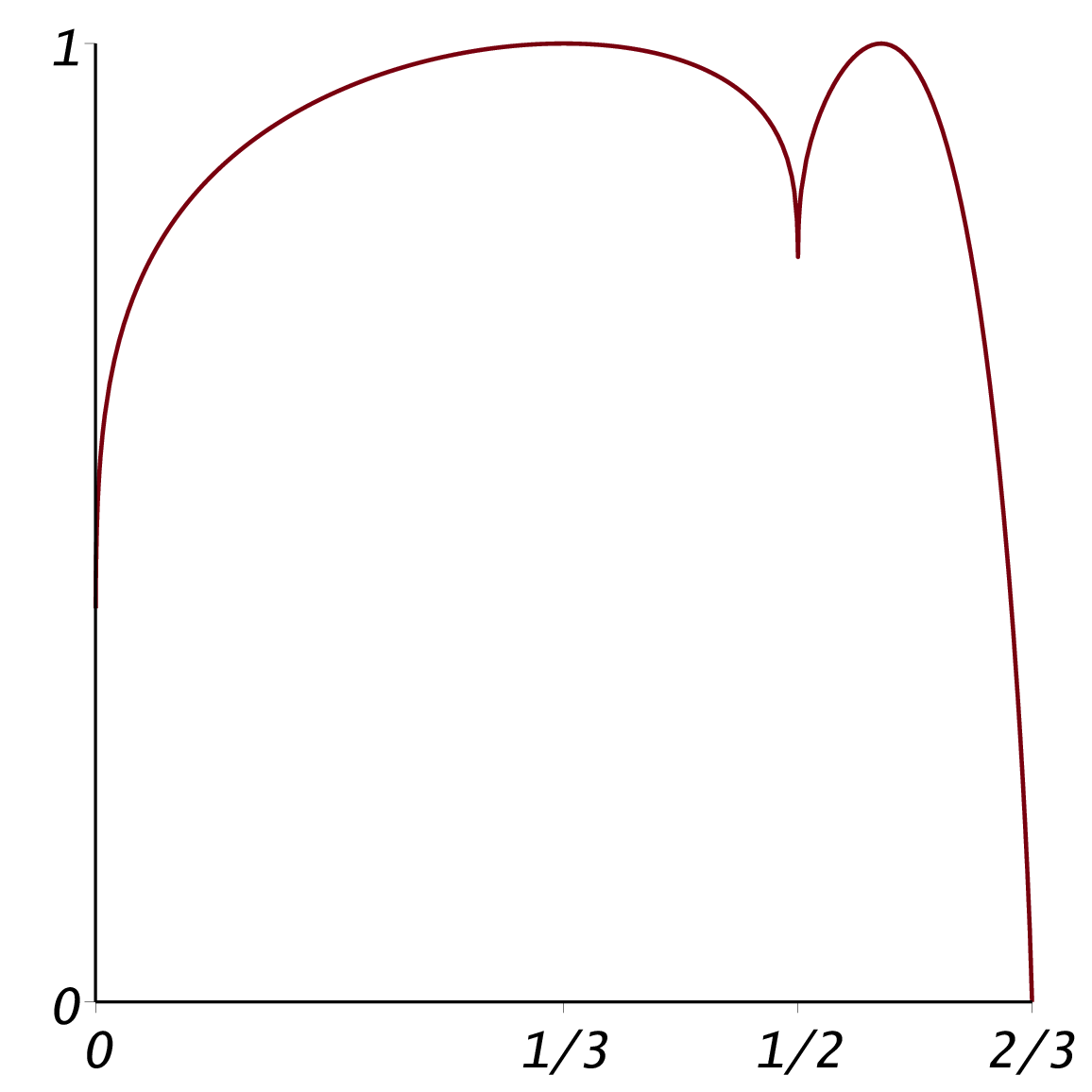, height=.2\textheight, width=.35\textwidth}
\caption{Graph of $d(a)$. Note that $d(0)=0$ and $d(1/2)=\log_3 2$.}
\label{fig:d}
\end{center}
\end{figure}


\subsection{Infinite derivatives}

A large part of the paper is devoted to the infinite derivatives of $M_{k,a}$. We will be interested in the sets
\[
D_{k,a}^{+\f}:=\{x:M_{k,a}'(x)=+\f\}, \qquad D_{k,a}^{-\f}:=\{x:M_{k,a}'(x)=-\f\},
\]
and
\[
D_{k,a}^{\pm\f}:=D_{k,a}^{+\f}\cup D_{k,a}^{-\f}.
\]
Here we find fundamentally different behavior between the parameter regions $0<a<1/2$ and $1/2<a<1$, but also between even and odd $k$. Again the case $a=1/2$ is special; we state some precise results in Section \ref{sec:infinite-derivative-one-half}.

For the case $0<a<1/2$, an important difference between $M_{k,a}$ for $k\geq 1$ and Okamoto's function is that the latter is monotone increasing and can hence only have an infinite derivative of $+\f$, whereas for $k\geq 1$, both $+\f$ and $-\f$ occur as infinite derivatives. However, the sets of points where these values occur have different Hausdorff dimensions. Precisely, we prove the following, in which $d(a)=h(\phi(a))$ as defined in \eqref{eq:d}, and 
\[
\tilde{d}(a):=h(1-2a)=-\frac{2a\log a+(1-2a)\log(1-2a)}{\log 3}.
\]
Essentially, if $M_{k,a}$ has an infinite derivative at $x$, its sign depends on whether the frequency of $1$'s in the ternary expansion of $x$ is greater or less than $1-2a$.

\begin{theorem} \label{thm:D-inf-dimensions}
Let $0<a<1/2$. Then $\dim_H D_{k,a}^{\pm\f}=d(a)$. Furthermore, if $k$ is even, then $\dim_H D_{k,a}^{+\f}=d(a)$. On the other hand, if $k$ is odd, we have the following:
\begin{enumerate}[(a)]
\item If $0<a<1/3$, then 
\[
\dim_H D_{k,a}^{+\f}=d(a)>\tilde{d}(a)=\dim_H D_{k,a}^{-\f}.
\]
\item If $a=1/3$, then
\[
\dim_H D_{k,a}^{+\f}=\dim_H D_{k,a}^{-\f}=1.
\]
\item If $1/3<a<1/2$, then 
\[
\dim_H D_{k,a}^{+\f}=\tilde{d}(a)<d(a)=\dim_H D_{k,a}^{-\f}.
\]
\end{enumerate}
\end{theorem}

Theorem \ref{thm:D-inf-dimensions} shows that derivatives of $+\f$ dominate when $a<1/3$, whereas derivatives of $-\f$ dominate when $a>1/3$. Note that (b) was proved in \cite{Dalaklis} for the case $k=1$. 

We do not know the Hausdorff dimension of $D_{k,a}^{-\f}$ for even $k$, but we show in Section \ref{sec:infinite-derivative-small-a} that this set is uncountable even then, and derive a fairly high lower bound for its dimension in the case $k=2$. We suspect that the Hausdorff dimension of $D_{k,a}^{-\f}$ is positive when $k\equiv 2 \pmod{4}$, but zero when $k\equiv 0 \pmod{4}$. In Section \ref{sec:infinite-derivative-small-a} we also give some fairly sharp explicit conditions, in terms of the frequency of $1$'s in the ternary expansion of $x$, for $M_{k,a}$ to have an infinite derivative at $x$.

As stated above, the number $1-2a$ plays the role of a critical frequency of $1$'s which separates points in $D_{k,a}^{-\f}$ from points in $D_{k,a}^{+\f}$, at least for odd $k$. When the frequency of $1$'s is exactly $1-2a$, the sign of the infinite derivative at $x$ depends in a subtle way on the deviation of $l_n(x)$ from $(1-2a)n$, and a surprising alternating pattern emerges. 

\begin{theorem} \label{thm:alternating-pattern}
Let $0<a<1/2$ and $k\geq 1$. There is a polynomial $q_k$ of degree $k$, depending only on $k$ and not on $a$, having a set of $k$ distinct real roots $t_1<t_2<\dots<t_k$ which is symmetric about $0$, such that the following holds: Suppose the limit
\begin{equation} \label{eq:delta-x}
\delta(x):=\lim_{n\to\f} \frac{(1-2a)n-l_n(x)}{\sqrt{n}}
\end{equation}
exists and is finite. Put $t_0:=-\f$ and $t_{k+1}:=+\f$. For $i=0,1,\dots,k$, if
\[
t_i\sqrt{2a(1-2a)}<\delta(x)<t_{i+1}\sqrt{2a(1-2a)},
\]
then 
\[
M_{k,a}'(x)=\begin{cases}
+\f & \mbox{if $k-i$ is even},\\
-\f & \mbox{if $k-i$ is odd}.
\end{cases}
\]
\end{theorem}

The polynomials $\{q_k\}$ are given by a simple recursion; see \eqref{eq:q-polynomials} below. For ease of presentation we have assumed the existence of the limit $\delta(x)$ in \eqref{eq:delta-x}, but this assumption can be weakened. A full statement is given in Theorem \ref{thm:alternating-infinities}.

\bigskip

We next consider the case $a>1/2$. Before stating the result, we first recall some facts about unique $\beta$-expansions. Given a number $1<\beta<2$, a {\em $\beta$-expansion} of a number $x$ is an expression of the form
\[
x=\sum_{n=1}^\infty \frac{c_n}{\beta^n},
\]
where $c_n\in\{0,1\}$ for all $n$. Such an expansion exists if and only if $x\in I_\beta:=[0,\frac{1}{\beta-1}]$, but it is in general not unique. Let $\mathcal{U}_\beta$ denote the set of all $x\in I_\beta$ having a unique expansion. 

\begin{theorem} \label{thm:infinite-derivatives}
For each $k\in\NN$ and $a\in(1/2,1)$, we have
\[
\dim_H D_{k,a}^{+\infty}=\dim_H D_{k,a}^{-\infty}=\frac{\log\beta}{\log 3}\dim_H \mathcal{U}_\beta,
\]
where $\beta:=1/a$.
\end{theorem}

For the case $k=0$ (i.e.~$F_a$), this result was established in \cite{Allaart}. Comparing with the previous theorem, we observe that for $a>1/2$ the sets $\dim_H D_{k,a}^{+\infty}$ and $\dim_H D_{k,a}^{-\infty}$ have the same dimension, whereas for $a<1/2$ they do not (unless $a=1/3$).

The size of the univoque set $\mathcal{U}_\beta$ was described in the celebrated theorem of Glendinning and Sidorov \cite{GlenSid}. Let $\varphi:=(1+\sqrt{5})/2$ be the golden ratio, and recall that the {\em Thue-Morse sequence} is the sequence $(t_j)_{j=0}^\infty$ of $0$'s and $1$'s given by $t_j=s_j \mod 2$, where $s_j$ is the number of $1$'s in the binary representation of $j$. Thus,
\begin{equation}
(t_j)_{j=0}^\infty=0110\ 1001\ 1001\ 0110\ 1001\ 0110\ 0110\ 1001\ \dots
\label{eq:Thue-Morse}
\end{equation}
Let $\hat{a}\approx .5595$ be the unique root in $(0,1)$ of the equation $\sum_{j=1}^\infty t_j a^j=1$. Its reciprocal $\hat{\beta}:=1/\hat{a}\approx 1.787$ is known as the {\em Komornik-Loreti constant}, introduced in \cite{KomLor}. (It is the smallest base in which the number 1 has a unique expansion.)

\begin{theorem}[Glendinning and Sidorov] \label{thm:Glendinning-Sidorov}
The set $\mathcal{U}_\beta$ is:
\begin{enumerate}[(i)]
\item empty if $1<\beta\leq\varphi$;
\item countably infinite if $\varphi<\beta<\hat{\beta}$;
\item uncountable but of Hausdorff dimension $0$ if $\beta=\hat{\beta}$;
\item of positive Hausdorff dimension if $\hat{\beta}<\beta\leq 2$.
\end{enumerate}
\end{theorem}

It was shown furthermore in \cite{KKL} (with a corrected proof in \cite{Allaart-Kong}) that the function $\beta\mapsto (\log\beta)\dim_H \mathcal{U}_\beta$ is an increasing devil's staircase. Thus, we obtain the following consequences for the set $D_{k,a}^\infty$:

\begin{corollary} \label{cor:infinite-derivative-size}
For each $k\in\NN$, the set $D_{k,a}^{+\infty}$ is:
\begin{enumerate}[(i)]
\item empty if $1>a\geq 1/\varphi\approx 0.6180$;
\item countably infinite, containing only rational points, if $\hat{a}<a<1/\varphi$;
\item of Hausdorff dimension $0$ if $a=\hat{a}$;
\item of positive Hausdorff dimension if $0<a<\hat{a}$.
\end{enumerate}
Moreover, on the interval $1/2<a<\hat{a}$, the function $a\mapsto \dim_H D_{k,a}^{+\infty}$ is continuous and nonincreasing in the manner of a devil's staircase; that is, there is a countable family $(I_j)_{j\in\NN}$ of disjoint subintervals of $(1/2,\hat{a})$, whose union has full Lebesgue measure in $(1/2,\hat{a})$, such that $\dim_H D_{k,a}^{+\infty}$ is constant on $I_j$ for each $j$.
\end{corollary}

(The same statements hold also for $D_{k,a}^{-\infty}$.)

\begin{remark}
{\rm
We do not know if $D_{k,a}^{\pm\infty}$ is uncountable when $a=\hat{a}$. For $k=0$ (i.e. Okamoto's function $F_a$) this question was answered in the affirmative in \cite{Allaart-2017a}. But for $k\geq 1$ the situation is more complicated. 
}
\end{remark}

The remainder of this article is organized as follows. In Section \ref{sec:prelim}, we give functional equations for $M_{k,a}$, describe a useful symmetry property, prove Proposition \ref{prop:Holder-continuity}, and develop crucial expressions for the increments and oscillation of $M_{k,a}$ over ternary rational intervals. In Section \ref{sec:box-dimension} we prove Theorem \ref{thm:box-dimension}. Section \ref{sec:differentiability} deals with the finite derivatives of $M_{k,a}$, proving Proposition \ref{prop:only-zero} and Theorems \ref{thm:differentiability-summary}-\ref{thm:M_k-differentiability}. The infinite derivatives of $M_{k,a}$ are discussed in Sections \ref{sec:infinite-derivative-one-half}--\ref{sec:infinite-derivative-high}: Section \ref{sec:infinite-derivative-one-half} deals with the special case $a=1/2$; Section \ref{sec:infinite-derivative-small-a} treats the case $a<1/2$ and proves Theorems \ref{thm:D-inf-dimensions} and \ref{thm:alternating-pattern}; and Section \ref{sec:infinite-derivative-high} deals with the case $a>1/2$, proving Theorem \ref{thm:infinite-derivatives}.

\section{Preliminaries} \label{sec:prelim}

For a number $x\in[0,1]$, we identify $x$ with its ternary expansion. Thus we write simply $M_{k,a}(x_1x_2\dots)$ instead of $M_{k,a}(0.x_1x_2\dots)$, where $x_i\in\{0,1,2\}$ for all $i$. We also write $\sigma(x)$, or simply $\sigma x$, to denote $3x\!\!\mod 1$, so $\sigma$ acts as the left shift on the ternary expansion of $x$. If $(x_i)$ is the ternary expansion of $x$, we express this by $x=(0.x_1x_2\dots)_3$, i.e. $x=\sum_{i=1}^\f x_i/3^i$.

\begin{lemma}[Functional equations] \label{lem:FE}
For $x=(0.x_1x_2\dots)_3$, we have
\begin{gather}
F_a(x)=\begin{cases}
aF_a(\sigma x) & \mbox{if $x_1=0$},\\
(1-2a)F_a(\sigma x)+a & \mbox{if $x_1=1$},\\
aF_a(\sigma x)+1-a & \mbox{if $x_1=2$},
\end{cases} \label{eq:Fa-FE} \\
M_{1,a}(x)=\begin{cases}
aM_{1,a}(\sigma x)+F_a(\sigma x) & \mbox{if $x_1=0$},\\
(1-2a)M_{1,a}(\sigma x)-2F_a(\sigma x)+1 & \mbox{if $x_1=1$},\\
aM_{1,a}(\sigma x)+F_a(\sigma x)-1 & \mbox{if $x_1=2$},
\end{cases} \label{eq:M1-FE}
\end{gather}
and for $k\geq 2$,
\begin{equation} \label{eq:Mk-FE}
M_{k,a}(x)=\begin{cases}
aM_{k,a}(\sigma x)+kM_{k-1,a}(\sigma x) & \mbox{if $x_1=0$ or $2$},\\
(1-2a)M_{k,a}(\sigma x)-2k M_{k-1,a}(\sigma x) & \mbox{if $x_1=1$}.
\end{cases}
\end{equation}
\end{lemma}

\begin{proof}
The functional equation for $F_a$ is a restatement of \eqref{eq:Okamoto-FE}. The others follow from it by repeated differentiation of \eqref{eq:Fa-FE} with respect to $a$.
\end{proof}

\begin{lemma}[Symmetry] \label{lem:symmetry}
Let $I=[x_l,x_r]=[j3^{-n},(j+1)3^{-n}]$ be a ternary interval. Then for each $k\in\NN_0$,
\begin{equation} \label{eq:symmetry}
M_{k,a}(x_l+u)-M_{k,a}(x_l)=M_{k,a}(x_r)-M_{k,a}(x_r-u) \qquad\mbox{for all $u\in[0,3^{-n}]$}.
\end{equation}
In particular, for each $k\geq 1$ we have
\[
M_{k,a}(1-x)=-M_{k,a}(x) \qquad \mbox{for all $x\in[0,1]$}.
\]
\end{lemma}

\begin{proof}
Note that $M_{0,a}=F_a$ satisfies \eqref{eq:symmetry}. Hence, by differentiating repeatedly with respect to $a$, we see that $M_{k,a}$ satisfies this property for all $k\in\NN$. The second statement follows from the first since, by the functional equations \eqref{eq:M1-FE} and \eqref{eq:Mk-FE} and using $F_a(0)=0$ and $F_a(1)=1$, it follows easily by induction that $M_{k,a}(0)=M_{k,a}(1)=0$ for all $k\geq 1$.
\end{proof}

The following explicit expression for $F_a(x)$ was first given in \cite{Kobayashi}.

\begin{lemma} \label{lem:Okamoto-series}
For $x=(0.x_1 x_2\dots)_3$, we have for $F_a(x)$ the explicit expression
\begin{equation} \label{eq:Okamoto-series}
F_a(x)=\sum_{n=0}^\f a^{n-l_n(x)}(1-2a)^{l_n(x)}q(x_{n+1}),
\end{equation}
where $q(0)=0, q(1)=a$ and $q(2)=1-a$.
\end{lemma}

Define the closed intervals
$$I_{n,j}:=[j3^{-n},(j+1)3^{-n}], \qquad n\in\NN, \quad j=0,1,\dots,3^n-1.$$ 
Further, for $n\in\NN$ and $x\in(0,1)$, let $I_n(x)$ denote the ternary interval $I_{n,j}$ containing $x$. (If $x=j/3^n$, we choose $I_n(x):=[j/3^n,(j+1)/3^n]$ for definiteness.)

\begin{proof}[Proof of Proposition \ref{prop:Holder-continuity}]
Write $l_n:=l_n(x)$ for brevity. Differentiating \eqref{eq:Okamoto-series} $k$ times with respect to $a$ using Leibniz' generalized product rule gives
\begin{equation} \label{eq:M-ternary-expression}
M_{k,a}(x)=\sum_{n=0}^\f \sum_{i_1+i_2+i_3=k} \binom{k}{i_1,i_2,i_3} \frac{\partial^{i_1}}{\partial a^{i_1}}a^{n-l_n}\cdot \frac{\partial^{i_2}}{\partial a^{i_2}}(1-2a)^{l_n} \cdot \frac{\partial^{i_3}}{\partial a^{i_3}}q(x_{n+1}),
\end{equation}
where we sum over triples $(i_1,i_2,i_3)$ of nonnegative integers summing to $k$. Assume first that $a\neq 1/2$. Since $|\partial^i q(x_{n+1})/\partial a^i|\leq 1$ for all $i\geq 0$, we can estimate
\begin{align*}
&\left|\sum_{i_1+i_2+i_3=k} \binom{k}{i_1,i_2,i_3} \frac{\partial^{i_1}}{\partial a^{i_1}}a^{n-l_n}\cdot \frac{\partial^{i_2}}{\partial a^{i_2}}(1-2a)^{l_n} \cdot \frac{\partial^{i_3}}{\partial a^{i_3}}q(x_{n+1})\right|\\
&\qquad \leq \sum_{i_1+i_2+i_3=k} \binom{k}{i_1,i_2,i_3}\frac{(n-l_n)!}{(n-l_n-i_1)!}a^{n-l_n-i_1} \cdot \frac{l_n!}{(l_n-i_2)!}|1-2a|^{l_n-i_2} 2^{i_2} \cdot 1\\
&\qquad \leq \sum_{i_1+i_2+i_3=k} \binom{k}{i_1,i_2,i_3}(n-l_n)^{i_1}a^{n-l_n-i_1} \cdot l_n^{i_2}|1-2a|^{l_n-i_2} 2^{i_2} \cdot 1^{i_3}\\
&\qquad = a^{n-l_n}|1-2a|^{l_n}\left(\frac{n-l_n}{a}+\frac{2l_n}{|1-2a|}+1\right)^k\\
&\qquad \leq Ca^{n-l_n}|1-2a|^{l_n} n^k \leq Cb^n n^k,
\end{align*}
where $C$ is a sufficiently large constant, $b:=\max\{a,|1-2a|\}$, and the second-to-last inequality holds for all $n$, since $l_n\leq n$.

Now let $x,y\in[0,1]$ with $x<y$ and take $m\in\NN$ such that $3^{-m-1}\leq|x-y|<3^{-m}$. Suppose first that $x$ and $y$ belong to the same ternary interval $I_{m,j}$ of order $m$. Then $x$ and $y$ agree in their first $m$ ternary digits. Since the terms in the summation \eqref{eq:M-ternary-expression} for $n<m$ depend only on the first $m$ ternary digits of $x$, it follows from the above that
\[
|M_{k,a}(x)-M_{k,a}(y)|\leq 2\sum_{n=m}^\f Cb^n n^k\leq C'b^m m^k
\]
for some constant $C'$ depending only on $b$. If instead, $x$ and $y$ belong to adjacent ternary intervals, say $I_{m,j}$ and $I_{m,j+1}$, we let $z=(j+1)3^{-m}$ denote the common endpoint of $I_{m,j}$ and $I_{m,j+1}$ and apply the above estimate to the pairs $(x,z)$ and $(z,y)$ to get the same inequality but with $2C'$ in place of $C'$. Note that, although $x$ and $z$ do not have their first $m$ ternary digits in common, we have by the symmetry property of Lemma \ref{lem:symmetry} that
\[
M_{k,a}(x)-M_{k,a}(z)=M_{k,a}(j3^{-m})-M_{k,a}(j3^{-m}+z-x),
\]
and the points $j3^{-m}$ and $j3^{-m}+z-x$ do agree in their first $m$ digits. Now observe that
\[
b^m m^k=b^{-1}(3^{m+1})^{\log_3 b}m^k<b^{-1}|x-y|^{-\log_3 b}\left(\log\frac{1}{|x-y|}\right)^k.
\]
Recalling that $\gamma=-\log_3 b$, this completes the proof for the case $a\neq 1/2$.

When $a=1/2$, we have that
\[
\frac{\partial^i}{\partial a^i}(1-2a)^{l_n}\bigg|_{a=1/2}=\begin{cases}
l_n!(-2)^{l_n} & \mbox{if $i=l_n$},\\
0 & \mbox{otherwise}.
\end{cases}
\]
Thus, the $n$th term in \eqref{eq:M-ternary-expression} vanishes when $l_n>k$, while for $l_n\leq k$ we can estimate its absolute value by
\begin{align*}
\sum_{i_1+i_3=k-l_n} &\binom{k}{i_1,l_n,i_3}(n-l_n)^{i_1}\left(1/2\right)^{n-l_n-i_1}l_n! 2^{l_n}\\
&=\left(1/2\right)^n 2^{2l_n}\frac{k!}{(k-l_n)!} \sum_{i_1=0}^{k-l_n} \binom{k-l_n}{i_1}\big(2(n-l_n)\big)^{i_1}\\
&\leq k! 2^{2k} \left(1/2\right)^n \big(2(n-l_n)+1\big)^{k-l_n}
\leq C_k \left(1/2\right)^n n^k
\end{align*}
for a suitably large constant $C_k$. The rest of the proof proceeds in the same way as before.
\end{proof}

We next develop some explicit expressions for the increments of $M_{k,a}$ over ternary rational intervals, as well as bounds on the oscillation over such intervals. 
For $n\in\NN$ and $j=0,1,\dots,3^n-1$, define
\[
\Delta_{k,a}(I_{n,j}):=M_{k,a}\left(\frac{j+1}{3^n}\right)-M_{k,a}\left(\frac{j}{3^n}\right).
\]
In what follows, we also let $l(i)$ denote the number of $1$'s in the ternary representation of $i$ for any non-negative integer $i$. 
We write $0^\f$ for the infinite sequence of all $0$'s.

We first consider the case $a\neq 1/2$. The description of the increments of $M_{k,a}$ over ternary rational intervals given in the next lemma is central to many of the proofs in this paper.

\begin{lemma} \label{lem:ternary-increments}
Assume $a\neq 1/2$. There is a polynomial $P_k(n,l)$ in $n$ and $l$ of the form
\[
P_k(n,l)=\{(1-2a)n-l\}^k+R_k(n,l),
\]
where $R_k(n,l)$ is a polynomial of degree at most $k-1$ in $n$ and $l$, such that, for all $n\in\NN$ and $0\leq i<3^n$,
\[
\Delta_{k,a}(I_{n,i})=a^{n-l(i)-k}(1-2a)^{l(i)-k}P_k(n,l(i)).
\]
\end{lemma}

(By the degree of a polynomial $R(n,l)$, we mean the maximum value of $p+q$ such that $R(n,l)$ contains a term $c_{p,q}n^p l^q$. Although $P_k(n,l)$ and $R_k(n,l)$ also depend on the parameter $a$, we suppress this dependence in the notation.)
For example, 
\begin{align*}
R_1(n,l)&=0,\\
R_2(n,l)&=-(1-2a)\{(1-2a)n-l\}-2al, 
\end{align*}
etc. 

\begin{proof}
Note that
\[
\Delta_{0,a}(I_{n,i})=a^{n-l(i)}(1-2a)^{l(i)}.
\]
Write $l(i)=l$ for short. Differentiating $k$ times with respect to $a$ gives, using Leibniz's rule,
\begin{align}
\Delta_{k,a}(I_{n,i})&=\sum_{j=0}^k \binom{k}{j}\frac{\partial^j}{\partial a^j}\big(a^{n-l}\big)\cdot\frac{\partial^{k-j}}{\partial a^{k-j}}\big((1-2a)^l\big)  \label{eq:Leibniz-formula} \\
&=\sum_{j=0}^k \binom{k}{j}\frac{(n-l)!}{(n-l-j)!} a^{n-l-j} \frac{l!}{(l-k+j)!} (1-2a)^{l-k+j}(-2)^{k-j} \notag \\
&=a^{n-l-k}(1-2a)^{l-k}\sum_{j=0}^k \binom{k}{j}\frac{(n-l)!}{(n-l-j)!} a^{k-j} \frac{l!}{(l-k+j)!} (1-2a)^{j}(-2)^{k-j}. \notag
\end{align}
Now
\[
\frac{(n-l)!}{(n-l-j)!}=(n-l)^j + \ \mbox{lower order terms}
\]
and similarly,
\[
\frac{l!}{(l-k+j)!}=l^{k-j} + \ \mbox{lower order terms}.
\]
Thus,
\begin{align*}
\Delta_{k,a}(I_{n,i})
&=a^{n-l-k}(1-2a)^{l-k}\left(\sum_{j=0}^k \binom{k}{j}(n-l)^j l^{k-j}(-2a)^{k-j}(1-2a)^{j}+R_k(n,l)\right)\\
&=a^{n-l-k}(1-2a)^{l-k}\left(\big\{(n-l)(1-2a)-2al\big\}^k+R_k(n,l)\right)\\
&=a^{n-l-k}(1-2a)^{l-k}\left(\{(1-2a)n-l\}^k+R_k(n,l)\right),
\end{align*}
for some polynomial $R_k(n,l)$ of degree at most $k-1$.
\end{proof}

For an interval $I$ and function $f$, denote
\[
\osc(f;I):=\sup_{x,y\in I}|f(x)-f(y)|.
\]

\begin{lemma} \label{lem:oscillation-bounds}
Assume $a\neq 1/2$. There is for each $k\in\NN_0$ a constant $C_k$ such that, for every $n\in\NN$ and $j=0,1,\dots,3^n-1$,
\begin{equation} \label{eq:oscillation-bound}
\osc(M_{k,a};I_{n,j})\leq C_k a^{n-l(j)}|1-2a|^{l(j)}n^k.
\end{equation}
\end{lemma}

\begin{proof}
This is immediate from the proof of Proposition \ref{prop:Holder-continuity}.
\end{proof}


For $a=1/2$ we have the following explicit expression for the increments over ternary intervals.

\begin{lemma} \label{lem:difference-formulas}
Assume $0\leq l\leq k\leq n$, and let $j\in\{0,1,\dots,3^n-1\}$ be such that $l(j)=l$. Then
\[
\Delta_{k,1/2}(I_{n,j})=\frac{k!}{(k-l)!}\cdot\frac{(n-l)!}{(n-k)!}\left(1/2\right)^{n-k}(-2)^l.
\]
\end{lemma}

\begin{proof}
Starting from \eqref{eq:Leibniz-formula}, set $a=1/2$ and note that since $l\leq k\leq n$, only the term with $j=k-l$ does not vanish.
\end{proof}

In the next lemma and its proof, we write $M_k:=M_{k,1/2}$ to avoid clutter.

\begin{lemma} \label{prop:box-containment}
Fix $n\geq 2k-1$, and let $I=[x_l,x_r]=I_{n,j}$ be a ternary interval of order $n$. Then 
\[
\min\{M_{k}(x_l),M_{k}(x_r)\}\leq M_{k}(x)\leq \max\{M_{k}(x_l),M_{k}(x_r)\} \qquad\forall\,x\in I.
\]
In other words, the graph of $M_{k}$ over $I$ lies inside the rectangular `box' with opposite vertices $(x_l,M_{k}(x_l))$ and $(x_r,M_{k}(x_r))$ and edges parallel to the coordinate axes.
\end{lemma}

Note the conclusion of the lemma fails for $n\leq 2k-2$, as in that case the functional equation shows that $M_{k}(0)=M_{k}(3^{-n})=0$.

\begin{proof}
Since $M_{k}$ is continuous, the conclusion will follow by induction if we show that for each $m\geq 2k-1$ and each $j\in\{0,1,\dots,3^m-1\}$, $M_k\big((3j+1)3^{-m-1}\big)$ and $M_k\big((3j+2)3^{-m-1}\big)$ lie between $M_k(j3^{-m})$ and $M_k((j+1)3^{-m})$. By symmetry (cf. Lemma \ref{lem:symmetry}), it is enough to show this for $(3j+1)3^{-m-1}$.
Let $l:=l(j)$. Note that
\[
M_{k}\left(\frac{j+1}{3^m}\right)-M_{k}\left(\frac{j}{3^m}\right)=\Delta_{k,1/2}(I_{m,j})
\]
and
\[
M_{k}\left(\frac{3j+1}{3^{m+1}}\right)-M_{k}\left(\frac{j}{3^m}\right)=\Delta_{k,1/2}(I_{m+1,3j}).
\]
Since $l(3j)=l(j)=l$, Lemma \ref{lem:difference-formulas} shows that $\Delta_{k,1/2}(I_{m,j})$ and $\Delta_{k,1/2}(I_{m+1,3j})$ have the same sign, and
\[
\frac{|\Delta_{k,1/2}(I_{m+1,3j})|}{|\Delta_{k,1/2}(I_{m,j})|}=1/2\cdot \frac{m+1-l}{m+1-k}\leq 1,
\]
since $m\geq 2k-1\geq 2k-l-1$. Hence, $M_k\big((3j+1)3^{-m-1}\big)$ lies between $M_k(j3^{-m})$ and $M_k((j+1)3^{-m})$, as desired.
\end{proof}

\section{Box-counting dimension of the graph} \label{sec:box-dimension}

\begin{proof}[Proof of Theorem \ref{thm:box-dimension}]
Let $G_{k,a}$ denote the graph of $M_{k,a}$.

{\em Step 1: Upper bound.} Assume first that $a\neq 1/2$. For fixed $n$, the number of words $\mathbf{i}\in\{0,1,2\}^n$ with $l$ zeros is
$\binom{n}{l}2^{n-l}.$
Lemma \ref{lem:oscillation-bounds} implies that the number of $3^{-n}$-mesh cubes intersecting $G_{k,a}$ is bounded above by
\begin{align*}
N_{3^{-n}}(G_{k,a})&\leq C_k\left(3^n+2\right)n^k\sum_{l=0}^n\binom{n}{l}2^{n-l}a^{n-l}|1-2a|^l\\
&= C_k\left(3^n+2\right)n^{k}\big(2a+|1-2a|\big)^n, 
\end{align*}
where $C_k$ is the constant from Lemma \ref{lem:oscillation-bounds}.
The stated upper bound follows upon observing that
\[
2a+|1-2a|=\begin{cases} 1 & \mbox{if $a<1/2$},\\
4a-1 & \mbox{if $a>1/2$}.
\end{cases}
\]

If $a=1/2$, then by Lemmas \ref{lem:difference-formulas} and \ref{prop:box-containment}, we have, for an interval $I_{n,j}$ with $l(j)=l$,
\[
\osc(M_{k,1/2};I_{n,j})=|\Delta_{k,1/2}(I_{n,j})|=\frac{k!}{(k-l)!}\cdot\frac{(n-l)!}{(n-k)!}\left(1/2\right)^{n-k}2^{l},
\] 
for $0\leq l\leq k\leq n$. Furthermore, if $l(j)>k$, then $M_{k,1/2}$ is constant over $I_{n,j}$ and so the graph of $M_{k,1/2}$ over $I_{n,j}$ intersects only one $3^{-n}$-mesh square. Thus, for $n\geq k$,
\begin{align*}
N_{3^{-n}}\left(G_{k,1/2}\right)&\leq 3^n+\left(3^n+2\right)\sum_{l=0}^k\binom{n}{l}2^{n-l}\frac{k!}{(k-l)!}\cdot\frac{(n-l)!}{(n-k)!}\left(1/2\right)^{n-k}2^l\\
&=3^n+\left(3^n+2\right)2^k\sum_{l=0}^k\binom{k}{l}\frac{n!}{(n-k)!}\\
&\leq 3^n+\left(3^n+2\right)2^{2k}n^k,
\end{align*}
yielding the desired upper bound of 1.

{\em Step 2: Lower bound.} The lower bound is trivial for $a\leq 1/2$, so let $a>1/2$. Then $1-2a<0$, and so $|(1-2a)n-l|\geq|1-2a|n=(2a-1)n$ for all $0\leq l\leq n$. By Lemma \ref{lem:ternary-increments}, there exists therefore a constant $\delta>0$  such that for all sufficiently large $n$,
\[
|\Delta_{k,a}(I_{n,j})|\geq a^{n-l(j)}(2a-1)^{l(j)}\delta n^k, \qquad \mbox{for all $0\leq j<3^n$}.
\]
(Here we have included the factors $a^{-k}$ and $(2a-1)^{-k}$ in the constant $\delta$.) We may further assume $n$ is so large that $\delta n^k\geq 1$.
Since the oscillation of $M_{k,a}$ over $I_{n,j}$ is at least $|\Delta_{k,a}(I_{n,j})|$, the number of $3^{-n}$-mesh squares that intersect the graph of $M_{k,a}$ over $I_{n,j}$ is at least
\[
3^n |\Delta_{k,a}(I_{n,j})|\geq 3^n a^{n-l(j)}(2a-1)^{l(j)}.
\]
Summing over $j=0,1,\dots,3^n-1$ and noting once more that for each $l\in\{0,1,\dots,n\}$ there are precisely $\binom{n}{l}2^{n-l}$ values of $j$ with $l(j)=l$, we get
\begin{equation*}
N_{3^{-n}}\left(G_{k,a}\right)\geq 3^n\sum_{l=0}^n\binom{n}{l}2^{n-l}a^{n-l}(2a-1)^l
=3^n\big(2a+(2a-1)\big)^n=3^n(4a-1)^n
\end{equation*}
for all sufficiently large $n$, yielding the desired lower bound.
\end{proof}

\section{Finite derivatives} \label{sec:differentiability}

In this section we consider the set of points where $M_{k,a}$ has a finite derivative. Throughout, $P_k(n,l)$ denotes the polynomial from Lemma \ref{lem:ternary-increments}.

\begin{lemma} \label{lem:no-polynomial-limit}
Assume $a\neq 1/2$. Let $x\in(0,1)$ and put $l_n:=l_n(x)$. Then, for $k\geq 1$, the sequence $(P_k(n,l_n))_n$ cannot have a finite limit as $n\to\infty$.
\end{lemma}

\begin{proof}
Suppose, to the contrary, that $\lim_{n\to\infty} P_k(n,l_n)$ exists and is finite. Then
    \begin{equation*}
       \lim_{n\to\infty} \big[P_k(n+1,l_{n+1})-P_k(n,l_n)\big]= 0.
    \end{equation*}
		From the functional equations in Lemma \ref{lem:FE} (or, alternatively, from the precise expressions in the proof of Lemma \ref{lem:ternary-increments}), we can derive the difference equations
		\begin{align}
		P_k(n+1,l)-P_k(n,l)&=k(1-2a)P_{k-1}(n,l_n), \label{eq:difference-eq1}\\
		P_k(n+1,l+1)-P_k(n,l)&=-2kaP_{k-1}(n,l_n) \label{eq:difference-eq2}.
		\end{align}
		Since $l_{n+1}\in\{l_n,l_n+1\}$, it follows that $\lim_{n\to\infty} P_{k-1}(n,l_n) = 0$.
		By induction, we then conclude that $\lim_{n\to\infty} P_0(n,l_n) = 0$. However, by definition, $P_0(n,l_n) = 1$ and we have a contradiction. 
\end{proof}

\begin{lemma} \label{lem:polynomial-growth}
Suppose
\[
\liminf_{n\to\infty}\left|\frac{l_n(x)}{n}-(1-2a)\right|>0.
\]
Then there exists $\delta>0$ such that $|P_k(n,l_n(x))|\geq\delta n^k$ for all sufficiently large $n$. In particular, this holds when $a>1/2$.
\end{lemma}

\begin{proof}
This follows immediately from Lemma \ref{lem:ternary-increments}.
\end{proof}

Recall that $I_n(x)$ denotes the ternary interval $I_{n,j}$ containing $x$.
Throughout this section, we use the following well-known fact: If $M_{k,a}'(x)=d\in\RR$, then
\[
\frac{\Delta_{k,a}(I_n(x))}{|I_n(x)|}=3^n \Delta_{k,a}(I_n(x))\to d \qquad\mbox{as $n\to\infty$},
\]

Recall the definition of $\phi(a)$ from \eqref{eq:phi}. It is not difficult to verify that 
\begin{gather}
0<a<\frac13 \qquad \Longrightarrow \qquad \frac13<\phi(a)<1-2a, \label{eq:where-is-phi-1}\\
\frac13<a<\frac12 \qquad \Longrightarrow \qquad \frac13>\phi(a)>1-2a. \label{eq:where-is-phi-2}
\end{gather}

We first collect a few special cases.

\begin{proposition} \mbox{\ } \label{prop:special-cases}
\begin{enumerate}[(a)]
\item If $a\geq 2/3$, then $M_{k,a}$ does not have a finite derivative at any point.
\item If $a= 1/3$ and $k\geq 1$, then $M_{k,a}$ does not have a finite derivative at any point.
\end{enumerate}
\end{proposition}

\begin{proof}
Let $x\in(0,1)$ and write $l_n:=l_n(x)$. 

(a) Suppose $a\geq 2/3$. Then $3^n a^{n-l_n}|1-2a|^{l_n}\geq 1$ for all $n$, and by Lemma \ref{lem:polynomial-growth}, $|P_k(n,l_n)|\to\infty$. So by Lemma \ref{lem:ternary-increments}, $3^n\Delta_{k,a}(I_n(x))$ does not have a finite limit for any $x$. Hence, $M_{k,a}$ does not have a finite derivative at any point.

(b) Suppose $a=1/3$ and $k\geq 1$. Then by Lemma \ref{lem:ternary-increments} we have that 
    \begin{equation*}
        3^n\Delta_{k,a}(I_{n,j}) = 3^{-2k}P_k(n,l_n),
    \end{equation*}
		but this does not have a finite limit by Lemma \ref{lem:no-polynomial-limit}. Hence, $M_{k,a}'(x)$ does not exist as a finite number.
\end{proof}

\begin{remark}
{\rm
For $k=1$, the conclusion of Proposition \ref{prop:special-cases} (b) was proved in \cite{Dalaklis}. We observe that this result is quite different from the situation for $k=0$, as $M_{0,1/3}$ is the identity function. 
}
\end{remark}

The case $a=1/2$ is also special:

\begin{proposition} \label{prop:one-half}
Let $a=1/2$ and $x\in(0,1)$. Then $M_{k,a}$ has a finite derivative at $x$ if and only if there is an $n\in\NN$ such that $l_n(x)=k+1$ and $3^n x\not\in\mathbb{Z}$. In that case, $M_{k,a}'(x)=0$.
\end{proposition}

\begin{proof}
Let $x=(0.x_1x_2\dots)_3$, and first assume the stated condition. It follows from the functional equation (Lemma \ref{lem:FE}) that the graph of $M_{k,1/2}$ is flat on $I_n(x)$, and since $x$ lies in the interior of this interval, $M_{k,1/2}'(x)=0$.

Next, suppose $l:=\sup_n l_n(x)\leq k$. Then Lemma \ref{lem:difference-formulas} implies that
\[
3^n|\Delta_{k,1/2}(I_n(x))|=\frac{k!}{(k-l)!}\cdot \frac{(n-l)!}{(n-k)!}\left(\frac32\right)^n 2^{k+l}\to\infty,
\]
so $M_{k,1/2}'(x)$ does not exist as a finite real number.

Finally, if $l_n(x)=k+1$ and $3^n x\in\mathbb{Z}$, we can assume, by choosing $n$ minimal with this property, that $x_n=1$. Put $\bar{x}:=1-x$. Then the terminating ternary expansion of $\bar{x}$ contains exactly $k$ $1$'s, because we can alternatively write $x=(0.x_1\dots x_{n-1}02^\f)_3$. So by the previous case, $M_{k,1/2}$ does not have a finite derivative at $\bar{x}$. But then, by symmetry (Lemma \ref{lem:symmetry}), it does not have one at $x$ either.
\end{proof}



\begin{corollary} \label{cor:one-half-diff-ae}
For each $k\geq 1$, $M_{k,1/2}$ is differentiable almost everywhere.
\end{corollary}

\begin{proof}
By Proposition \ref{prop:one-half}, $M_{k,1/2}'(x)=0$ whenever the ternary expansion of $x$ contains infinitely many $1$'s. This is the case for Lebesgue-almost all $x\in(0,1)$.
\end{proof}

We now get to the main case: $0<a<2/3$ and $a\not\in\{1/3,1/2\}$. We divide this case in two parts. For both, we introduce the constant
\begin{equation} \label{eq:C0}
C_0:=C_0(a):=\frac{1}{\log a-\log|1-2a|}.
\end{equation}
Note that $C_0<0$ for $a<1/3$, whereas $C_0>0$ for $a>1/3$. Furthermore, $C_0(a)$ tends to $-\infty$ as $a\nearrow 1/3$, to $+\infty$ as $a\searrow 1/3$, and to $0$ as $a\to 1/2$.

The next two propositions show that the only possible finite value of $M_{k,a}'$ is $0$, and characterize exactly the set of points where this happens.

\begin{proposition} \label{prop:large-a-derivative}
Let $1/3<a<2/3$ and $a\neq 1/2$. For $x\in(0,1)$ and $n\in\NN$, set 
\[
r_n: = l_n(x)-n\phi(a). 
\]
Then $M_{k,a}$ has a finite derivative at $x$ if and only if 
\begin{equation} \label{eq:overshoot}
r_n-kC_0\log n \to \infty \qquad\mbox{as $n\to\infty$}, 
\end{equation}
in which case $M_{k,a}'(x)=0$. 
\end{proposition}

\begin{proof}
First assume \eqref{eq:overshoot}. Let $y\neq x$ be arbitrary and choose $n\in\NN$ so that $3^{-(n+1)}<|x-y|<3^{-n}$. Let $j$ be the integer such that $x\in I_{n,j}$; then $y\in I_{n,j-1}\cup I_{n,j}\cup I_{n,j+1}$. Since $l(j-1)=l(j)\pm 1$ and $l(j+1)=l(j)\pm 1$, Lemma \ref{lem:oscillation-bounds} and the triangle inequality yield that 
    \begin{equation}\label{ineq:DiffQuo}
        \left\vert\frac{M_{k,a}(x)-M_{k,a}(y)}{x-y}\right\vert \leq 3^{n+1}(2B_k)a^{n-l_n(x)}\vert1-2a\vert^{l_n(x)}n^k,
    \end{equation}
    where $B_k:=C_k a/|1-2a|$. The definition of $\phi(a)$ is equivalent to $3a^{1-\phi(a)}|1-2a|^{\phi(a)}=1$, and so
    \begin{equation}\label{eq:rn}
        3^na^{n-l_n(x)}\vert1-2a\vert^{l_n(x)}n^k = \left(\frac{|1-2a|}{a}\right)^{r_n}n^k=e^{-r_n/C_0}n^k\to 0,
    \end{equation}
		as can be seen by taking logarithms. Hence, $M_{k,a}'(x)=0$.

Conversely, suppose $M_{k,a}'(x)=d\in\RR$. Then certainly $3^n\Delta_{k,a}(I_n(x))\to d$. Write $l_n:=l_n(x)$ for short. We consider two cases.

\medskip
{\em Case 1.} Suppose $\liminf\big|\frac{l_n}{n}-(1-2a)\big|>0$. Then by Lemma \ref{lem:polynomial-growth} there exists $\delta>0$ such that for all large enough $n$, $|P_k(n,l_n)|\geq\delta n^k$. It follows that
\[
\lim_{n\to\infty} \frac{|P_k(n+1,l_{n+1})|}{|P_k(n,l_n)|}=1,
\]
since the difference equations \eqref{eq:difference-eq1} and \eqref{eq:difference-eq2} imply that $P_k(n+1,l_{n+1})-P_k(n,l_n)=o(n^{k-1})$. Since $|1-2a|<1/3$ for $1/3<a<2/3$, we can choose $\eps>0$ so small that 
\[
3|1-2a|(1+\eps)<1<3a(1-\eps), 
\]
and then choose $n_0$ so large that 
\[
1-\eps<\frac{|P_k(n+1,l_{n+1})|}{|P_k(n,l_n)|}<1+\eps \qquad\forall n\geq n_0.
\]
Now
\[
\frac{3^{n+1}|\Delta_{k,a}(I_{n+1}(x))|}{3^n|\Delta_{k,a}(I_n(x))|}=3a^{1-(l_{n+1}-l_n)}|1-2a|^{l_{n+1}-l_n}\cdot \frac{|P_k(n+1,l_{n+1})|}{|P_k(n,l_n)|},
\]
and since $3a^{1-(l_{n+1}-l_n)}|1-2a|^{l_{n+1}-l_n}=3a$ or $3|1-2a|$, it follows that the above ratio cannot converge to $1$. Hence, $M_{k,a}'(x)=d=0$. But then, for all large enough $n$,
\[
\delta \left(\frac{|1-2a|}{a}\right)^{r_n}n^k=3^n a^{n-l_n}|1-2a|^{l_n}\delta n^k\leq 3^n\Delta_{k,a}(I_n(x))\to 0,
\]
and taking logarithms yields \eqref{eq:overshoot}.

\medskip
{\em Case 2.} Suppose alternatively that $\lim_{n\to\infty}\frac{l_n}{n}=1-2a$. This is of course only possible if $a<1/2$.
Note by \eqref{eq:where-is-phi-2} that $1-2a<\phi(a)$, and since $a>|1-2a|$, it follows that
\[
3^n a^{n-l_n}|1-2a|^{l_n}\to \infty.
\]
Further, by Lemma \ref{lem:no-polynomial-limit}, $P_k(n,l_n)$ cannot tend to $0$. Therefore, $3^n|\Delta_{k,a}(I_n(x))|$ cannot have a finite limit as $n\to\infty$, contradicting that $M_{k,a}'(x)=d\in\RR$.
\end{proof}

\begin{proposition} \label{prop:small-a-derivative}
Let $0<a<1/3$. let $C_0$ be defined by \eqref{eq:C0}. For $x\in(0,1)$ and $n\in\NN$, define $r_n$ as in Proposition \ref{prop:large-a-derivative}. Then $M_{k,a}$ has a finite derivative at $x$ if and only if 
\[
r_n+k|C_0|\log n \to -\infty \qquad \mbox{as $n\to\infty$}, 
\]
in which case $M_{k,a}'(x)=0$. 
\end{proposition}

\begin{proof}
The proof is analogous to that of Proposition \ref{prop:large-a-derivative}, and is omitted.
\end{proof}

Together, Propositions \ref{prop:special-cases}, \ref{prop:large-a-derivative} and \ref{prop:small-a-derivative} imply Proposition \ref{prop:only-zero} and Theorem \ref{thm:differentiability-summary}.


\begin{corollary} \label{cor:derivativ-zero-rough}
Let $x\in(0,1)$ and define
\[
\lambda_*(x):=\liminf_{n\to\infty}\frac{l_n(x)}{n}, \qquad \lambda^*(x):=\limsup_{n\to\infty}\frac{l_n(x)}{n}.
\]
\begin{enumerate}[(a)]
\item If $0<a<1/3$, then 
\[
\{x:\lambda^*(x)<\phi(a)\}\subset D_{k,a}^0\subset \{x:\lambda^*(x)\leq\phi(a)\}.
\]
\item If $1/3<a<2/3$, then
\[
\{x:\lambda_*(x)>\phi(a)\}\subset D_{k,a}^0\subset \{x:\lambda_*(x)\geq\phi(a)\}.
\]
\end{enumerate}
\end{corollary}

\begin{proof}
Let $0<a<1/3$ and $\lambda^{*}<\phi(a)$. Then $r_n=l_n(x)-\phi(a)n\leq -\delta n$ for some fixed $\delta>0$ and all large enough $n$, so certainly $r_n+k|C_0|\log n\to-\infty$. On the other hand, if $\lambda^{*}>\phi(a)$, then there is a subsequence $(n_i)$ and a $\delta>0$ such that $r_{n_i}\geq \delta n_i$ for all $i$. Hence (a) follows from Proposition \ref{prop:small-a-derivative}. The proof of (b) is analogous.
\end{proof}

\begin{proof}[Proof of Theorem \ref{thm:M_k-differentiability}]
Statement ({\em a}) is the same as Theorem \ref{thm:differentiability-summary} (c). Statements ({\em b}) and ({\em c}) follow from Corollary \ref{cor:derivativ-zero-rough} in the same way as in \cite[Section 4]{Allaart}.
\end{proof}

\begin{theorem} \label{thm:big-set-difference}
For $a\in(0,2/3)\backslash\big\{1/3\big\}$ and any $k\geq 1$ there exist uncountably many points $x$ such that $M'_{k,a}(x) = 0$ but $M_{k+1,a}$ does not have a finite derivative at $x$. Moreover, the set of such points has positive Hausdorff dimension.
\end{theorem}

\begin{proof}
We first consider the case $a=1/2$. For $x\in(0,1)$, let $l(x):=\sup_n l_n(x)$. If $x$ is not a ternary rational point, then the condition of Proposition \ref{prop:one-half} simplifies to $l(x)\geq k+1$. Thus, the set difference $D_{k,1/2}^0\backslash D_{k+1,1/2}^0$ is, up to a countable set of endpoints, equal to the set of points $x$ such that $l(x)=k+1$. This set has Hausdorff dimension $\log_3 2$, as it is a countable union of affine copies of the ternary Cantor set.

Next, we assume $0<a<1/3$. Recall that $C_0<0$ in this case.
Fix an arbitrary $\delta\in(0,1)$. We construct a point $x$ with the desired property by prescribing the positions of the 1's in its ternary expansion. For $n\in\NN$, let
\[
b_n:=\left\lfloor\frac{n}{\phi(a)}-\frac{k+\delta}{\log 3a}\log n\right\rfloor,
\]
and put $x:=\sum_{j=1}^\infty d_j 3^{-j}$, where
\begin{equation} \label{eq:special-ternary-expansion}
d_j=\begin{cases}
1 & \mbox{if $j=b_n$ for some $n\in\NN$},\\
0\ \mbox{or}\ 2\ \mbox{arbitrarily} & \mbox{otherwise}.
\end{cases}
\end{equation}
We first verify that $M_{k,a}'(x)=0$. For all large enough $n$,
\begin{align*}
l_{b_n}(x)&-\phi(a)b_n+k|C_0|\log b_n=n-\phi(a)b_n+k|C_0|\log b_n\\
&\leq n-\phi(a)\left(\frac{n}{\phi(a)}-\frac{k+\delta}{\log 3a}\log n-1\right)+k|C_0|\log\left(\frac{n}{\phi(a)}-\frac{(k+\delta)}{\log 3a}\log n\right)\\
&\leq \phi(a)\left(\frac{k+\delta}{\log 3a}\log n+1\right)+k|C_0|\log\left(\frac{2n}{\phi(a)}\right)\\
&=-(k+\delta)|C_0|\log n+\phi(a)+k|C_0|\{\log n+\log 2-\log\phi(a)\}\\
&=-\delta |C_0|\log n+\phi(a)+k|C_0|\{\log 2-\log\phi(a)\}\to -\infty.
\end{align*}
Here the second-to-last step follows from the definitions of $\phi(a)$ and $C_0$.
Now if $j\in\NN$ is arbitrary, choose $n$ so that $b_{n-1}<j\leq b_{n}$. Then
\begin{align*}
l_j(x)-\phi(a)j+k|C_0|\log j &\leq l_{b_{n}}(x)-\phi(a)b_{n-1}+k|C_0|\log b_{n}\\
&\leq l_{b_{n}}(x)-\phi(a)b_{n}+k|C_0|\log b_{n}+\phi(a)\{b_{n}-b_{n-1}\}\\
&\to -\infty,
\end{align*}
because $b_n-b_{n-1}$ is bounded in $n$. Thus, we have verified the condition of Proposition \ref{prop:small-a-derivative}. A similar calculation, which we omit, shows that $l_j(x)-\phi(a)j+(k+1)|C_0|\log j\to +\infty$. Hence, $M_{k,a}'(x)=0$ but $M_{k+1,a}'(x)$ does not exist as a finite real number.

It is not difficult to show that the set of points $x$ whose ternary expansion is given by \eqref{eq:special-ternary-expansion} has Hausdorff dimension $(1-\phi(a))\log_3 2>0$. Hence, the set of points for which $M_{k,a}'(x)=0$ but $M_{k+1,a}$ does not have a finite derivative at $x$ has at least this dimension. We leave the details to the interested reader.

Finally, the case where $1/3<a<2/3, a\neq 1/2$ can be dealt with in a similar way; we omit the details.
\end{proof}

\begin{proof}[Proof of Corollary \ref{cor:strictly-descending-sets}]
That the differences $D_{k,a}^0\backslash D_{k+1,a}^0$ are descending in $k$ for each $a\in(0,2/3)\backslash \{1/3\}$ follows immediately from Propositions \ref{prop:one-half}, \ref{prop:large-a-derivative} and \ref{prop:small-a-derivative}. That they have positive Hausdorff dimension is the result of Theorem \ref{thm:big-set-difference}.
\end{proof}

\section{Infinite derivatives: $a=1/2$} \label{sec:infinite-derivative-one-half}

In this section we describe the infinite derivatives of $M_{k,1/2}$. Since the parameter $a$ will be fixed at $1/2$, we drop it from the notation and write simply $M_k$ instead of $M_{k,1/2}$.
 
For $x\in[0,1]$, $n\in\NN$ and $d\in\{0,2\}$, let 
\begin{equation} \label{eq:run-length}
\rho_n(x;d):=\max\{m\geq 0: x_{n+1}=\dots=x_{n+m}=d\}
\end{equation}
be the run length of consecutive digits $d$ starting at the $(n+1)$th ternary digit of $x$, where we make the convention that the maximum of the empty set is zero.

\begin{theorem} \label{thm:infinite-derivatives-half}
Fix $k\in\NN$. Let $x\in(0,1)$, and let $l:=l(x):=\sup_n l_n(x)$ be the total number of $1$'s in the ternary expansion of $x$. 
\begin{enumerate}[(a)]
\item If $l>k$, then $M_k$ does not have an infinite derivative at $x$.
\item If $l=k$, then $M_k$ has an infinite derivative at $x$ if and only if Eidswick's condition holds (see \cite{Eidswick}): 
\[
cn-\rho_n(x;d)\to\f, \qquad d=0,2,
\]
where $c:=\log_2 3-1$. If this is the case, $M_k'(x)=(-1)^k\cdot \f$.
\item If $l<k$, then sufficient for $M_k$ to have an infinite derivative at $x$ is that there exists $N\in\NN$ such that
\[
\rho_n(x;d)\leq \log_2 n-\log_2(k-l)-2 \qquad \forall n\geq N, \forall d\in\{0,2\};
\]
and necessary is that there exists $N\in\NN$ such that
\[
\rho_n(x;d)\leq \log_2 n-\log_2(k-l) \qquad \forall n\geq N, \forall d\in\{0,2\}.
\]
Moreover, if $M_k$ has an infinite derivative at $x$, then $M_k'(x)=(-1)^l\cdot\f$.
\end{enumerate}
\end{theorem}

(An exact characterization would involve the entire sequence $(\rho_n(x;d))$ of run lengths in a complicated way, and even though one can be written down, it is not enlightening and therefore we omit it here. On the other hand, the conditions given above are simple to understand, and are sharp to within a small constant.)

\begin{corollary} \label{cor:full-dimension}
For all $l\leq k$, let $F_{k,l}^\f$ be the set
\[
F_{k,l}^\f:=\{x\in[0,1]: l(x)=l\ \mbox{and}\ M_k'(x)=\pm\f\}.
\]
Then $F_{k,l}^\f$ has Hausdorff dimension $s:=\log_3 2$. Furthermore, letting $\mathcal{H}^s(E)$ denote the $s$-dimensional Hausdorff measure of a set $E$, we have that $\mathcal{H}^s(F_{k,l}^\f)=0$ for $l<k$, whereas $\mathcal{H}^s(F_{k,k}^\f)>0$. As a result,
\[
\mathcal{H}^s(D_{k,1/2}^{+\f})>0 \quad\mbox{and}\quad \mathcal{H}^s(D_{k,1/2}^{-\f})=0 \quad\mbox{when $k$ is even},
\]
and
\[
\mathcal{H}^s(D_{k,1/2}^{-\f})=0 \quad\mbox{and}\quad \mathcal{H}^s(D_{k,1/2}^{+\f})>0 \quad\mbox{when $k$ is odd}.
\]
\end{corollary}

\begin{proof}[Proof of Theorem \ref{thm:infinite-derivatives-half}]
({\em a}) and ({\em b}): If $x_1\dots x_n$ contains exactly $k$ 1's, including at the $n$th digit, then the graph of $M_k$ over the interval $[0.x_1\dots x_n,0.x_1\dots x_n+3^{-n}]$ is an affine copy of the Cantor devil's staircase $F_{1/2}$. This follows by setting $a=1/2$ in the functional equations \eqref{eq:M1-FE} and \eqref{eq:Mk-FE} and iterating: eventually, all higher order partial derivatives will vanish. Thus, ({\em b}) follows from Eidswick's theorem \cite{Eidswick}. Moreover, if $(x_i)$ contains at least $k+1$ 1's, then $x$ lies in one of the flat segments of a devil's staircase (possibly at an endpoint), so the right derivative of $M_k$ vanishes at $x$, and this gives ({\em a}).

({\em c}) We assume $l$ is even and explore when $M_k'(x)=+\f$. Note that $M_k'(x)$ cannot equal $-\f$ because for all sufficiently large $n$ (namely, all $n$ beyond the position of the last 1), the secant slope of $M_k$ over the interval $I_n(x)=[0.x_1\dots x_n,0.x_1\dots x_n+3^{-n}]$ is positive. Note further that $x$ cannot be a ternary rational point, because $M_k$ has local extrema at such points. Thus, $(x_i)$ does not end in $0^\f$ or $2^\f$, and hence $(x_i)$ contains both infinitely many 0's and infinitely many 2's.

We first consider whether $M_k$ has an infinite {\em right} derivative at $x$. We denote the right derivative by $(M_k)_+'$. Let $n_0$ be the position of the last 1 in $(x_i)$, and define the set $E:=\{n>n_0: x_n=0\}$. For each $n\in E$, put $z_n:=0.x_1\dots x_{n-1}20^\f$. Then $z_n>x$ and $3^{-n}<z_n-x<2\cdot 3^{-n}$. We certainly need that
\begin{equation} \label{eq:the-critical-sequence}
\frac{M_k(z_n)-M_k(x)}{z_n-x}\to\f,
\end{equation}
but this is also sufficient: For arbitrary $y>x$, let $n$ be the index such that $x_1\dots x_{n-1}=y_1\dots y_{n-1}$ and $x_n<y_n$. If $y$ is close enough to $x$, then $n>n_0$, and hence $x_n=0$; i.e. $n\in E$. Since $y_n=1$ or $2$, Lemmas \ref{prop:box-containment} and \ref{lem:difference-formulas} imply that $M_k(y)\geq M_k(z_n)$, recalling that $l$ is even. Thus, if \eqref{eq:the-critical-sequence} holds, then $M_k'(x)=+\f$.

Take $n>n_0$ and assume also that $n\geq 2k$. Let $r:=\rho_n(x;2)$ be the run length of consecutive $2$'s starting with $x_{n+1}$; see \eqref{eq:run-length}. Since the ternary expansion of $x$ contains no more 1's after the $n$th digit, $x$ lies between $0.x_1\dots x_{n-1}02^r 0^\f$ and $0.x_1\dots x_{n-1}02^r 10^\f$, and so
\begin{equation} \label{eq:Mx-bounds}
M_k(x_1\dots x_{n-1}02^r0^\f)\leq M_k(x)\leq M_k(x_1\dots x_{n-1}02^r10^\f)
\end{equation}
by Lemma \ref{prop:box-containment}, since $l$ is even. Hence,
\[
\frac{M_k(z_n)-M_k(x)}{z_n-x}<3^n\big\{M_k(x_1\dots x_{n-1}20^\f)-M_k(x_1\dots x_{n-1}02^r0^\f)\big\}.
\]
To calculate the difference in square brackets, we first compute the corresponding difference for $F_a$ and then differentiate $k$ times. Set $\mathbf{i}:=x_1\dots x_{n-1}$. Note that
\[
F_a(\mathbf{i}20^\f)-F_a(\mathbf{i}0^\f)=a^{n-l-1}(1-2a)^l(1-a) \qquad\mbox{for all $n>n_0$},
\]
and so
\begin{align*}
F_a(\mathbf{i}20^\f)-F_a(\mathbf{i}02^r0^\f)&=a^{n-l-1}(1-2a)^l(1-a)-\sum_{j=n+1}^{n+r}a^{j-l-1}(1-2a)^l(1-a)\\
&=a^{n-l-1}(1-2a)^l\big(1-a-a(1-a^r)\big)\\
&=a^{n-l-1}(1-2a)^{l+1}+a^{n-l+r}(1-2a)^l.
\end{align*}
Differentiating $k$ times and substituting $a=1/2$ gives, after some calculations,
\begin{align}
\begin{split}
M_k(\mathbf{i}&20^\f)-M_k(\mathbf{i}02^r0^\f)\\
&=\frac{k!}{(k-l)!}(-2)^l\left(1/2\right)^{n-k+r}\left[\frac{(n-l+r)!}{(n-k+r)!}-2^{r+1}(k-l)\frac{(n-l-1)!}{(n-k)!}\right].
\end{split}
\label{eq:lower-critical-difference}
\end{align}
Suppose $r>\log_2 n-\log_2(k-l)$. Then $2^{r+1}(k-l)>2n$, and we claim this implies that the difference in square brackets is negative for all sufficiently large $n$. First we estimate
\begin{equation} \label{eq:first-estimate}
\frac{(n-l+r)!}{(n-k+r)!}-2^{r+1}(k-l)\frac{(n-l-1)!}{(n-k)!}<(n+r)^{k-l}-2n(n-k+1)^{k-l-1}.
\end{equation}
Choose $\eps>0$ so small that
\[
(1+\eps)^{k-l}<2(1-\eps)^{k-l-1}.
\]
If $r\geq \eps n$, then for large enough $n$ the factor $2^{r+1}$ in \eqref{eq:first-estimate} will dominate and make the left-hand side negative. So assume $r<\eps n$. Then, for $n$ so large that $n-k+1>(1-\eps)n$,
\begin{equation*}
(n+r)^{k-l}<(1+\eps)^{k-l}n^{k-l}<2(1-\eps)^{k-l-1}n^{k-l}
<2n(n-k+1)^{k-l-1},
\end{equation*}
as required. Thus, an infinite derivative at $x$ is impossible. This proves the necessity statement in ({\em c}).

For the sufficiency statement, observe that $M_k(x)\leq M_k(\mathbf{i}02^r10^\f)$, by \eqref{eq:Mx-bounds}. Lemma \ref{lem:difference-formulas} gives
\[
M_k(\mathbf{i}02^r10^\f)-M_k(\mathbf{i}02^r0^\f)=(-2)^l\frac{k!}{(k-l)!}\cdot\frac{(n+r+1-l)!}{(n+r+1-k)!}\left(1/2\right)^{n+r+1-k}.
\]
Subtracting this from \eqref{eq:lower-critical-difference} yields, after some manipulation,
\[
M_k(\mathbf{i}20^\f)-M_k(\mathbf{i}02^r10^\f)=(-2)^l\frac{k!}{(k-l)!}\left(1/2\right)^{n-k+r+1}A_k(n,l,r),
\]
where
\[
A_k(n,l,r):=\frac{(n-l+r)!}{(n-k+r+1)!}(n+r+1+l-2k)-2^{r+2}(k-l)\frac{(n-l-1)!}{(n-k)!}.
\]
Since $l$ is even, it follows that $M_k(\mathbf{i}20^\f)-M_k(\mathbf{i}02^r10^\f)>0$ if and only if $A_k(n,l,r)>0$. We can estimate $A_k(n,l,r)$ by
\[
A_k(n,l,r)>(n-k+r+2)^{k-l-1}(n+r+1+l-2k)-2^{r+2}(k-l)(n-l-1)^{k-l-1}.
\]
Assume $r\leq \log_2 n-\log_2(k-l)-2$. Then $2^{r+2}(k-l)\leq n$. If, in fact, $r\leq 2k-l$, then $2^{r+2}(k-l)(n-l-1)^{k-l-1}=O(n^{k-l-1})$ whereas $(n-k+r+2)^{k-l-1}(n+r+1+l-2k)$ is of order $n^{k-l}$. Otherwise, we have
\begin{align*}
(n-k+r+2)^{k-l-1}(n+r+1+l-2k)&>(n+k-l+2)^{k-l-1}(n+1)\\
&>n(n-l-1)^{k-l-1}\\
&\geq 2^{r+2}(k-l)(n-l-1)^{k-l-1}.
\end{align*}
Hence, in either case, $A_k(n,l,r)>0$,
at least for large enough $n$. Since $A_k(n,l,r)$ is an integer (recall $l<k$ and $n\geq 2k$), it follows that
\[
M_k(\mathbf{i}20^\f)-M_k(\mathbf{i}02^r10^\f)\geq 2^l\frac{k!}{(k-l)!}\left(1/2\right)^{n-k+r+1}.
\]
But this implies
\begin{align*}
\lim_{n\to\f, n\in E}\frac{M_k(z_n)-M_k(x)}{z_n-x}&\geq \lim_{n\to\f}\frac{M_k(\mathbf{i}20^\f)-M_k(\mathbf{i}02^r10^\f)}{2\cdot 3^{-n}}\\
&\geq\lim_{n\to\f} 3^n2^{l-1}\frac{k!}{(k-l)!}\left(1/2\right)^{n-k+\rho_n(x;2)+1}\\
&\geq \lim_{n\to\f} 3^n2^{l-1}\frac{k!}{(k-l)!}\left(1/2\right)^{n+\log_2 n}=+\f.
\end{align*}
Hence, $(M_k)_+'(x)=+\f$.

If $l$ is odd, all the signs are reversed and we get $(M_k)_+'(x)=-\f$.

Finally, the symmetry property of Lemma \ref{lem:symmetry}, applied to the full unit interval (that is, $M_k(1-x)=-M_k(x)$ for all $x\in[0,1]$) yields that $(M_k)_-'(x)=(M_k)_+'(1-x)$, so the conditions for $M_k$ to have an infinite {\em left} derivative at $x$ are the same as those for an infinite {\em right} derivative at $1-x$. Since the expansion of $1-x$ is obtained from that of $x$ by changing $0$'s to $2$'s and vice versa, this means we obtain the same conditions for the run lengths of $0$'s.
\end{proof}

\begin{proof}[Proof of Corollary \ref{cor:full-dimension}]
Let $l\leq k$. By Theorem \ref{thm:infinite-derivatives-half}, $M_k$ has an infinite derivative at $x$ whenever $l(x)\leq k$ and the run lengths $\rho_n(x;0)$ and $\rho_n(x;d)$ are bounded. Let $K$ be the set of points $x$ for which this is the case, and for $m\geq 2$, let $K_m$ be the set of points $x=(0.x_1x_2\dots)_3$ such that $x_i\in\{0,2\}$ for all $i$, $x_{jm-1}=0$ for all $j\in\NN$, and $x_{jm}=2$ for all $j\in\NN$. Clearly, $K\supset K_m$ for all $m$. Note that $K_m$ is a Cantor set of $2^{m-2}$ parts with contraction ratio $3^{-m}$ (because the last two digits of each successive block of $m$ digits are fixed and the others can be freely chosen from $\{0,2\}$). Thus,
\[
\dim_H K_m=\frac{(m-2)\log 2}{m\log 3},
\]
and it follows that
\[
\dim_H F_{k,l}^\f\geq \dim_H K\geq\sup_m \dim_H K_m=\log_3 2=s.
\]
On the other hand, each $F_{k,l}^\f$ is contained in a countable union of affine copies of the ternary Cantor set $C$, hence $\dim_H F_{k,l}^\f\leq s$.

Next, observe that $\mathcal{H}^s(C)=1$ (see \cite[p.~53]{Falconer}), so $\mathcal{H}^s|_C$ is a probability measure on $C$. In fact, it is the measure that makes the ternary digits of $x$ independent and identically distributed, with digits $0$ and $2$ each having probability $1/2$. Let first $l<k$. By a refinement of the Borel-Cantelli lemma (see \cite[Example 6.5]{Billingsley}), the set of $x\in C$ such that $\rho_n(x;2)>\log_2 n-\log_2(k-l)$ for infinitely many $n$ has full $\mathcal{H}^s$-measure in $C$ because the infinite series
\[
\sum_{n=1}^\infty 2^{-(\log_2 n-\log_2(k-l))}=(k-l)\sum_{n=1}^\infty \frac{1}{n}
\]
diverges, so it follows from the necessary condition in Theorem \ref{thm:infinite-derivatives-half} ({\em c}) that $\mathcal{H}^s(F_{k,l}^\f)=0$.

If $l=k$, then $F_{k,k}^\f$ contains the set 
\[
\{x\in [0,1]: l(x)=k\ \mbox{and}\ \rho_n(x;d)<(c/2)n\ \mbox{for all sufficiently large $n$ and $d\in\{0,2\}$}\},
\]
and this set has full $\mathcal{H}^s$ measure in $\{x\in C: l(x)=k\}$ by the Borel-Cantelli lemma since $\sum_{n=1}^\f 2^{-(c/2)n}<\f$. Hence, $\mathcal{H}^s(F_{k,k}^\f)>0$.
The statements about $\mathcal{H}^s(D_{k,1/2}^{+\f})$ and $\mathcal{H}^s(D_{k,1/2}^{-\f})$ now follow since Theorem \ref{thm:infinite-derivatives-half} shows that $M_k'(x)=(-1)^l\cdot\f$ when $M_k$ has an infinite derivative at $x$.
\end{proof}

\section{Infinite derivatives: $a<1/2$} \label{sec:infinite-derivative-small-a}

In this section, we investigate the infinite derivatives of $M_{k,a}$ when $a<1/2$. While we are unable to state exact ``if and only if" conditions, the first theorem below gives a partial characterization in terms of the (upper and lower) frequency of $1$'s in the ternary expansion of $x$; this will be enough to prove Theorem \ref{thm:D-inf-dimensions}.
Recall the definition of $\phi(a)$ from \eqref{eq:phi}, as well as the inequalities \eqref{eq:where-is-phi-1} and \eqref{eq:where-is-phi-2}. 

\begin{theorem} \label{thm:infinite-derivative-detail}
Let $k\in\NN$ and $x\in(0,1)$, and let
\[
\lambda_*:=\liminf_{n\to\f}\frac{l_n(x)}{n}, \qquad \lambda^*:=\limsup_{n\to\f}\frac{l_n(x)}{n}.
\]
\begin{enumerate}[(a)]
\item Suppose $0<a<1/3$.
\begin{enumerate}[(i)]
\item If $\lambda_*>1-2a,$ then $M_{k,a}'(x)=(-1)^k\cdot\f$.
\item If 
$$\phi(a)<\lambda_*\leq\lambda^*<1-2a,$$ 
then $M_{k,a}'(x)=+\f$.
\item If $\lambda_*<\phi(a),$ then $M_{k,a}$ does not have an infinite derivative at $x$.
\end{enumerate}
\item Suppose $a=1/3$.
\begin{enumerate}[(i)]
\item If $\lambda^*<1/3,$ then $M_{k,a}'(x)=+\f$.
\item If $\lambda_*>1/3,$ then $M_{k,a}'(x)=(-1)^k\cdot\f$.
\end{enumerate}
\item Suppose $1/3<a<1/2$.
\begin{enumerate}[(i)]
\item If $\lambda^*<1-2a,$ then $M_{k,a}'(x)=+\f$.
\item If 
$$1-2a<\lambda_*\leq\lambda^*<\phi(a),$$ 
then $M_{k,a}'(x)=(-1)^k\cdot\f$.
\item If $\lambda^*>\phi(a),$ then $M_{k,a}$ does not have an infinite derivative at $x$.
\end{enumerate}
\item For all $0<a<1/2$, if 
$$\lambda_*<1-2a<\lambda^*,$$ 
then $M_{k,a}$ does not have an infinite derivative at $x$.
\end{enumerate}
\end{theorem}

We do not give a proof of this theorem here, as it is a direct consequence of Theorem \ref{thm:alternating-infinities} below.
This next theorem provides more details about the infinite derivatives in the boundary case of Theorem \ref{thm:infinite-derivative-detail}, when $l_n(x)/n\to 1-2a$. It reveals a surprising alternating pattern when $(1-2a)n-l_n(x)$ is of order $\sqrt{n}$. In order to state it, we first define a sequence $(q_k)$ of polynomials inductively by
\begin{align}
\begin{split}
q_1(t)&=t,\\
q_{k+1}(t)&=tq_k(t)-q_k'(t), \qquad k=1,2,\dots.
\end{split}
\label{eq:q-polynomials}
\end{align}

\begin{lemma} \label{lem:polynomial-roots}
For each $k\in\NN$, $q_k$ has $k$ distinct real zeros, distributed symmetrically about $0$. Furthermore, $q_k$ is even when $k$ is even, and $q_k$ is odd when $k$ is odd.
\end{lemma}

\begin{proof}
The second statement is immediate from the definition and induction. It follows that the graph of $q_k$ is symmetric about the $y$-axis when $k$ is even, and is symmetric about the origin when $k$ is odd. Thus, the set of zeros of $q_k$ is symmetric about $0$. Furthermore, $q_k(0)=0$ when $k$ is odd.

We now prove inductively that $q_k$ has $k$ distinct real zeros. This is clear for $k=1$. Suppose the statement is true for $k$, and consider $q_{k+1}(t)=tq_k(t)-q_k'(t)$. We first show that between any two consecutive zeros of $q_k$ there is a zero of $q_{k+1}$. Let $t_1<t_2<\dots<t_k$ be the zeros of $q_k$, and take $i\in\{1,2,\dots,k-1\}$. Since $q_k$ has no double zeros, its graph crosses the $t$-axis transversally (i.e. with a non-zero derivative). Thus, either $q_k'(t_i)<0$ and $q_k'(t_{i+1})>0$ or the other way around; without loss of generality assume the former. Then $q_{k+1}(t_i)=-q_k'(t_i)>0$ and $q_{k+1}(t_{i+1})=-q_k'(t_{i+1})<0$. Therefore, $q_{k+1}$ has a zero between $t_i$ and $t_{i+1}$.

Next, we claim that $q_{k+1}$ has a zero to the left of $t_1$ and another to the right of $t_k$. Assume first that $k$ is even. Then $q_k$ is even, so $q_k'(t_1)<0$ and $q_k'(t_k)>0$. Hence $q_{k+1}(t_1)>0$ and $q_{k+1}(t_k)<0$. Since $q_{k+1}$ is odd with positive leading coefficient, the claim follows. A similar argument applies when $k$ is odd.
\end{proof}

We list the first eight polynomials of the sequence $(q_k)$ here:
\begin{alignat*}{2}
q_1(t)&=t, & \qquad\qquad  q_5(t)&=t^5-10t^3+15t, \\
q_2(t)&=t^2-1, & \qquad\qquad  q_6(t)&=t^6-15t^4+45t^2-15, \\
q_3(t)&=t^3-3t, & \qquad\qquad  q_7(t)&=t^7-21t^5+105t^3-105t, \\
q_4(t)&=t^4-6t^2+3, & \qquad\qquad  q_8(t)&=t^8-28t^6+210t^4-420t^2+105.
\end{alignat*}

In the extremal cases of Theorem \ref{thm:alternating-infinities} below, we assume that
\begin{equation} \label{eq:not-going-to-zero}
3^n a^{n-l_n(x)}(1-2a)^{l_n(x)}\geq 1 \qquad\mbox{for all sufficiently large $n$}.
\end{equation}
Sufficient for \eqref{eq:not-going-to-zero} is that either (i) $a<1/3$ and $\lambda_*>\phi(a)$; (ii) $a=1/3$; or (iii) $1/3<a<1/2$ and $\lambda^*<\phi(a)$.

Theorem \ref{thm:alternating-pattern} in the Introduction is a special case of the following more general result.

\begin{theorem} \label{thm:alternating-infinities}
Let $0<a<1/2$. Let $t_1<t_2<\dots<t_k$ be the zeros of the polynomial $q_k$ defined by \eqref{eq:q-polynomials}, and put $\tilde{t}_i:=t_i\sqrt{2a(1-2a)}$ for $i=1,\dots,k$. Also define
\[
\delta_*:=\liminf_{n\to\f} \frac{(1-2a)n-l_n(x)}{\sqrt{n}}, \qquad \delta^*:=\limsup_{n\to\f} \frac{(1-2a)n-l_n(x)}{\sqrt{n}}.
\]
\begin{enumerate}[(a)]
\item If $\delta^*<\tilde{t}_1$ and \eqref{eq:not-going-to-zero} holds, then $M_{k,a}'(x)=(-1)^k\cdot\f$.
\item If 
$\tilde{t}_{i}<\delta_*\leq\delta^*<\tilde{t}_{i+1}$ 
for some $i\in\{1,2,\dots,k-1\}$, then 
\[
M_{k,a}'(x)=\begin{cases}
+\f & \mbox{if $k-i$ is even},\\
-\f & \mbox{if $k-i$ is odd}.
\end{cases}
\]
\item If $\delta_*>\tilde{t}_k$ and \eqref{eq:not-going-to-zero} holds, then $M_{k,a}'(x)=+\f$.
\item If $\delta^*<\tilde{t}_i<\delta_*$ for some $i\in\{1,2,\dots,k\}$, then $M_{k,a}$ does not have an infinite derivative at $x$.
\end{enumerate}
\end{theorem}

\begin{example} \label{ex:infinite-derivatives-M_2}
Let $k=2$. Note that $q_2(t)=t^2-1$, so $\tilde{t}_1=\sqrt{2a(1-2a)}=:c_0$ and $\tilde{t}_2=-c_0$.
Defining
\[
\gamma_*:=\liminf_{n\to\f} \frac{|(1-2a)n-l_n(x)|}{\sqrt{n}}, \qquad \gamma^*:=\limsup_{n\to\f} \frac{|(1-2a)n-l_n(x)|}{\sqrt{n}},
\]
we can state Theorem \ref{thm:alternating-infinities} for this case as follows:
\begin{enumerate}[(i)]
\item If $\gamma^*<c_0$,
then $M_{2,a}'(x)=-\f$.
\item If \eqref{eq:not-going-to-zero} holds and $\gamma_*>c_0$,
then $M_{2,a}'(x)=+\f$.
\item If $\gamma_*<c_0<\gamma^*$,
then $M_{2,a}$ does not have an infinite derivative at $x$.
\end{enumerate}
\end{example}

\begin{example} \label{ex:infinite-derivatives-M_3}
Take $k=3$. Since $q_3(t)=t^3-3t$, we have $\tilde{t}_1=-c_0'$, $\tilde{t}_2=0$ and $\tilde{t}_3=c_0'$, where $c_0':=\sqrt{6a(1-2a)}$. Thus, Theorem \ref{thm:alternating-infinities} implies:
\begin{enumerate}[(i)]
\item If either $\delta_*>c_0'$ and \eqref{eq:not-going-to-zero} holds, or $-c_0'<\delta_*\leq\delta^*<0$,
then $M_{3,a}'(x)=+\f$.
\item If either $\delta^*<-c_0'$ and \eqref{eq:not-going-to-zero} holds, or $0<\delta_*\leq\delta^*<c_0'$,
then $M_{3,a}'(x)=-\f$.
\end{enumerate}
\end{example}

We give one more example, because it illustrates a qualitative difference between $k\equiv 0 \pmod 4$ and $k\equiv 2 \pmod 4$.

\begin{example} \label{ex:infinite-derivatives-M_4}
Take $k=4$. The roots of $q_4(t)=t^4-6t+3$ are $\pm \sqrt{3-\sqrt{6}}$ and $\pm \sqrt{3+\sqrt{6}}$. Hence, setting
\[
c_1:=\sqrt{2a(1-2a)(3-\sqrt{6})}, \qquad c_2:=\sqrt{2a(1-2a)(3+\sqrt{6})},
\]
we have $\tilde{t}_1=-c_2, \tilde{t}_2=-c_1, \tilde{t}_3=c_1$ and $\tilde{t}_4=c_2$.
Define $\gamma_*$ and $\gamma^*$ as in Example \ref{ex:infinite-derivatives-M_2}. Theorem \ref{thm:alternating-infinities} implies:
\begin{enumerate}[(i)]
\item If either \eqref{eq:not-going-to-zero} holds and $\gamma_*>c_2$, or $\gamma^*<c_1$,
then $M_{4,a}'(x)=+\f$.
\item If 
\[
c_1<\delta_*\leq\delta^*<c_2 \qquad\mbox{or}\qquad -c_2<\delta_*\leq\delta^*<-c_1,
\]
then $M_{4,a}'(x)=-\f$.
\end{enumerate}
\end{example}

Before proving the theorem we need a few lemmas. Recall that the polynomials $P_k(n,l)$ were defined in Lemma \ref{lem:ternary-increments}.

\begin{lemma} \label{lem:P-recursion}
For all $k\in\NN$ we have the recursion
\begin{equation} \label{eq:P-recursion}
P_{k+1}(n,l)=a(1-2a)\frac{\partial P_k}{\partial a}(n,l)+\{(1-2a)n-l-(1-4a)k\}P_k(n,l).
\end{equation}
\end{lemma}

\begin{proof}
Since $n$ and $l$ are fixed, we write $\Delta_{k,a}:=\Delta_{k,a}(I_{n,j})$, where $j$ is such that $l(j)=l$. By Lemma \ref{lem:ternary-increments}, $\Delta_{k,a}=a^{n-l-k}(1-2a)^{l-k}P_k(n,l)$. Differentiating with respect to $a$ and simplifying gives
\begin{align*}
\Delta_{k+1,a}=\frac{\partial\Delta_{k,a}}{\partial a} &= a^{n-l-k-1}(1-2a)^{l-k-1}\bigg(a(1-2a)\frac{\partial P_k}{\partial a}(n,l)\\
& \qquad\qquad\qquad +\{(1-2a)n-l-(1-4a)k\}P_k(n,l)\bigg).
\end{align*}
Applying Lemma \ref{lem:ternary-increments} in reverse yields the result.
\end{proof}

\begin{lemma} \label{lem:polynomial-convergence}
Let $(l_n)$ be any sequence of integers such that $0\leq l_n\leq n$. Then
\[
P_k(n,l_n)\sim \big(2a(1-2a)n\big)^{k/2}q_k\left(\frac{(1-2a)n-l_n}{\sqrt{2a(1-2a)n}}\right) \qquad \mbox{as $n\to\f$},
\]
where, as usual, $a_n\sim b_n$ means that $a_n/b_n\to 1$ as $n\to\f$.
\end{lemma}

\begin{proof}
Put $g(n):=(1-2a)n-l_n$. We will prove the statement by induction using the previous lemma. In fact, we will prove a bit more, namely
\begin{equation} \label{eq:polynomial-derivative-convergence}
\frac{\partial^j P_k}{\partial a^j}(n,l_n)\sim n^{(k+j)/2}\big(2a(1-2a)\big)^{(k-j)/2}(-2)^j q_k^{(j)}\left(\frac{g(n)}{\sqrt{2a(1-2a)n}}\right)
\end{equation}
for $j=0,1,\dots,k$, where $q_k^{(j)}$ denotes the $j$th derivative of $q_k$. Setting $j=0$ then gives the lemma.

We will need the following linear recursion for $q_k^{(j)}$:
\begin{equation} \label{eq:q-derivative-recursion}
q_{k+1}^{(j)}(t)=-q_k^{(j+1)}(t)+tq_k^{(j)}(t)+jq_k^{(j-1)}(t).
\end{equation}
For $j=0$ this is just the definition of $q_{k+1}$, provided we put $q_k^{(-1)}\equiv 0$. For $j\geq 1$ it follows readily by induction.

Observe that \eqref{eq:polynomial-derivative-convergence} holds for $k=1$, as $q_1(t)=t$ and
$P_1(n,l_n)=(1-2a)n-l_n$. 
Suppose \eqref{eq:polynomial-derivative-convergence} holds for $k$. Differentiating \eqref{eq:P-recursion} $j$ times with respect to $a$ gives
\begin{align*}
\frac{\partial^j P_{k+1}}{\partial a^j}(n,l) &=a(1-2a)\frac{\partial^{j+1}P_{k}}{\partial a^{j+1}}(n,l) + \{(1-2a)n-l-(1-4a)(k-j)\}\frac{\partial^j P_k}{\partial a^j}(n,l) \\
& \qquad\qquad - \left(2jn-4\sum_{\nu=k-j+1}^k\nu\right)\frac{\partial^{j-1}P_{k}}{\partial a^{j-1}}(n,l),
\end{align*}
for $j\geq 1$. Substituting $l=l_n$ and discarding constant terms yields
\[
\frac{\partial^j P_{k+1}}{\partial a^j}(n,l_n)\sim a(1-2a)\frac{\partial^{j+1}P_{k}}{\partial a^{j+1}}(n,l_n) + g(n)\frac{\partial^j P_k}{\partial a^j}(n,l_n)
 -2jn \frac{\partial^{j-1}P_{k}}{\partial a^{j-1}}(n,l_n).
\]
Applying the induction hypothesis \eqref{eq:polynomial-derivative-convergence} to each term on the right hand side and using \eqref{eq:q-derivative-recursion}, we see that \eqref{eq:polynomial-derivative-convergence} holds for $k+1$ in place of $k$, as desired.
\end{proof}

\begin{proof}[Proof of Theorem \ref{thm:alternating-infinities}]
We will prove ({\em b}) for the case when $k-i$ is even. Let us assume $k$ and $i$ are both odd; the case when both are even is similar. Without loss of generality we may take $i=1$; the assumption then is that $\tilde{t}_1<\delta_*\leq\delta^*<\tilde{t}_2$. Then for some $\eps>0$ and all sufficiently large $n$ we have
\begin{equation} \label{eq:sandwich}
(\tilde{t}_1+\eps)\sqrt{n}\leq (1-2a)n-l_n\leq (\tilde{t}_2-\eps)\sqrt{n}.
\end{equation}
By the symmetry argument outlined at the end of the proof of Theorem \ref{thm:infinite-derivatives}, it suffices to show that $M_{k,a}$ has an infinite right derivative.

Observe that $q_k(t)>0$ for $t\in(t_1,t_2)$. Let 
\[
g(n):=(1-2a)n-l_n, \qquad \tilde{g}(n):\frac{g(n)}{\sqrt{2a(1-2a)n}}, \qquad \tilde{\eps}:=\frac{\eps}{\sqrt{2a(1-2a)n}}.
\]
By \eqref{eq:sandwich}, $t_1+\tilde{\eps}<\tilde{g}(n)<t_2-\tilde{\eps}$ for all large enough $n$. 
Thus, 
$$q_k(\tilde{g}(n))\geq \min\{q_k(t_1+\tilde{\eps}),q_k(t_1-\tilde{\eps})\}>0$$ 
for all large enough $n$, as $q_k$ is unimodal over $(t_1,t_2)$. Hence, Lemma \ref{lem:polynomial-convergence} yields $P_k(n,l_{n})\to\f$, and in fact,
\begin{equation} \label{eq:P_k-estimate}
P_k(n,l_n)\geq \delta n^{k/2} \qquad\mbox{for all large enough $n$},
\end{equation}
where $\delta>0$ is a constant depending only on $a$ and $\eps$.

Observe that, if $l_n$ is replaced with $l_n':=l_n+c_n$ for some bounded sequence $(c_n)$, then \eqref{eq:sandwich} holds also for $l_n'$ in place of $l_n$, so the same argument shows that $P_k(n,l_n+c_n)\to\f$ as well. Even more is true, namely,
\begin{equation} \label{eq:bounded-sequence-equivalence}
P_k(n,l_n+c_n)\sim P_k(n,l_n).
\end{equation}
This is clear if $g(n)\to\pm\f$. But when $k$ is odd, $(g(n))$ cannot have a bounded subsequence in view of \eqref{eq:sandwich} since $(t_1,t_2)$ does not contain $0$. (When $k$ is even, $(t_1,t_2)$ can contain $0$ and so $(g(n))$ can have a bounded subsequence, but even then \eqref{eq:bounded-sequence-equivalence} still holds because $q_k(0)\neq 0$.)

In fact, the sequence $(c_n)$ may even tend to $\pm\f$, as long as it does so more slowly than $\sqrt{n}$. In that case, we have, instead of \eqref{eq:sandwich},
\[
\left(\tilde{t}_1+\frac{\eps}{2}\right)\sqrt{n}\leq g(n)\pm c_n\leq \left(\tilde{t}_2-\frac{\eps}{2}\right)\sqrt{n}
\]
for all large enough $n$, and we conclude \eqref{eq:bounded-sequence-equivalence} in the same way. We will need this fact later on in the proof.

Now let $h>0$ be given, and take $n\in\NN$ such that $3^{-n+1}\leq h<3^{-n+2}$. Let $j$ and $j'$ be the integers such that $x\in I_{n,j}$ and $x+h\in I_{n,j'}$. Then $3\leq j'-j\leq 9$. For $i\in\NN$, let $l(i)$ denote the number of $1$'s in the ternary representation of $i$. Since $|l(i+3^\nu)-l(i)|=1$ for all $i$ and $\nu=0,1,2,\dots$, it follows that
\begin{equation} \label{eq:l-difference}
|l(i)-l(j)|\leq 4 \qquad\mbox{for $i=j,j+1,\dots,j'$}.
\end{equation}
Define constants
\[
b:=\max\{a,1-2a\}, \qquad r:=\min\left\{\frac{a}{1-2a},\frac{1-2a}{a}\right\}.
\]
By Lemma \ref{lem:ternary-increments}, we may write
\[
\Delta_{k,a}(I_{n,i})=a^{n-l(i)-k}(1-2a)^{l(i)-k}P_k(n,l(i)).
\]
By \eqref{eq:l-difference} and \eqref{eq:bounded-sequence-equivalence}, there is an $n_0\in\NN$ such that
\[
P_k(n,l(i))\geq \frac12 P_k(n,l_n) \qquad \mbox{for all $j\leq i\leq j'$ and all $n>n_0$}.
\]
Again using \eqref{eq:l-difference}, it follows from \eqref{eq:P_k-estimate} that
\begin{align}
\begin{split}
\Delta_{k,a}(I_{n,i})&\geq \frac12 r^4 a^{n-l_n-k}(1-2a)^{l_n-k}P_k(n,l_n), \\
&\geq \frac12 r^4 a^{n-l_n-k}(1-2a)^{l_n-k}\delta n^{k/2}, \qquad i=j,\dots,j'.
\end{split}
\label{eq:Delta-lower-estimate}
\end{align}

We have established that the increments over the ternary intervals of order $n$ between $x$ and $x+h$ are large. Furthermore, $x$ and $x+h$ are separated by at least one full ternary interval of order $n$. Next, we must make sure that the oscillations of $M_{k,a}$ over $I_{n,j}$ and $I_{n,j'}$ are not too much bigger than the increments over these intervals, so that $M_{k,a}(x)$ is not too much bigger than $M_{k,a}((j+1)3^{-n})$ and $M_{k,a}(x+h)$ is not too much smaller than $M_{k,a}(j'3^{-n})$. Unfortunately, we do not have an analogue of Lemma \ref{prop:box-containment} for $a<1/2$, so the analysis is by necessity somewhat more technical.

Let $m:=m_n:=\lfloor n^{1/4}\rfloor$, and note that for any $y\in J:=\bigcup_{i=j}^{j'}I_{n,i}$, we have $|l_{n+m_n}(y)-l_{n+m_n}(x)|\leq m_n$. Hence, by \eqref{eq:bounded-sequence-equivalence} applied to $n+m_n$ in place of $n$, we see that
\[
P_k(n+m_n,l_{n+m_n}(y))>\frac12 P_k(n+m_n,l_{n+m_n}(x))>0 \qquad \forall y\in J.
\]
Thus,
\begin{equation} \label{eq:nonnegative-increments}
\Delta_{k,a}(I_{n+m,\nu})\geq 0 \qquad \mbox{for $\nu=3^m j, \dots, 3^m(j'+1)-1$},
\end{equation}
i.e., the increments over the ternary intervals of order $n+m$ in $J$ are nonnegative.
Now we use Lemma \ref{lem:oscillation-bounds} to estimate, for $0\leq\nu<3^m$,
\begin{align*}
\osc(M_{k,a};I_{n+m,3^m j+\nu})&\leq a^{n+m-l(3^m j+\nu)}(1-2a)^{l(3^m j+\nu)}C_k(n+m)^k\\
&\leq a^{n-l_n}(1-2a)^{l_n}b^m C_k(n+m)^k,
\end{align*}
where we used that $l(3^m j+\nu)\geq l(j)=l_n$. Similarly,
\begin{align*}
\osc(M_{k,a};I_{n+m,3^m j'+\nu})&\leq a^{n+m-l(3^m j'+\nu)}(1-2a)^{l(3^m j'+\nu)}C_k(n+m)^k\\
&\leq a^{n-l(j')}(1-2a)^{l(j')}b^m C_k(n+m)^k\\
&\leq r^{-4}a^{n-l_n}(1-2a)^{l_n}b^m C_k(n+m)^k,
\end{align*}
where the last inequality used \eqref{eq:l-difference}. 
Since $m\leq n$ and $x$ lies in one of the intervals $I_{n+m,3^m j+\nu}$ for $0\leq\nu<3^m$, it follows by \eqref{eq:nonnegative-increments} that
\[
M_{k,a}(x)\leq M_{k,a}\left(\frac{j+1}{3^n}\right)+b^m a^{n-l_n}(1-2a)^{l_n}C_k 2^k n^k,
\]
and likewise, since $x+h$ lies in one of the intervals $I_{n+m,3^m j'+\nu}$ for $0\leq\nu<3^m$,
\[
M_{k,a}(x+h)\geq M_{k,a}\left(\frac{j'}{3^n}\right)-r^{-4}b^m a^{n-l_n}(1-2a)^{l_n}C_k 2^k n^k.
\]
Hence,
\begin{align*}
M_{k,a}(x+h)-M_{k,a}(x)&\geq M_{k,a}\left(\frac{j'}{3^n}\right)-M_{k,a}\left(\frac{j+1}{3^n}\right)\\
& \qquad -(1+r^{-4})b^m a^{n-l_n}(1-2a)^{l_n}C_k 2^k n^k\\
&\geq \Delta_{k,a}(I_{n,j+1})-(1+r^{-4})b^m a^{n-l_n}(1-2a)^{l_n}C_k 2^k n^k,
\end{align*}
where we used that $j'\geq j+3$ and $\Delta_{k,a}(I_{n,i})\geq 0$ for $i=j,\dots,j'$. Finally, applying \eqref{eq:Delta-lower-estimate}, we obtain
\begin{align*}
M_{k,a}(x+h)-M_{k,a}(x)&\geq a^{n-l_n}(1-2a)^{l_n}\left\{1/2 r^4\delta n^{k/2}-(1+r^{-4})b^m C_k 2^k n^k\right\}\\
&\geq a^{n-l_n}(1-2a)^{l_n}n^{k/2} \cdot\frac14 r^4\delta,
\end{align*}
for all sufficiently large $n$, since $b^{m_n}n^{k/2}\to 0$ as $n\to\f$. (Recall $m_n=\lfloor n^{1/4}\rfloor$.) 
Note that \eqref{eq:not-going-to-zero} holds in view of \eqref{eq:where-is-phi-1} and \eqref{eq:where-is-phi-2}, because the hypothesis $\tilde{t}_1<\delta_*\leq\delta^*<\tilde{t}_2$ implies that $l_n(x)/n\to 1-2a$.
Recalling that $h<3^{-n+2}$, we conclude that
\begin{equation} \label{eq:final-inequality}
\frac{M_{k,a}(x+h)-M_{k,a}(x)}{h}\geq 3^{n-2} a^{n-l_n}(1-2a)^{l_n}n^{k/2} \cdot\frac14 r^4\delta^k\to\f
\end{equation}
as $n\to\f$. This proves ({\em b}). For ({\em a}) and ({\em c}), observe that the hypotheses do not imply that $l_n(x)/n\to 1-2a$, so we additionally assume \eqref{eq:not-going-to-zero} to ensure \eqref{eq:final-inequality}.

It remains to prove ({\em d}). Suppose $\delta_*<\tilde{t}_i<\delta^*$ for some $i\in\{1,2,\dots,k\}$. Then, as can be seen from the above argument, $P_k(n,l_n)$ oscillates around $0$ infinitely many times, and hence, so does $\Delta_{k,a}(I_n(x))$. Therefore, $M_{k,a}$ cannot have an infinite derivative at $x$.
\end{proof}

\begin{remark}
{\rm
For $a=1/3$, Dalaklis et al.~\cite{Dalaklis} showed that $M_{1,1/3}'(x)=\f$ (resp. $-\f$) if and only if $n-3l_n(x)\to\f$ (resp. $-\f$). This condition is weaker than the one in Theorem \ref{thm:alternating-infinities}. Unfortunately, we have not been able to either prove or disprove a similar result for $k=1$ and $a\neq 1/3$. For $k\geq 2$, however, it is clear from Theorem \ref{thm:alternating-infinities} that if one knows only that $n-3l_n(x)\to\f$, then $M_{k,1/3}'(x)$ can be either $+\f$ or $-\f$ or $M_{k,1/3}$ may not have an infinite derivative at all; finer information, about the asymptotics of $(n-3l_n)/\sqrt{n}$, is needed.
}
\end{remark}

\subsection{The dimension of $D_{k,a}^{\pm\f}$}

We can now prove Theorem \ref{thm:D-inf-dimensions}. Recall that $d(a)=h(\phi(a))$ and $\tilde{d}(a)=h(1-2a)$. It follows from \eqref{eq:where-is-phi-1}, \eqref{eq:where-is-phi-2}, and the graph of $h$ that $\tilde{d}(a)\leq d(a)$, with equality if and only if $a=1/3$.


\begin{proof}[Proof of Theorem \ref{thm:D-inf-dimensions}]
We sketch the main idea of the proof. Suppose $0<a<1/3$. By Theorem \ref{thm:infinite-derivative-detail} ({\em a}), $D_{k,a}^{+\f}$ is contained in the set $\{x:\lambda_*(x)\geq\phi(a)\}$, and since $\phi(a)>1/3$, this set has dimension $h(\phi(a))=d(a)$ by \cite[Lemma 4.2]{Allaart}, which generalizes a well-known theorem of Eggleston \cite{Eggleston}. On the other hand, for each sufficiently small $\eps>0$, $D_{k,a}^{+\f}$ contains the set $\{x:l_n(x)/n\to\phi(a)+\eps\}$, which has dimension $h(\phi(a)+\eps)$. So, letting $\eps\to 0$ we find that $\dim_H D_{k,a}^{+\f}\geq h(\phi(a))$. If in addition $k$ is odd, then $D_{k,a}^{-\f}$ can similarly be compared to the sets $\{x:\lambda_*(x)\geq 1-2a\}$ and $\{x:l_n(x)/n\to 1-2a+\eps\}$, so in the same way it follows that $\dim_H D_{k,a}^{-\f}=h(1-2a)$. (We note that the sets
\[
\{x:\phi(a)<\lambda_*(x)\leq \lambda^*(x)<1-2a\} \qquad \mbox{and} \qquad \{x: \lambda_*(x)>\phi(a)\}
\]
have the same Hausdorff dimension. The reason is, that the dimension is already determined by the set of points for which the frequency of $1$'s is ``near" $\phi(a)$, so the extra condition $\lambda^*(x)<1-2a$ is irrelevant.)

The proofs of (b) and (c) are similar.
\end{proof}

We do not know the dimension of $D_{k,a}^{-\f}$ when $k$ is even, but we know from Theorem \ref{thm:alternating-infinities} that this set is not empty. For $k=2$ we sketch a proof here that $D_{2,a}^{-\f}$ actually has fairly large dimension. We suspect this to be the case whenever $k \equiv 2 \pmod 4$, whereas for $k\equiv 0 \pmod 4$ the set $D_{k,a}^{-\f}$ seems to be much smaller.

An upper bound for $\dim_H D_{2,a}^{-\f}$ is easy to come by: It follows from Example \ref{ex:infinite-derivatives-M_2} that
\[
D_{2,a}^{-\f}\subset\left\{x:\limsup_{n\to\f} \frac{|(1-2a)n-l_n(x)|}{\sqrt{n}}\leq c_0\right\}\subset \left\{x: \lim_{n\to\f} \frac{l_n(x)}{n}=1-2a\right\},
\]
and so $\dim_H D_{2,a}^{-\f}\leq h(1-2a)=\tilde{d}(a)$. To obtain a lower bound, observe that by Example \ref{ex:infinite-derivatives-M_2}, $D_{2,a}^{-\f}$ contains for each $c<c_0$ the set
\[
E_c:=\{x: |l_n-(1-2a)n|\leq c\sqrt{n}\ \mbox{for all $n$}\}.
\]
Let
\[
E:=\bigcup_{c<c_0} E_c=\bigcup_{c<c_0, c\in\QQ} E_c.
\]
We will now construct a lower estimate for the dimension of $E$. When $a=1/3$, an obvious estimate can be obtained by considering the set of points $x$ whose ternary expansion has a $1$ precisely in every third place, and $0$'s or $2$'s arbitrarily everywhere else. This set has Hausdorff dimension $\log 4/\log 27$, and on this set, $|l_n-(1-2a)n|$ is in fact  bounded. This can be extended to arbitrary rational $a$. To get a better estimate, and also cover the case of irrational $a$, we need to introduce more randomness, though not ``too much". More precisely, we look for a probability measure $\mu$ on $[0,1]$ with large Hausdorff dimension that gives full measure to $E$. A natural candidate is the measure $\mu$ which makes the ternary digits $\{x_i\}$ of a point $x\in[0,1]$ independent, identically distributed (i.i.d.) random variables taking the values $0, 1$ and $2$ with probabilities $a, 1-2a$ and $a$, respectively. By the Strong Law of Large Numbers, $l_n(x)/n\to 1-2a$ with $\mu$-probability 1. Put $\xi_i=1$ if $x_i=1$, and $\xi_i=0$ otherwise; then $\xi_1,\xi_2,\dots$ are i.i.d. random variables under $\mu$ with mean $1-2a$ and variance $\sigma^2:=2a(1-2a)$, and $l_n=\sum_{i=1}^n \xi_i$. The Law of the Iterated Logarithm (LIL - see \cite[Theorem 9.5]{Billingsley}) implies that
\[
\limsup_{n\to\f} \frac{l_n-(1-2a)n}{\sigma\sqrt{2n\log\log n}}=1, \qquad \mu-\mbox{a.s.},
\]
and so, with probability one, $l_n-(1-2a)n>2\sqrt{a(1-2a)n}>c_0\sqrt{n}$ infinitely often. Hence $\mu(E)=0$, which is not what we wanted. The problem is, essentially, that the i.i.d. digit model contains ``too much" randomness. To overcome this, we replace $\mu$ with a Markov measure $\mu_p$, where $p$ is a parameter whose value we will choose later. Unfortunately, in what follows we need to impose a restriction on $a$ and assume $1/8\leq a\leq 3/8$.

Consider the Markovian matrix
\[
\mathcal{P}=[p_{ij}]_{i,j=0}^2=\begin{bmatrix} \displaystyle \frac{1-r}{2} & \displaystyle r & \displaystyle \frac{1-r}{2} \\ \displaystyle \frac{1-p}{2} & \displaystyle  p & \displaystyle \frac{1-p}{2} \\ \displaystyle \frac{1-r}{2} & \displaystyle r & \displaystyle \frac{1-r}{2} \end{bmatrix},
\]
where $p,r\in[0,1]$. We consider $p$ an arbitrary parameter and choose $r$ so that the vector $\mathbf{a}:=[a_0\ \ a_1\ \ a_2]:=[a \ \ 1-2a \ \ a]$ is an invariant measure; that is, $\mathbf{a}\mathcal{P}=\mathbf{a}$. This gives
\begin{equation} \label{eq:r-probability}
r=\frac{(1-2a)(1-p)}{2a},
\end{equation}
assuming this quantity lies between $0$ and $1$. We now let $\mu_p$ be the Markov measure such that $\mu_p(x_1=i)=a_i$ and $\mu_p(x_{n+1}=j|x_n=i)=p_{ij}$, for $i,j=0,1,2$ and $n\in\NN$. By standard theory, $\mu_p$ has entropy
\begin{align*}
H(a,p):&=-\sum_{i=0}^2 a_i\sum_{j=0}^2 p_{ij}\log p_{ij}\\
&=-2a(1-r)\log\left(\frac{1-r}{2}\right)-2ar\log r\\
& \qquad -(1-2a)(1-p)\log\left(\frac{1-p}{2}\right)-(1-2a)p\log p,
\end{align*}
and so $\mu_p$ has Hausdorff dimension
\[
\dim_H \mu_p=\frac{H(a,p)}{\log 3}.
\]
(We recall that the Hausdorff dimension of a probability measure $\mu$ is defined by
\[
\dim_H \mu:=\inf\{\dim_H F: F\ \mbox{is a Borel set with}\ \mu(F)=1\}.)
\]
We wish to choose $p$ so that $\mu_p$ gives full measure to $E$. Note that under $\mu_p$, the random variables $\{\xi_n\}$ are no longer independent, so the process $\{l_n\}$ is not a classical random walk, but rather, a {\em correlated random walk}. We derive a LIL for $\{l_n\}$ by modifying the proof of \cite[Proposition 6.3]{Allaart-2009}, as follows. Put
\begin{gather*}
\tau_1:=\inf\{n\in\NN: x_n=1\},\\
\tau_{2m}:=\inf\{n>\tau_{2m-1}: x_n\neq 1\}, \qquad m\in\NN,\\
\tau_{2m+1}:=\inf\{n>\tau_{2m}: x_n=1\}, \qquad m\in\NN.
\end{gather*}
Let $T_m:=\tau_{m}-\tau_{m-1}$ and $U_m:=\tau_{2m}-\tau_{2m-2}=T_{2m-1}+T_{2m}$ for $m\in\NN$. For $m\geq 2$, $T_m$ has a geometric distribution with parameter $1-p$ if $m$ is even, and with parameter $r$ if $m$ is odd. Furthermore, $T_1, T_2,\dots$ are independent. Thus, for $m\geq 2$,
\begin{gather*}
\EE(T_{2m})=\frac{1}{1-p}, \qquad \var(T_{2m})=\frac{p}{(1-p)^2},\\
\EE(T_{2m-1})=\frac{1}{r}, \qquad \var(T_{2m-1})=\frac{1-r}{r^2},
\end{gather*}
where all expectations and variances are with respect to $\mu_p$. From this, 
\begin{equation} \label{eq:mean-of-U}
\EE(U_m)=\frac{1}{1-p}+\frac{1}{r}=\frac{1}{(1-2a)(1-p)}.
\end{equation}
Now observe that $l_n-(2a-1)n$ attains its local maxima at the times $n=\tau_{2m}-1$, and $l_{\tau_{2m}}-l_{\tau_{2m}-1}$ is bounded. Thus, it suffices to develop a LIL for the process $S_m:=l_{\tau_{2m}}$. We can express $S_m-\EE(S_m)$ as a sum of i.i.d. random variables by $S_m=Z_1+\dots+Z_m$, where
\begin{align*}
Z_m:&=l_{\tau_{2m}}-l_{\tau_{2m-2}}-(1-2a)(\tau_{2m}-\tau_{2m-2})\\
&=\tau_{2m}-\tau_{2m-1}-(1-2a)(\tau_{2m}-\tau_{2m-2})\\
&=2a(\tau_{2m}-\tau_{2m-1})-(1-2a)(\tau_{2m-1}-\tau_{2m-2})\\
&=2a T_{2m}-(1-2a)T_{2m-1}.
\end{align*}
This shows that $\EE(Z_m)=0$ and, after some calculation,
\[
\var(Z_m)=4a^2\var(T_{2m})+(1-2a)^2\var(T_{2m-1})=\frac{2a(4a-1+p)}{(1-p)^2}=:s^2,
\]
where we used the independence of $T_{2m-1}$ and $T_{2m}$ and substituted the expression \eqref{eq:r-probability} for $r$. It follows from the classical LIL that
\[
\limsup_{m\to\f}\frac{l_{\tau_{2m}}-(1-2a)\tau_{2m}}{\sqrt{2s^2 m\log\log m}}=1 \qquad \mbox{a.s.}
\]
Let $v(m):=\sqrt{2m\log\log m}$. Since $\tau_{2m}=U_1+\dots+U_m$ and $U_2,\dots,U_m$ are i.i.d. with mean given by \eqref{eq:mean-of-U}, the SLLN implies
\[
\frac{v(\tau_{2m})}{v(m)}\to \frac{1}{\sqrt{(1-2a)(1-p)}} \qquad\mbox{a.s.}
\]
Hence,
\begin{align*}
\limsup_{n\to\f}\frac{l_n-(1-2a)n}{v(n)} &= \limsup_{m\to\f}\frac{l_{\tau_{2m}}-(1-2a)\tau_{2m}}{v(\tau_{2m})}\\
&=\limsup_{m\to\f}\frac{l_{\tau_{2m}}-(1-2a)\tau_{2m}}{sv(m)}\cdot \frac{sv(m)}{v(\tau_{2m})}\\
&=\sqrt{\frac{2a(1-2a)(4a-1+p)}{1-p}} \qquad\mbox{a.s.}
\end{align*}
The corresponding $\liminf$ can be dealt with in a similar way.
Therefore, in order that $\mu_p(E)=1$, we need to choose $p$ so that
\[
\sqrt{\frac{4a(1-2a)(4a-1+p)}{1-p}}<c_0=\sqrt{2a(1-2a)},
\]
i.e. 
\[
p<p_c(a):=1-\frac{8a}{3}.
\]
This is possible if and only if $a<3/8$. Furthermore, substituting $p=p_c(a)$ into \eqref{eq:r-probability}, the requirement that $r\leq 1$ is equivalent to $a\geq 1/8$. Under these restrictions on $a$, we obtain, by continuity,
\[
\dim_H D_{2,a}^{-\f}\geq \dim_H E \geq \dim_H \mu_{p_c(a)}=\frac{H(a,p_c(a))}{\log 3}=:\tilde{H}(a).
\]
For example, if $a=1/3$ we get $\dim_H D_{2,a}^{-\f}\geq 0.9433$. We plot the graph of $\tilde{H}(a)$, together with the upper bound $h(1-2a)$, in Figure \ref{fig:dim-bounds}. We do not know how to obtain good lower bounds for $a<1/8$ or $3/8<a<1/2$.

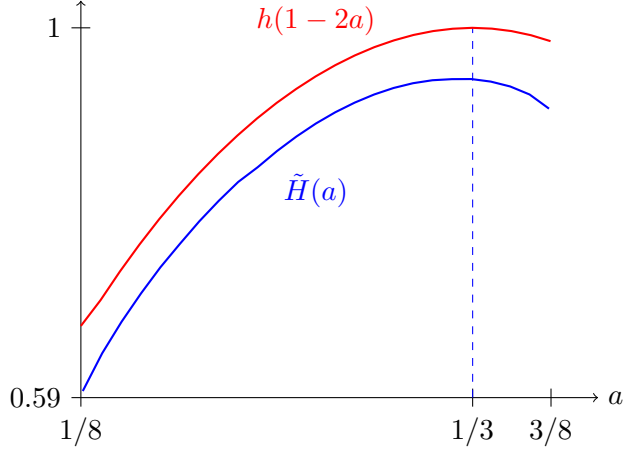
\begin{figure}[h] 
\begin{center}
\begin{tikzpicture}[xscale=25,yscale=12]
\draw [->] (0.12,0.59) node[anchor=east] {$0.59$}  -- (0.4,0.59) node[anchor=west] {$a$};
\draw [->] (1/8,0.58) node[anchor=north] {$1/8$} -- (1/8,1.03);
\draw(3/8,0.58) node[anchor=north]{$3/8$}--(3/8,0.6);
\draw(1/3,0.58) node[anchor=north]{$1/3$}--(1/3,0.6);
\draw(0.12,1) node[anchor=east] {$1$}--(0.13,1) ;
\draw[dashed,blue](1/3,0.59) --(1/3,1);
\draw[variable=\a,domain=0.125:0.375,red,thick] plot({\a},{(-2*\a*ln(\a)-(1-2*\a)*ln(1-2*\a))/ln(3)})(0.25,0.98)node[above,scale=1pt]{$h(1-2a)$};
\draw[variable=\a,domain=0.126:0.374,blue,thick] plot({\a},{(-(2/3)*\a*(8*\a-1)*ln((8*\a-1)/6)-(16/3)*\a*(1-2*\a)*ln(4/3)-(8/3)*\a*(1-2*\a)*ln(\a*(1-2*\a))-(1-2*\a)*(1-8*\a/3)*ln(1-8*\a/3))/ln(3)})(0.25,0.85)node[below,scale=1pt] {$\tilde{H}(a)$};
\end{tikzpicture}
\end{center}
\caption{The upper and lower bounds for $\dim_H D_{2,a}^{-\f}$, for $\frac18\leq a\leq\frac38$.}
\label{fig:dim-bounds}
\end{figure}

The situation is different for $k=4$. By Example \ref{ex:infinite-derivatives-M_4}, a necessary condition for $M_{4,a}'(x)=-\f$ is that either
\[
c_1\leq\delta_*\leq\delta^*\leq c_2 \qquad\mbox{or}\qquad -c_2\leq\delta_*\leq\delta^*\leq -c_1.
\]
Although the set (call it $E$) of points $x$ satisfying this condition is clearly uncountable, its dimension should be very small, if not zero. For instance, if $c_1\leq\delta_*\leq\delta^*\leq c_2$, we must have, for every $\eps>0$,
\[
(1-2a)n-(c_2+\eps)\sqrt{n}<l_n<(1-2a)n-(c_1-\eps)\sqrt{n} \qquad\mbox{for all $n$ sufficiently large}.
\] 
This is a very tight restriction, hence $E$ is a very small subset of the set $\{x: l_n(x)/n\to 1-2a\}$. In fact, we do not know whether $\dim_H D_{4,a}^{-\f}>0$.

\section{Infinite derivatives: $a>1/2$} \label{sec:infinite-derivative-high}

The situation regarding infinite derivatives is very different when $a>1/2$. The reason is, that if the ternary expansion of $x$ has infinitely many $1$'s, then the increments $\Delta_{k,a}(I_n(x))$ change sign infinitely many times in view of Lemma \ref{lem:ternary-increments}, because $1-2a<0$. Therefore, the only points at which there can possibly be an infinite derivative are those whose ternary expansion has only a finite number of $1$'s. This condition is necessary, but not sufficient, as we will see below. Contrary to the case $a<1/2$ studied in the previous section, the polynomial $P_k(n,l)$ behaves much more regularly here: Depending on the parity of $k$, there is an integer $n_0\in\NN$ such that either $P_k(n,l)>0$ for all $n\geq n_0$ and all $l\leq n$, or $P_k(n,l)>0$ for all $n\geq n_0$ and all $l\leq n$.

\begin{lemma} \label{lem:alternating-slopes}
Assume $a>1/2$. There is an integer $n_0\in\NN$ such that for all $n>n_0$, the signs of $\Delta_{k,a}(I_{n,j})$ alternate: that is, 
\[
\sgn(\Delta_{k,a}(I_{n,j+1}))=-\sgn(\Delta_{k,a}(I_{n,j})), \qquad j=0,1,\dots,3^n-2. 
\]
Furthermore, for each $n>n_0$ and $0\leq j<3^n$, we have that
\begin{equation} \label{eq:zigzag}
|\Delta_{k,a}(I_{n+1,3j})|=|\Delta_{k,a}(I_{n+1,3j+2})|>|\Delta_{k,a}(I_{n+1,3j+1})|.
\end{equation}
\end{lemma}

Thus, for all sufficiently large $n$, the $n$th piecewise linear approximation of $M_{k,a}$ follows the same `zig zag' pattern as Okamoto's function itself.

\begin{proof}
Observe that $(1-2a)n-l<0$ for all $n$ and $l$, and its absolute value increases with $l$. Thus, $|(1-2a)n-l|^k\geq n^k(2a-1)^k$. Furthermore, there is a constant $C>0$ such that $R_k(n,l)\leq Cn^{k-1}$ for all $n$ and all $l\leq n$, where $R_k(n,l)$ is the polynomial from Lemma \ref{lem:ternary-increments}. Put
\[
n_0:=\frac{C}{(2a-1)^k}.
\]
Then $|(1-2a)n-l|^k>R_k(n,l)$ for all $n>n_0$ and all $l\leq n$, so for $n>n_0$ the sign of $P_k(n,l)$ is the same as the sign of $\{(1-2a)n-l\}^k$; that is, $P_k(n,l)>0$ if $k$ is even, and $P_k(n,l)<0$ if $k$ is odd. Therefore, the sign of $\Delta_{k,a}(I_{n,j})$ is equal to $(-1)^k$ times the sign of $(1-2a)^{l(j)}$. Since $|l(j+1)-l(j)|=1$ for all $j$, this proves the alternating signs statement. In addition, we see that $|P_k(n,l)|\to\infty$ as $n\to\infty$, uniformly in $l$.

Next, since $a>1/2$ implies $a>2a-1=|1-2a|$, we can choose $\eps>0$ so that
\[
\frac{a}{|1-2a|}>1+\eps,
\]
and then we can find $n_1\in\NN$ such that
\[
(1+\eps)^{-1}<\frac{P_k(n,l+1)}{P_k(n,l)}<1+\eps \qquad \forall n>n_1,\ \forall l\leq n.
\]
Together with Lemma \ref{lem:ternary-increments}, this establishes \eqref{eq:zigzag}, because
\[
\frac{\Delta_{k,a}(I_{n+1,3j+1})}{\Delta_{k,a}(I_{n+1,3j})}=\frac{|1-2a|}{a}\cdot\frac{P_k(n+1,l(j)+1)}{P_k(n+1,l(j))}.
\]
\end{proof}

We see from the above proof that for all large enough $n$, $P_k(n,l)$ has the same sign as $(1-2a)^{-k}$. Hence, $\Delta_{k,a}(I_{n,j})$ has the same sign as $(1-2a)^{l(j)}$ for all large enough $n$.

For $d,\xi\in\{0,1,2\}$, define
\[
\delta_d(\xi):=\begin{cases}
1 & \mbox{if $\xi=d$},\\
0 & \mbox{otherwise}.
\end{cases}
\]

\begin{proposition} \label{prop:limsup-conditions}
Assume $a>1/2$. Let $x=(0.x_1x_2\dots)_3$, and assume that $l:=\sup_{n\in\NN} l_n(x)<\infty$. Put
\[
L_a^+(x):=\limsup_{n\to\f} \sum_{j=1}^\f a^{j}\delta_2(x_{n+j}).
\]
In order that $M_{k,a}$ have an infinite right derivative at $x$ it is sufficient that $L_a^+(x)<1$, and it is necessary that $L_a^+(x)\leq 1$. Furthermore, if $M_{k,a}$ has an infinite right derivative at $x$, then $(M_{k,a})_+'(x)=(-1)^{l}\cdot\f$.
\end{proposition}

\begin{proof}
Let $N_0$ be the position of the last 1 in $(x_i)$, and put $E:=\{n>N_0: x_n=0\}$. For each $n\in E$, put $z_n:=0.x_1\dots x_{n-1}20^\f$. Then $z_n>x$ and $3^{-n}<z_n-x<2\cdot 3^{-n}$. As in the proof of Theorem \ref{thm:infinite-derivatives-half}, in order that $M_{k,a}$ has an infinite right derivative at $x$, it is necessary and sufficient that
\begin{equation*} 
\frac{M_{k,a}(z_n)-M_{k,a}(x)}{z_n-x}\to\pm\f.
\end{equation*}
This is a consequence of Lemma \ref{lem:alternating-slopes}; we refer to the argument in the above-mentioned proof.

Take $n\in E$. Then Lemma \ref{lem:Okamoto-series} implies
\[
F_a(z_n)-F_a(x)=a^{n-l-1}(1-a)(1-2a)^l-\sum_{j=n+1}^\f a^{j-l-1}(1-a)(1-2a)^l\delta_2(x_j).
\]
Differentiating $k$ times with respect to $a$ gives
\begin{align*}
M_{k,a}&(z_n)-M_{k,a}(x)=\sum_{i=0}^k\binom{k}{i}\frac{\partial^i}{\partial a^i}(a^{n-l-1}-a^{n-l})\cdot\frac{\partial^{k-i}}{\partial a^{k-i}}(1-2a)^l\\
&\qquad\qquad\qquad\qquad -\sum_{j=n+1}^\f \sum_{i=0}^k\binom{k}{i}\frac{\partial^i}{\partial a^i}(a^{j-l-1}-a^{j-l})\cdot\frac{\partial^{k-i}}{\partial a^{k-i}}(1-2a)^l\delta_2(x_j)\\
&=\sum_{i=0}^k\binom{k}{i}\left(\pi_{n-l-1,i}a^{n-l-i-1}-\pi_{n-l,i}a^{n-l-i}\right)\pi_{l,k-i}(-2)^{k-i}(1-2a)^{l-k+i}\\
& \quad -\sum_{j=n+1}^\f \sum_{i=0}^k\binom{k}{i}\left(\pi_{j-l-1,i}a^{j-l-i-1}-\pi_{j-l,i}a^{j-l-i}\right)\pi_{l,k-i}(-2)^{k-i}(1-2a)^{l-k+i}\delta_2(x_j),
\end{align*}
where, for nonnegative integers $n$ and $m$ with $n\geq m$, we denote
\[
\pi_{n,m}:=\frac{n!}{m!}=n(n-1)\cdots(n-m+1).
\]
After reindexing, the summation over $j$ can be written as
\begin{align*}
\sum_{m=1}^\f \sum_{i=0}^k\binom{k}{i}&\left(\pi_{n+m-l-1,i}a^{n+m-l-i-1}-\pi_{n+m-l,i}a^{n+m-l-i}\right)  \\
& \qquad \times \pi_{l,k-i}(-2)^{k-i}(1-2a)^{l-k+i}\delta_2(x_{n+m}).
\end{align*}
Thus, $M_{k,a}(z_n)-M_{k,a}(x)$ takes the form
\[
M_{k,a}(z_n)-M_{k,a}(x)=a^{n-l-k-1}(1-2a)^{l}\left(c_k^{(n)} n^k+c_{k-1}^{(n)}n^{k-1}+\dots+c_1^{(n)} n+c_0^{(n)}\right),
\]
for certain coefficients $c_k^{(n)},\dots,c_0^{(n)}$ that depend on $n$ only through $\{\delta_2(x_{n+m}): m\in\NN\}$ and are therefore bounded. In particular,
\[
c_k^{(n)}=(1-a)\left(1-\sum_{m=1}^\f a^m\delta_2(x_{n+m})\right),
\]
as only the terms with $i=k$ include a term $n^k$. If $L_a^+(x)=\limsup_{n\to\f}a^m\delta_2(x_{n+m})<1$, then $c_k^{(n)}$ is bounded away from $0$, and so the sign of $M_{k,a}(z_n)-M_{k,a}(x)$ is eventually determined by the dominant term $c_k^{(n)} n^k$, so there is a constant $\eps>0$ such that 
\[
|M_{k,a}(z_n)-M_{k,a}(x)|\geq a^{n-l-k-1}(1-a)|1-2a|^l\eps n^k 
\]
for all sufficiently large $n\in E$. Since furthermore $z_n-x<2\cdot 3^{-n}$ and $(3a)^n\to\f$, it follows that $(M_{k,a})_+'(x)=(-1)^l\cdot\f$.

On the other hand, suppose $L_a^+(x)>1$. Then there is a subsequence $(n_i)$ and a constant $\eps>0$ such that $c_k^{(n_i)}<-(1-a)\eps$ for all $i$, which implies that $M_{k,a}(z_{n_i})-M_{k,a}(x)$ has the opposite sign to $(1-2a)^l$ for all sufficiently large $i$. If $M_{k,a}$ were to have an infinite derivative at $x$, say $+\f$, we would need to have the secant slope $3^n\Delta_{k,a}(I_n(x))=3^n a^{n-l-k}(1-2a)^{l-k}P_k(n,l)$ tending to $+\f$ as well. But the sign of $M_{k,a}(z_{n_i})-M_{k,a}(x)$ is the opposite of the sign of $\Delta_{k,a}(I_{n_i}(x))$. Hence, $M_{k,a}$ does not have an infinite derivative at $x$.
\end{proof}



\begin{proof}[Proof of Theorem \ref{thm:infinite-derivatives}]
We prove the statement for $\dim_H D_{k,a}^{+\f}$; the proof for $\dim_H D_{k,a}^{-\f}$ is essentially the same.
Recall that $M_{k,a}$ can only have an infinite derivative at $x$ if $l:=\sup_n l_n(x)<\f$. When this is the case, Proposition \ref{prop:limsup-conditions} gives both a necessary and a sufficient condition for $M_{k,a}$ to have an infinite right derivative at $x$. Analogously, define
\[
L_a^-(x):=\limsup_{n\to\f} \sum_{j=1}^\f a^{j}\delta_0(x_{n+j}).
\]
Then, in order that $M_{k,a}$ have an infinite {\em left} derivative at $x$, sufficient is that $L_a^-(x)<1$, and necessary is that $L_a^-(x)\leq 1$. This follows from Proposition \ref{prop:limsup-conditions} and the symmetry $M_{k,a}(1-x)=-M_{k,a}(x)$, or can be proved from scratch in the same way as Proposition \ref{prop:limsup-conditions}. Define now the sets
\[
E_a:=\left\{x\in(0,1): \sup_n l_n(x)\in 2\NN, \ L_a^+(x)<1\ \mbox{and}\ L_a^-(x)<1\right\}, \qquad \frac12<a<1.
\]
Since $L_a^+(x)$ and $L_a^-(x)$ are clearly decreasing in $a$, we have
\[
E_a\subset D_{k,a}^{+\infty} \subset \bigcap_{1/2<a'<a} E_{a'}.
\]
Hence $\dim_H E_a\leq \dim_H D_{k,a}^{+\infty}\leq \dim_H E_{a'}$ for all $1/2<a'<a$. Since the function $\beta\mapsto \dim_H \mathcal{U}_\beta$ is continuous (see \cite{KKL,Allaart-Kong}), it therefore suffices to show that 
\begin{equation} \label{eq:bi-Holder-equivalence}
\dim_H E_a=\frac{-\log a}{\log 3}\dim_H \mathcal{U}_{1/a} \qquad\mbox{for all $\frac12<a<1$}.
\end{equation}
The proof of this fact is given in \cite[Section 5]{Allaart}. (The idea is to apply a transformation $\phi: C\to [0,a/(1-a)]$, where $C$ is the ternary Cantor set, by
\[
\phi(x)=\sum_{n=1}^\f a^n \cdot \frac{x_n}{2},
\]
where $x_1 x_2\dots$ is the ternary expansion of $x$ which does not contain the digit 1. Then the condition that $L_a^+(x)<1$ and $L_a^-(x)<1$ is very close to the condition that $\phi(x)\in \mathcal{U}_{1/a}$, so $\phi(C\cap E_a)$ is ``close" to $\mathcal{U}_{1/a}$ -- close enough to have the same Hausdorff dimension. Furthermore, it can be shown that the map $\phi$, restricted to $C\cap E_a$, is bi-H\"older continuous with exponent $-\log a/\log 3$. Making these statements precise yields \eqref{eq:bi-Holder-equivalence}, since $E_a$ is the countable union of affine images of $C\cap E_a$.)
\end{proof}


\begin{thebibliography}{12}

\bibitem{Allaart-2009}
{\sc P. C. Allaart}, On a flexible class of continuous functions with uniform local structure. {\em J. Math. Soc. Japan} {\bf 61} (2009), no. 1, 237--262.

\bibitem{Allaart}
{\sc P. C. Allaart}, The infinite derivatives of Okamoto's self-affine functions: an application of $\beta$-expansions. {\em J. Fractal Geom.} {\bf 3} (2016), no. 1, 1--31.

\bibitem{Allaart-2017a}
{\sc P.~C. Allaart}, Differentiability of a two-parameter family of self-affine functions. {\em J. Math. Anal. Appl.} {\bf 450} (2017), no. 2, 954--968.

\bibitem{Allaart-2017b}
{\sc P.~C. Allaart}, On univoque and strongly univoque sets. {\em Adv. Math.} {\bf 308} (2017), 575--598.

\bibitem{Allaart-Kawamura}
{\sc P.~C. Allaart} and {\sc K. Kawamura}, The improper infinite derivatives of Takagi's nowhere-differentiable function, {\em J. Math. Anal. Appl.} {\bf 372} (2010), no.~2, 656--665. 

\bibitem{Allaart-Kong}
{\sc P.~C. Allaart} and {\sc D.~Kong}, On the continuity of the Hausdorff dimension of the univoque set. {\em Adv. Math.} {\bf 354} (2019), 106729, 24 pp.

\bibitem{Barany-Prokaj}
{\sc B. B\'ar\'any} and {\sc R. D. Prokaj}, On the dimension theory of Okamoto's function. Preprint, https://arxiv.org/abs/2501.10584 (2025)

\bibitem{Billingsley}
{\sc P.~Billingsley}, {\em Probability and measure}, 3rd Edition, Wiley, New York (1995).

\bibitem{Bourbaki} 
{\sc N. Bourbaki}, {\em Functions of a real variable}, Translated from the 1976 French original by Philip Spain, Springer, Berlin, 2004.

\bibitem{Dalaklis}
{\sc N. Dalaklis, K. Kawamura, T. Mathis} and {\sc M. Paizanis}, The partial derivative of Okamoto's functions with respect to the parameter. {\em Real Anal. Exchange} {\bf 48} (2023), no. 1, 165--178.

\bibitem{Darst}
{\sc R.~B. Darst}, The Hausdorff dimension of the nondifferentiability set of the Cantor function is $[{\rm ln}(2)/{\rm ln}(3)]^2$, {\em Proc. Amer. Math. Soc.} {\bf 119} (1993), no.~1, 105--108. 

\bibitem{Eggleston}
{\sc H. Eggleston}, The fractional dimension of a set defined by decimal properties. {\em Quart. J. Math. Oxford Ser.} {\bf 20} (1949), 31--36.

\bibitem{Eidswick} 
{\sc J. A. Eidswick}, A characterization of the nondifferentiability set of the Cantor function. {\em Proc. Amer. Math. Soc.} {\bf 42} (1974), no. 1, 214--217.

\bibitem{Falconer-2004}
{\sc K.~J. Falconer}, One-sided multifractal analysis and points of non-differentiability of devil's staircases, {\em Math. Proc. Cambridge Philos. Soc.} {\bf 136} (2004), no.~1, 167--174. 

\bibitem{Falconer}
{\sc K.~J.~Falconer}, {\em Fractal geometry: Mathematical foundations and applications,} 3rd Edition, John Wiley \& Sons Ltd., Chichester (2014).

\bibitem{GlenSid} 
{\sc P. Glendinning} and {\sc N. Sidorov}, Unique representations of real numbers in non-integer bases. {\em Math. Res. Lett.} {\bf 8} (2001), no. 4, 535--543.

\bibitem{JKPS}
{\sc T.~M. Jordan, M. Kesseb\"ohmer, M. Pollicott} and {\sc B. Stratmann}, Sets of nondifferentiability for conjugacies between expanding interval maps, {\em Fund. Math.} {\bf 206} (2009), 161--183. 

\bibitem{Katsuura} 
{\sc H. Katsuura}, Continuous nowhere-differentiable functions - an application of contraction mappings. {\em Amer. Math. Monthly} {\bf 98} (1991), no. 5, 411--416.

\bibitem{Kobayashi} 
{\sc K. Kobayashi}, On the critical case of Okamoto's continuous non-differentiable functions. {\em Proc. Japan Acad. Ser. A Math. Sci.} {\bf 85} (2009), no. 8, 101--104.

\bibitem{KKL}
{\sc V.~Komornik, D.~Kong} and {\sc W.~Li}, Hausdorff dimension of univoque sets and devil's staircase. {\em Adv. Math.} {\bf 305} (2017), 165--196.

\bibitem{KomLor} 
{\sc V. Komornik} and {\sc P. Loreti}, Unique developments in non-integer bases. {\em Amer. Math. Monthly} {\bf 105} (1998), 636--639.

\bibitem{McCollum} 
{\sc J. McCollum}, Further notes on a family of continuous, nondifferentiable functions. {\em New York J. Math.} {\bf 17} (2011), 569--577.

\bibitem{Okamoto} 
{\sc H. Okamoto}, A remark on continuous, nowhere differentiable functions. {\em Proc. Japan Acad. Ser. A Math. Sci.} {\bf 81} (2005), no. 3, 47--50.

\bibitem{Perkins} 
{\sc F. W. Perkins}, An elementary example of a continuous non-differentiable function. {\em Amer. Math. Monthly} {\bf 34} (1927), 476--478.

\bibitem{Seuret} 
{\sc S. Seuret}, On multifractality and time subordination for continuous functions. {\em Adv. Math.} {\bf 220} (2009), 936--963.

\bibitem{Troscheit}
{\sc S. Troscheit}, H\"older differentiability of self-conformal devil's staircases, {\em Math. Proc. Cambridge Philos. Soc.} {\bf 156} (2014), no.~2, 295--311. 

\end{thebibliography}
\end{document}